\documentclass[11pt,reqno]{amsart}

\usepackage{todonotes}
\usepackage[utf8]{inputenc}
\usepackage[T1]{fontenc}
\usepackage[english]{babel}
\usepackage{csquotes}

\usepackage[margin=3cm]{geometry}

\usepackage{xcolor}
\usepackage[colorlinks=true, linkcolor=black, citecolor=black]{hyperref} 
\usepackage[nameinlink]{cleveref}

\usepackage[onehalfspacing]{setspace}

\numberwithin{equation}{section} 

\usepackage{paralist}
\usepackage{enumitem}
\setlist[enumerate,1]{label = (\alph*)}

\usepackage{times}
\usepackage{txfonts}
\usepackage{dsfont}
\renewcommand{\mathbb}{\mathds}

\usepackage{amsmath,amssymb}
\usepackage{mathdots}
\usepackage{extpfeil}
\usepackage{extarrows}
\usepackage{oubraces}
\usepackage{array}

\usepackage{tikz}
\usetikzlibrary{cd,arrows}
\tikzcdset{
	arrow style=tikz,
	diagrams={>={Straight Barb[scale=0.9, width=5pt]}} }

\tikzset{epi/.code={\pgfsetarrowsend{Computer Modern Rightarrow[width=5pt, length=3pt] Computer Modern Rightarrow[width=5pt, length=3pt]}}}

\tikzset{epi_mini/.code={\pgfsetarrowsend{Computer Modern Rightarrow[width=5pt, length=3pt, scale=0.85] Computer Modern Rightarrow[width=5pt, length=3pt, scale=0.85]}}}

\usepackage[backend=biber,style=alphabetic,maxbibnames=4,doi=false,isbn=false,url=false]{biblatex}
\bibliography{stabNder.bib}

\theoremstyle{plain}
\newtheorem{thm}{Theorem}[section]
\newtheorem{prp}[thm]{Proposition}
\newtheorem{cor}[thm]{Corollary}
\newtheorem{lem}[thm]{Lemma}

\newtheorem{thmA}{Theorem}

\theoremstyle{definition}
\newtheorem{dfn}[thm]{Definition}
\newtheorem{ntn}[thm]{Notation}
\newtheorem{con}[thm]{Construction}
\newtheorem{asn}[thm]{Assumption}

\theoremstyle{remark}
\newtheorem{rmk}[thm]{Remark}
\newtheorem{exa}[thm]{Example}

\newenvironment{proofof}[1]{%
	\renewcommand{\proofname}{Proof of #1}\proof}{\endproof}

\Crefname{subsection}{Subsection}{Subsections}
\crefname{rmk}{Remark}{Remarks}

\Crefname{thm}{Theorem}{Theorems}
\Crefname{thmA}{Theorem}{Theorems}
\Crefname{prp}{Proposition}{Propositions}
\Crefname{cor}{Corollary}{Corollaries}
\Crefname{lem}{Lemma}{Lemmas}
\Crefname{cnj}{Conjecture}{Conjectures}
\Crefname{dfn}{Definition}{Definitions}
\Crefname{ntn}{Notation}{Notations}
\Crefname{con}{Construction}{Constructions}
\Crefname{asn}{Assumption}{Assumptions}
\Crefname{cnv}{Convention}{Conventions}
\Crefname{rmk}{Remark}{Remarks}
\Crefname{exa}{Example}{Examples}

\newcommand{\ZZ}{\mathbb{Z}}
\newcommand{\NN}{\mathbb{N}}

\newcommand{\A}{\mathcal{A}}

\newcommand{\C}{\mathcal{C}}
\newcommand{\D}{\mathcal{D}}
\newcommand{\E}{\mathcal{E}}
\newcommand{\F}{\mathcal{F}}
\newcommand{\K}{\mathcal{K}}
\newcommand{\cS}{\mathcal{S}}
\newcommand{\T}{\mathcal{T}}
\newcommand{\U}{\mathcal{U}}
\newcommand{\V}{\mathcal{V}}
\newcommand{\Z}{\mathcal{Z}}

\newcommand{\mMor}{\mm\Mor}
\newcommand{\smMor}{\sm \Mor}
\newcommand{\eMor}{\me\Mor}
\newcommand{\seMor}{\se \Mor}

\newcommand{\udots}{\reflectbox{$\ddots$}}
\newcommand{\op}[1]{{#1}^\textup{op}}
\newcommand{\sm}[1]{{#1}^\textup{sm}}
\newcommand{\mm}[1]{{#1}^\textup{m}}
\newcommand{\me}[1]{{#1}^\textup{e}}
\newcommand{\se}[1]{{#1}^\textup{se}}
\newcommand{\Ab}{\textup{Ab}}

\DeclareMathOperator{\id}{id}
\DeclareMathOperator{\mcm}{MCM}
\DeclareMathOperator{\im}{im}
\DeclareMathOperator{\coker}{coker}
\DeclareMathOperator{\Hom}{Hom}
\DeclareMathOperator{\Ext}{Ext}
\DeclareMathOperator{\perf}{perf}
\DeclareMathOperator{\tri}{tri}
\DeclareMathOperator{\APC}{APC}
\DeclareMathOperator{\TAPC}{TAPC}
\DeclareMathOperator{\Mor}{Mor}

\DeclareMathOperator{\Prj}{Proj}
\DeclareMathOperator{\Inj}{Inj}
\DeclareMathOperator{\syz}{syz}

\title[The stabilized $N$-derived category]{The stabilized bounded $\boldsymbol N$-derived category of an exact category}

\author[J.~Frank]{Jonas Frank}
\address{\linebreak
	Jonas Frank\\
	Department of Mathematics, University of Kaiserslautern-Landau\\ 
	67663 Kaiserslautern\\
	Germany
}
\email{\href{jfrank@rptu.de}{jfrank@rptu.de}}

\author[M.~Schulze]{Mathias Schulze}
\address{\linebreak
	Mathias Schulze\\
	Department of Mathematics, University of Kaiserslautern-Landau\\ 
	67663 Kaiserslautern\\
	Germany
}
\email{\href{mschulze@rptu.de}{mschulze@rptu.de}}

\subjclass[2020]{Primary 18G50, 18G80; Secondary 16E65, 18G35, 18G65}

\keywords{$N$-complex, exact category, derived category, singularity category, Cohen-Macaulay.}

\thanks{JF would like to thank Janina C.~Letz for helpful explanations.}

\begin{document}
%%%%%%%%%%%%%%%%%%%%%%%%%%%%%%%%%%%%%%%%%%%%%%%%%%%%%%%%%%%%%%%%%%%%%%%%%%%%%%%%

\begin{abstract}
 Buchweitz related the singularity category of a (strongly) Gorenstein ring and the stable category of maximal Cohen-Macaulay modules by a triangle equivalence. We phrase his result in a relative categorical setting based on $N$-complexes instead of classical 2-complexes. The role of Cohen-Macaulay modules is played by chains of monics in a Frobenius subcategory of an exact category. As a byproduct, we provide foundational results on derived categories of $N$-complexes over exact categories known from the Abelian case or for 2-complexes. 
\end{abstract}

\maketitle

\tableofcontents

\section*{Introduction}

The Auslander-Buchsbaum-Serre theorem states that a Noetherian local ring $S$ is regular if and only if all $S$-modules have a finite projective resolution. In 1980, Eisenbud {\cite{Eis80}} showed that, over a complete intersection $S$, any minimal free resolution with bounded ranks becomes 2-periodic after a finite number of steps, depending only on $\dim(S)$.

\smallskip

In an unpublished preprint from 1986, Buchweitz captured this phenomenon in categorical language.  To extract the asymptotic part of a complex he forms the Verdier quotient $$\underline \D^b(S) = \D^b(S)/\mathrm{Perf}(S)$$ of the bounded derived category by complexes of projective modules, the \emph{perfect} complexes. This \emph{stabilized derived category} of a general (possibly non-commutative) ring $S$ can be seen as a measure for its non-regularity. In 2004, it was rediscovered by Orlov {\cite{Orl04}} in a geometric setting and is today known as the \emph{singularity category}.

\smallskip

Buchweitz works over a (strongly) Gorenstein ring $S$. Beyond the (two-sided) injective dimension of $S$, non-zero syzygies of finite $S$-modules are maximal Cohen-Macaulay modules. Conversely, any such module can be considered a complex, when placed at different positions. This suggests an equivalence between $\underline \D^b(S)$ and the stable category $\underline \mcm(S)$ of maximal Cohen-Macaulay modules. Buchweitz confirmed this fact involving the category $\underline \APC(S)$ of complete resolutions. It serves both as a middleman between the two categories, who translates high to low syzygies and to compute Tate cohomology over $S$.

\smallskip

In 2021, Avramov, Briggs, Iyengar, and Letz {\cite{Buc21}} published an annotated version of Buchweitz's manuscript. In Appendix B, they prove Buchweitz's result under weaker hypotheses, where any finite $S$-module is assumed to have a totally reflexive syzygy of order depending on the module. 

\smallskip

Building on previous work of Murfet and Salarian {\cite{MS11}} and others, Christensen et al.~{\cite{Chr+23,CET20}} generalized Buchweitz's theorem to schemes.

\smallskip

This article provides a purely categorical formulation of Buchweitz's theorem: We replace the category of finite $S$-modules by a general exact idempotent complete category $\E$ and the subcategory $\mcm(S)$ by a Frobenius subcategory $\F$, subject to a list of conditions. As an additional direction of generalization, we pass to a bounded derived category $\D^b_N(\E)$ of \emph{$N$-complexes}, where the $N$th power of the differential is zero for $N \geq 2$. As a consequence, $\F$ becomes the category $\mMor_{N-2}(\F)$ of chains of $N-2$ many monics. Complete $N$-resolutions are then objects of the stable category $\underline{\textup{APC}}_N(\F) = \underline{\textup{TAPC}}_N(\F)$ of (totally) acyclic $N$-complexes over $\F$ with projective objects. More specifically, our main result is

\begin{thmA} \label{thm: buchweitz-intro}
	Let $\E$ be an exact idempotent complete category and $\F$ a Frobenius category, which is a fully exact, replete subcategory of $\E$ with $\Prj(\F)=\Prj(\E)$. Suppose that every object in $\E$ has a syzygy in $\F$. Then there is a commutative diagram of triangle equivalences
	\begin{center}
		\begin{tikzcd}[sep={20mm,between origins}]
			 \underline{\textup{APC}}_N(\F) \ar[rd, "\simeq", "\underline \tau^{\leq 0}"'] \ar[rr, "\underline \Omega^{1}", "\simeq"'] &&  \underline \mMor_{N-2}(\F)  \ar[ld, "\simeq"', "\underline \iota^0"] \\
			&\underline \D^b_N(\E),
		\end{tikzcd}
	\end{center}
	where $\iota^0$ is the embedding, $\tau^{\leq 0}$ the hard truncation and $\Omega^{1}$ the $N$-syzygy, at the respective positions.
\end{thmA}

Our proof that $\underline \tau^{\leq 0}$ is an equivalence is inspired by the work of Orlov {\cite{Orl09}}.

\smallskip

\Cref{thm: buchweitz-intro} is known in special cases: Bahiraei, Hafezi, and Nematbakhsh {\cite{BHN16}} establish an equivalence between $\underline \TAPC_N(S)$ and $\underline \D^b_N(S)$ over a left coherent ring $S$ using triangular matrix rings.
For $N=2$, Christensen et al.~prove the equivalences in \Cref{thm: buchweitz-intro} for a complete, hereditary cotorsion pair $(\U, \V)$ in an Abelian category $\A$, see {\cite[Thm.~3.10]{Chr+23}} and {\cite[Thm.~3.8]{CET20}}. In their setup, $\E = \V$ and $\F$ is the subcategory of right $\U$-Gorenstein objects in $\A$ with $\Prj(\F)=\U \cap \V=\Prj(\E)$, see {\cite[Thm.~2.11]{CET20}} and  {\cite[Prop.~2.7.(a), Thm.~3.6]{Chr+23}}. Brightbill and Miemietz {\cite{BM24}} prove \Cref{thm: buchweitz-intro}, without explicit mention of $\underline \tau^{\leq 0}$, under the assumption that $\E$ is a Gorenstein Abelian category and $\F$ its subcategory of Gorenstein projective objects.\footnote{The final version of their work was published shortly before the completion of this article.}

\smallskip

Already in 1942, $N$-complexes appeared in the work of Mayer {\cite{May42}} generalizing simplicial homology theory. Kapranov {\cite{Kap96}} and Dubois-Violette {\cite{Dub98}} studied their homological properties. They find application in physics and other areas of mathematics, see for instance {\cite{DH99,CSW07,Hen08}} and \cite{Est07,GH10,KQ15}. Cassel and Wambst {\cite{KW98}} studied $N$-resolutions, but only of single objects.

\smallskip 

\Cref{thm: buchweitz-intro} requires foundations of $N$-derived categories over exact categories. Some have been laid by Iyama, Kato and Miyachi {\cite{IKM17}} in the Abelian case, see also {\cite{YD15}} and  {\cite{YW15}}. General results on semiorthogonal decompositions by Jørgensen and Kato {\cite{JK15}} turn out particularly useful in this context. We combine this work with Keller's \cite{Kel90,Kel96} on 2-derived categories of exact categories and rely on a generalization for deflation-exact categories due to Henrard and van Roosmalen {\cite{HR20}}. In particular, we construct $N$-resolutions of bounded above $N$-complexes, see \Cref{subsection: resolutions}. This leads us to extend results of Verdier {\cite[Ch.~III, Thm.~1.2.3]{Ver96}} and  {\cite[Thm.~3.12]{IKM17}} in the following two theorems, formulated using Verdier's notation:

\begin{thmA} \label{thm: diamond}
	For an exact idempotent complete category $\E$ there is a diagram of canonical fully faithful, triangle functors and equivalences:
	
	
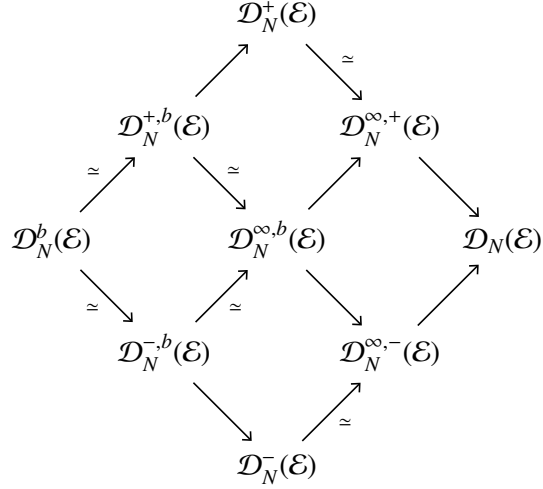
\begin{center}
		\begin{tikzcd}[sep={15mm,between origins}]
			&& \D^+_N(\E) \ar[rd, "\simeq"] && \\
			& \D^{+, b}_N(\E) \ar[ru] \ar[rd, "\simeq"] && \D^{\infty, +}_N(\E) \ar[rd] \\
			\D^b_N(\E) \ar[ru, "\simeq"] && \D^{\infty, b}_N(\E) \ar[ru] \ar[rd] && \D_N(\E) \\
			& \D^{-, b}_N(\E) \ar[ru, "\simeq"'] \ar[rd] \ar[lu, <-, "\simeq"] && \D^{\infty, -}_N(\E) \ar[ru] \\
			&& \D^-_N(\E) \ar[ru, "\simeq"']
		\end{tikzcd}
	\end{center}
\end{thmA}

\begin{thmA} \label{thm: IKM-resolution}
	Let $\E$ be an exact idempotent complete category with enough projectives.
	\begin{enumerate}
		\item The pair $(\K^{-}_N(\Prj(\E)), \K^{-, \varnothing}_N(\E))$ is a semiorthogonal decomposition of $\K^-_N(\E)$, which gives rise to a triangle equivalence $\D^-_N(\E) \simeq \K^-_N(\Prj(\E))$.
		\item The pair $(\K^{-, b_\E}_N(\Prj(\E)), \K^{-, \varnothing }_N(\E))$ is a semiorthogonal decomposition of $\K^{-, b}_N(\E)$, which gives rise to a triangle equivalence $\D^{b}_N(\E) \simeq \K^{-, b_\E}_N(\Prj(\E))$.
	\end{enumerate}
	The obvious dual statements hold as well.
\end{thmA}

\Cref{thm: buchweitz-intro,thm: diamond,thm: IKM-resolution} agree with Theorems \ref{thm: buchweitz}, \ref{thm: D-diamond} and \ref{thm: K-sod} of the main part.

\section{Preliminaries}

\label{section: preliminaries}

Unless stated otherwise, all (sub)categories and functors considered are assumed to be (full) additive.

\smallskip

Our main reference on the topic of \emph{triangulated categories} is Neeman's book {\cite{Nee01}}. However, we require the more general definition of a triangulated category whose suspension functor is only a autoequivalence instead of an automorphism. These two definitions agree up to a triangulated equivalence, see {\cite[\S 2]{KV87}} and {\cite[\S 2]{May01}}.\\
Recall that a \textit{triangle equivalence} is a  triangulated functor which is an equivalence of categories. Its quasi-inverse is automatically a triangulated functor, see {\cite[Prop.~1.4]{BK89}} for a more general statement.

\smallskip

 Unless stated otherwise, the image of a functor always means the full essential image. Given a fully faithful (triangulated) functor $F\colon \C' \to \C$, we tacitly identify $\C'$ up to equivalence with its image $F(\C')$, which is a strictly full (triangulated) subcategory of $\C$.

\smallskip

Bullets in diagrams represent arbitrary objects, and $E_l$ denotes the unit matrix of size $l \in \NN.$

\subsection{Exact categories}

For convenience of the reader, we recollect relevant definitions and results on the topic of \emph{exact categories} in the sense of Quillen {\cite{Qui73}}. Our main reference is Bühler's expository article {\cite{Buh10}}, which relies in part on work of Keller {\cite{Kel90,Kel96}}.

\begin{dfn} \
	\begin{enumerate}
		\item A \textbf{pullback} (dashed) is the limit of a diagram (solid) of the form
		\begin{center}
			\begin{tikzcd}[sep={15mm,between origins}]
				\bullet \ar[r, dashed,  "a'"] \ar[d, dashed,  "b'"]& \bullet \ar[d, "b"] \\
				\bullet \ar[r, "a"] & \bullet.
			\end{tikzcd}
		\end{center}

		\item A \textbf{pushout} (dashed) is the colimit of a diagram (solid) of the form
		\begin{center}
			\begin{tikzcd}[sep={15mm,between origins}]
				\bullet \ar[r, "a"] \ar[d, "b"] & \bullet \ar[d, dashed,  "b'"]\\
				\bullet \ar[r, dashed,  "a'"] & \bullet.
			\end{tikzcd}
		\end{center}
	\end{enumerate}

	By abuse of wording, the morphism $a'$ is called a \textbf{pullback}, resp.~\textbf{pushout}, of $a$ \textbf{along} $b$. The completed diagrams are also called a \textbf{pullback (square)}, resp.~\textbf{pushout (square)}. A square is called \textbf{bicartesian} if it is both a pullback and a pushout.
\end{dfn}

The following statements are known as the pasting laws for pullbacks and pushouts:

\begin{lem}[{\cite[Ex.~III.4.8]{Mac98}}] \label{lem: pasting-law}
	Consider the following commutative diagram in a category:
	\begin{center}
		\begin{tikzcd}[sep={17.5mm,between origins}]
			\bullet \ar[r] \ar[d] \ar[rd, phantom, "(X)"]
			& \bullet \ar[r] \ar[d] \ar[rd, phantom, "(Y)"] & \bullet \ar[d]\\
			\bullet \ar[r] & \bullet \ar[r] & \bullet
		\end{tikzcd}
	\end{center}

	\begin{enumerate}
		\item Suppose that $(Y)$ is a pullback. Then $(XY)$ is a pullback if and only if $(X)$ is so.
		\item Suppose that $(X)$ is a pushout. Then $(XY)$ is a pushout if and only if $(Y)$ is so.
	\end{enumerate}
	In particular, if $(X)$ and $(Y)$ are bicartesian, then so is $(XY)$. \qed
\end{lem}

In addition, we need the following reverse pasting laws over additive categories:

\begin{lem} \label{lem: pushpull-sep}
	Consider the following commutative diagram in an additive category:
	\begin{center}
		\begin{tikzcd}[sep={17.5mm,between origins}]
			A \ar[r, "r"] \ar[d, "a"] \ar[rd, phantom, "(X)"]
			& B \ar[r, "s"] \ar[d, "b"] \ar[rd, phantom, "(Y)"] & C \ar[d, "c"]\\
			A' \ar[r, "r'"] & B' \ar[r, "s'"]  & C'
		\end{tikzcd}
	\end{center}
	\begin{enumerate}
		\item \label{lem: pushpull-sep-pushout} If the outer square $(XY)$ is a pushout and $\begin{pmatrix} b & r'\end{pmatrix}\colon B \oplus A' \to B'$ is an epic, then $(Y)$ is a pushout.
		\item \label{lem: pushpull-sep-pullback} If the outer square $(XY)$ is a pullback and $\begin{pmatrix} s \\ b \end{pmatrix}\colon B \to C \oplus B'$ is a monic, then $(X)$ is a pullback.
	\end{enumerate}
\end{lem}

\begin{proof} \
	\begin{enumerate}[leftmargin=*]
		\item Consider the solid commutative diagram 
		
		\begin{center}
			\begin{tikzcd}[sep={17.5mm,between origins}]
				A \ar[r, "r"] \ar[d, "a"] & B \ar[r, "s"] \ar[d, "b"] & C \ar[d, "c"] \ar[rdd, "t"]\\
				A' \ar[r, "r'"] \ar[rrrd, "t'r'"'] & B' \ar[r, "s'"] \ar[rrd, "t'"] & C' \ar[rd, "u" near start, dashed] \\
				&&& P.
			\end{tikzcd}
		\end{center}
		
		\noindent
		The outer pushout square yields a unique dashed morphism $u$ with $u c = t$ and $u s' r' = t' r'$, hence with $u c = t$ and $u s' r' = t' r'$. Since $\begin{pmatrix} b & r' \end{pmatrix}$ is an epic, $u s' = t'$ follows from
		\begin{align*}
			u s' \begin{pmatrix} b & r' \end{pmatrix} &= \begin{pmatrix} u s' b & u s' r' \end{pmatrix} = \begin{pmatrix} u c s &  t' r' \end{pmatrix} = \begin{pmatrix} t s &  t' r' \end{pmatrix} =  \begin{pmatrix} t' b &  t' r' \end{pmatrix} = t' \begin{pmatrix} b & r' \end{pmatrix}.
		\end{align*}
		
		\item is dual to \ref{lem: pushpull-sep-pushout}. \qedhere
	\end{enumerate} 
\end{proof}

\begin{dfn}[{\cite[Def.~2.1]{Buh10}}] \label{dfn: exact}
	Let $\cS$ be a collection of pairs $(i, p)$ of morphisms in an additive category $\E$, where $i$ is a kernel of $p$ and $p$ is a cokernel of $i$. If $(i, p) \in \cS$, then $i$ is called an \textbf{($\boldsymbol \cS$-)admissible monic} and $p$ an \textbf{($\boldsymbol \cS$-)admissible epic}. The pairs $(i, p) \in \cS$ are referred to as a \textbf{short ($\boldsymbol \cS$-)exact sequences} in $\E$ and are displayed as
	\begin{center}
		\begin{tikzcd}[sep =large]
			A' \arrow[r, tail, "i"] & A \arrow[r, two heads, "p"] & A''.
		\end{tikzcd}
	\end{center}
	The collection $\cS$ is said to define an \textbf{exact structure} on $\E$ if $\cS$ is closed under isomorphisms and if the following axioms are satisfied:
	\begin{enumerate}
		\item For all objects $A \in \E$, the identity $\id_A$ is an admissible monic and epic.
		\item Composition preserves admissible monics and epics.
		\item \label{dfn: exact-pushpull} Pushout, resp.~pullback, along arbitrary morphisms exists for and preserves admissible monics, resp.~epics.
	\end{enumerate}
	
	In this case, $(\E, \cS)$, or just $\E$, is called an \textbf{exact category}.
\end{dfn}

\begin{rmk}[{\cite[Rem.~2.2]{Buh10}}] \label{rmk: E-op}
	If $\E$ is an exact category, then so is $\op \E$ with admissible monics and epics exchanged. Therefore, each statement on exact structures has a dual. For the sake of clarity, we formulate some less obvious dual statements explicitly.
\end{rmk}

\begin{dfn} Let $\A$ be an additive category.
	\begin{enumerate}
		\item A morphism in $\A$ is called a \textbf{split epic (monic)} if it has a right (left) inverse.
		\item A sequence of composable morphisms in $\A$ is called a  \textbf{split short exact} if it is isomorphic to
		\[
		\begin{tikzcd}
			A \ar[r] & A \oplus B \ar[r] & B
		\end{tikzcd}
		\]
		for $A, B \in \A$.
	\end{enumerate}
\end{dfn}

\begin{rmk}[{\cite[Rem.~7.4]{Buh10}}] \label{rmk: Buehler7.4}
	If a split epic $\begin{tikzcd}[cramped, sep=small] r\colon Y \ar[r, epi] & Z \end{tikzcd}$ in an additive category has a kernel $i\colon X \to Y$, then the sequence \begin{tikzcd}[cramped, sep=small] X \ar[r, "i"] & Y \ar[r, "r"] & Z \end{tikzcd} is split short exact.
\end{rmk}

\begin{prp}[{\cite[Prop.~2.9]{Buh10}}] \label{prp: Buehler2.9}
	In an exact category, finite direct sums of short exact sequences are again short exact. In particular, any split short exact sequence is short exact. \qed
\end{prp}

\begin{exa} \label{exa: exact} \
	\begin{enumerate}[leftmargin=*]
		\item \label{exa: exact-additive} Any additive category has an exact structure given by the split short exact sequences. We refer to it as the \textbf{split} exact structure.
		
		\item \label{exa: exact-abelian} Any Abelian category has the \textbf{maximal} exact structure with all monics and epics admissible.
		
		\item \label{exa: exact-diagram} Any diagram category over an exact category $\E$ has an exact structure, defined component-wise by the exact structure of $\E$. We refer to it as the \textbf{termwise} exact structure.
	\end{enumerate}
\end{exa}

\begin{prp}[Obscure axiom, {\cite[Prop.~2.16]{Buh10}}] \label{prp: obscure}
	In an exact category, the following statements hold:
	\begin{enumerate}
		\item \label{prp: obscure-mono} If a morphism $i\colon A \to B$ has a cokernel, and $b\colon B \to C$ is a morphism such that $bi\colon A \rightarrowtail C$ is an admissible monic, then $i$ is an admissible monic.
		
		\item \label{prp: obscure-epi} If a morphism $p\colon B \to C$ has a kernel, and $a\colon A \to B$ is a morphism such that $pa\colon A \twoheadrightarrow C$ is an admissible epic, then $p$ is an admissible epic. \qed
	\end{enumerate}
\end{prp}

\begin{cor}[{\cite[Cor.~2.18]{Buh10}}] \label{cor: ses-summand}
	Let $(i, p)$ and $(i', p')$ be two pairs of composable morphisms in an exact category. If their direct sum $(i \oplus i', p \oplus p')$ is short exact, then both $(i, p)$ and $(i', p')$ are short exact. \qed
\end{cor}

\begin{prp}[{\cite[Prop.~2.12]{Buh10}}] \label{prp: Buehler2.12} \
	\begin{enumerate}[leftmargin=*]
		\item \label{prp: Buehler2.12-push} For a square
		\begin{center}
			\begin{tikzcd}[sep={17.5mm,between origins}]
				A \ar[r, tail, "i"] \ar[d, "f"] & B \ar[d, "f'"] \\
				A' \ar[r, tail, "i'"] & B'
			\end{tikzcd}
		\end{center}
		in an exact category, the following statements are equivalent:
		\begin{enumerate}[label=(\arabic*)]
			\item \label{prp: Buehler2.12-push-1} The square is a pushout.
			\item \label{prp: Buehler2.12-push-2} The square is bicartesian.
			\item \label{prp: Buehler2.12-push-3} The sequence \begin{tikzcd}[sep=large] A \ar[r, "\begin{pmatrix} i \\ -f \end{pmatrix}", tail] & B \oplus A' \ar[r, "\begin{pmatrix} f' \hspace{2mm} i' \end{pmatrix}", two heads] & B' \end{tikzcd} is short exact.
			\item \label{prp: Buehler2.12-push-4} The square is part of a commutative diagram
			\begin{center}
				\begin{tikzcd}[sep={17.5mm,between origins}]
					A \ar[r, tail, "i"] \ar[d, "f"] & B \ar[d, "f'"] \ar[r, two heads] & C \ar[d, equal] \\
					A' \ar[r, tail, "i'"] & B' \ar[r, two heads] & C.
				\end{tikzcd}
			\end{center}
		\end{enumerate}
		
		\smallskip
		
		\item \label{prp: Buehler2.12-pull} For a square
		\begin{center}
			\begin{tikzcd}[sep={17.5mm,between origins}]
				A \ar[r, two heads, "p'"] \ar[d, "g'"] & B \ar[d, "g"] \\
				A' \ar[r, two heads, "p"] & B'
			\end{tikzcd}
		\end{center}
		in an exact category, the following statements are equivalent:
		\begin{enumerate}[label=(\arabic*)]
			\item The square is a pullback.
			\item The square is bicartesian.
			\item The sequence \begin{tikzcd}[sep=large] A \ar[r, "\begin{pmatrix} p' \\ g' \end{pmatrix}", tail] & B \oplus A' \ar[r, "\begin{pmatrix} - g \hspace{2mm} p \end{pmatrix}", two heads] & B' \end{tikzcd} is short exact.
			\item The square is part of a commutative diagram
			\begin{center}
				\begin{tikzcd}[sep={17.5mm,between origins}]
					K \ar[r, tail] \ar[d, equal] & A \ar[r, two heads, "p'"] \ar[d, "g'"] & B \ar[d, "g"] \\
					K \ar[r, tail] & A' \ar[r, two heads, "p"] & B'.
				\end{tikzcd}
			\end{center}
		\end{enumerate}
	\end{enumerate}
	\qed 
\end{prp}

\begin{cor} \label{cor: bicart} In an exact category the following statements hold:
	\begin{enumerate}
		\item \label{cor: bicart-already} Pushouts of  admissible monics and pullbacks of admissible epics are bicartesian squares.
		
		\item \label{cor: bicart-prep} A square
		\begin{center}
			\begin{tikzcd}[sep={17.5mm,between origins}]
				A \ar[r, tail] \ar[d] & B \ar[d, two heads] \\
				C \ar[r] & D
			\end{tikzcd}
		\end{center}
		is a pushout if and only if it is a pullback. In this case, opposite arrows are admissible morphisms of the same type.
	\end{enumerate}
\end{cor}

\begin{proof} \
	\begin{enumerate}[leftmargin=*]
		\item For the pushout, combine \Cref{prp: Buehler2.12}.\ref{prp: Buehler2.12-push}.\ref{prp: Buehler2.12-push-1} $\Leftrightarrow$ \ref{prp: Buehler2.12-push-2} with the pullback axiom, see \Cref{dfn: exact}.\ref{dfn: exact-pushpull}. The argument for the pullback is dual.
		\item If the given diagram is a pushout, then it is a pullback by \ref{cor: bicart-already} and $A \twoheadrightarrow C$ is an admissible epic by the pushout axiom, see \Cref{dfn: exact}.\ref{dfn: exact-pushpull}. The converse implication is dual. \qedhere
	\end{enumerate}
\end{proof}

\begin{prp}[{\cite[Prop.~2.15]{Buh10}}] \label{prp: Buehler2.15} In an exact category, pullback along an admissible epic preserves admissible monics and pushout along an admissible monic preserves admissible epics. \qed
\end{prp}

\begin{lem}[Noether lemma, {\cite[Ex.~3.7]{Buh10}}] \label{lem: Buehler3.7}
	Consider the commutative diagram
	
	\begin{center}
		\begin{tikzcd}[sep={17.5mm,between origins}]
			A' \ar[r, tail] \ar[d, tail] & B' \ar[r, two heads] \ar[d, tail] & C' \ar[d, tail, dashed]\\
			A \ar[r, tail] \ar[d, two heads] & B \ar[r, two heads]  \ar[d, two heads] & C  \ar[d, two heads, dashed]\\
			A'' \ar[r, tail] & B'' \ar[r, two heads] & C'' 
		\end{tikzcd}
	\end{center}
	
	in an exact category, where the rows and solid columns are short exact sequences. Then there exist unique morphisms $C' \to C$ and $C \to C''$ making the diagram commute, and $C' \rightarrowtail C \twoheadrightarrow C''$ a short exact sequence.  \qed
\end{lem}

\begin{dfn}
	A covariant functor $F\colon (\E', \cS') \to (\E, \cS)$ between exact categories is called \textbf{exact} if $(F(i), F(p)) \in \cS$ for all $(i, p) \in \cS'$. It is \textbf{fully exact} if, in addition, $(F(i), F(p)) \in \cS$ implies $(i, p) \in \cS'$, for all pairs $(i, p)$ of composable morphisms in $\E'$. Obvious dual notions are defined for contravariant functors.\\
	A subcategory $\E'$ of $\E$ is called  \textbf{(fully) exact} if it is an exact category itself and the inclusion functor $\E' \to \E$ is (fully) exact.\footnote{Bühler uses the term \emph{fully exact} for what we call \textit{extension-closed}, see \Cref{def: ext-closed}.} Note that subcategories of additive categories are fully exact with respect to the split exact structure, see \Cref{exa: exact}.\ref{exa: exact-additive}.
\end{dfn}

\begin{prp}[{\cite[Prop.~5.2]{Buh10}}] \label{prp: exact-pushpull}
	An exact functor preserves pushouts along admissible monics and pullbacks along admissible epics. \qed
\end{prp}

\begin{dfn}[{\cite[Def.~10.21]{Buh10}}] \label{def: ext-closed}
	Let $\E'$ be a subcategory of an exact category $\E$. We call $\E'$  \textbf{extension-closed} if any $X \in \E$ which fits into a short exact sequence \begin{tikzcd}[cramped] Y' \ar[r, tail] & X \ar[r, two heads] & Y\end{tikzcd} with $Y,Y'\in \E'$ is an object of $\E'$.
\end{dfn}

Part \ref{lem: fullyexact-extclosed} of \Cref{lem: fullyexact} is {\cite[Lem.~10.20]{Buh10}}, part \ref{lem: fullyexact-kerclosed} follows from \Cref{prp: Buehler2.12}:

\begin{lem} \label{lem: fullyexact} A subcategory $\E'$ of an exact category $\E$ is fully exact if one of the following conditions holds:
	
	\begin{enumerate}
		\item \label{lem: fullyexact-extclosed} $\E'$ is an extension-closed subcategory of $\E$.
		\item \label{lem: fullyexact-kerclosed} $\E'$ is closed in $\E$ under kernels of admissible epics and cokernels of admissible monics.
	\end{enumerate}
	
	In both cases, the exact structure is given by short exact sequences in $\E$ with objects in $\E'$. \qed
\end{lem}

\begin{ntn} \label{ntn: C}
	The category of sequences over an additive category $\A$ is the diagram category $\C(\A) := \text{Func}(T,\A)$ where $T$ is the infinite linear quiver
	
	\begin{center}
		\begin{tikzcd} \cdots \ar[r] & \underset {-1} \bullet \ar[r] & \underset 0 \bullet \ar[r] & \underset 1 \bullet \ar[r] & \cdots \end{tikzcd}
	\end{center}
	
	with vertices indexed by $\ZZ$ in ascending order. We denote denote objects of $\C(\A)$ by $X=(X,d_X)$, where $X=(X^k)_{k \in \ZZ}$ and $d_X= (d^k_X)_{k \in \ZZ}$ with $d^k_X\colon X^k \to X^{k+1}$ for $k \in \ZZ$. We omit the index $k$ of $d^k_X$ when it is clear from the context. Incomplete sequences are extended by zeros without explicit mention.
\end{ntn}

\begin{rmk} \label{rmk: C-direct-sum}
	Termwise finite coproducts of sequences over an additive category exist.
\end{rmk}

\begin{rmk} \label{rmk: C-exact}
	If $\E$ is an exact category, then $\C(\E)$ has two natural exact structures: The \textbf{termwise} exact structure from \Cref{exa: exact}.\ref{exa: exact-diagram} and the \textbf{termsplit} exact structure defined in the same way by the underlying additive category of $\E$, see \Cref{exa: exact}.\ref{exa: exact-additive}. Unless mentioned otherwise the termsplit exact structure is the \emph{default} choice.\\
	If $\E'$ is a (fully) exact subcategory of $\E$, then so is $\C(\E')$ in $\C(\E')$ due to the termwise exact structure. In particular, $C_N(\A')$ is fully exact in $C_N(\A)$ for a subcategory $\A'$ of an additive category $\A$.
\end{rmk}

\subsection{Stable categories}

\begin{dfn}
	An object $P$ of an exact category $\E$ is called a \textbf{projective} if the covariant functor $\Hom_\E(P, -)\colon \E \to \Ab$ is exact. Dually, an object $I$ of $\E$ is called an \textbf{injective} if the contravariant functor $\Hom_\E(-, I)\colon \E \to \Ab$ is exact. The respective subcategories of $\E$ are denoted  by $\Prj(\E)$ and $\Inj(\E)$. An object is called a  \textbf{projective-injective} if it is both projective and injective.
\end{dfn}

\begin{prp}[{\cite[Prop.~11.3, Cor.~11.4]{Buh10}}] \label{prp: proj-equiv}
	An object $P$ of an exact category $\E$ is projective if and only if any one of the following equivalent conditions holds:
	\begin{enumerate}[label=(\arabic*)]
		\item \label{prp: proj-equiv-lift} For any admissible epic $X \twoheadrightarrow Y$, any morphism $P \to Y$ lifts as follows:
		\begin{center}
			\begin{tikzcd}[sep={17.5mm,between origins}]
				& X \ar[d, two heads] \\
				P \ar[ru, dashed] \ar[r] & Y
			\end{tikzcd}
		\end{center}
	
		\item \label{prp: proj-equiv-Hom} The functor $\Hom_\E(P, -)$ sends admissible epics to surjections.
		
		\item \label{prp: proj-equiv-split} Every admissible epic $X \twoheadrightarrow P$ splits.
	\end{enumerate}
	If a morphism $P \to Z$ with $P \in \Prj(\E)$ admits a right inverse, then $Z$ is projective as well. \qed
\end{prp}

\begin{rmk} \label{rmk: Proj-exact}
	Due to \Cref{prp: proj-equiv}.\ref{prp: proj-equiv-split}, \Cref{rmk: Buehler7.4}, and their duals, the subcategories $\Prj(\E)$ and $\Inj(\E)$ of an exact category $\E$ are closed under direct summands and extensions and thus fully exact, see \Cref{lem: fullyexact}.\ref{lem: fullyexact-extclosed}. Their exact structure is the split exact structure, see \Cref{exa: exact}.\ref{exa: exact-additive}.
\end{rmk}

\begin{ntn}
	The \textbf{projectively  stable category} of  an exact category $\E$ is denoted by $\underline \E$. It has the same objects as $\E$, and two morphisms agree if their difference factors through a projective. Dually, we denote the \textbf{injectively stable category} by $\overline \E$. These constructions have a universal property, see {\cite[Prop.~II.8.1]{Mac98}}.
\end{ntn}

\begin{rmk} \label{rmk: factor}
	In an exact category $\E$, the following statements hold:
	\begin{enumerate}[leftmargin=*]
		\item \label{rmk: factor-proj} By \Cref{prp: proj-equiv}.\ref{prp: proj-equiv-lift}, a morphism $f\colon X \to Y$ in $\E$ is zero in $\underline \E$ if and only if it factors through any admissible epic $p_Y\colon P \twoheadrightarrow Y$ with $P \in \Prj(\E)$. Dually, $f$ is zero in $\overline \E$ if and only if it factors through any admissible monic $i_X\colon X \rightarrowtail I$ with $I \in \Inj(\E)$.
		\item \label{rmk: factor-zero} If $A \in \E$ with $A = 0$ in $\underline \E$, then $A \in \Prj(\E)$. Indeed, $\id_A=0$ in $\underline \E$ yields a $P \in \Prj(\E)$ and a morphism $P \to A$ in $\E$ with right inverse. Then $A \in \Prj(\E)$ due to the particular statement of \Cref{prp: proj-equiv}.
	\end{enumerate}
\end{rmk}

\begin{dfn}
	We say that an exact subcategory $\E'$ of $\E$ has \textbf{enough $\E$-projectives} if there is an admissible epic $P \twoheadrightarrow X$ in $\E'$ with $P \in \Prj(\E)$ for each $X \in \E'$. Having \textbf{enough $\E$-injectives} is defined dually. If this holds for $\E' = \E$, one says that $\E$ has enough \textbf{enough projectives}, resp.~\textbf{enough injectives}.
\end{dfn}

\begin{rmk} \label{rmk: enough-proj} An exact subcategory $\E'$ of $\E$ has enough $\E$-projectives if $\E'$ has enough projectives and $\Prj(\E') \subseteq \Prj(\E)$.
\end{rmk}

\begin{dfn} A \textbf{Frobenius (exact) category} is an exact category $\F$ with enough injectives, enough projectives, and $\Prj(\F) = \Inj(\F)$. In this case, $\underline \F = \overline \F$ is called the \textbf{stable category} of $\F$. By a \textbf{sub-Frobenius category} of a Frobenius category $\F$, we mean an exact subcategory $\F'$ which has enough $\F$-projectives and enough $\F$-injectives. This terminology is justified by \Cref{lem: sub-Frobenius}.\ref{lem: sub-Frobenius-b}.
\end{dfn}

\begin{con} \label{con: cone} Let $\E$ be an exact category with enough injectives. For each $X \in \E$, pick an admissible monic $i = i_X\colon X \rightarrowtail I(X)$ with $I(X) \in \Inj(\E)$ and cokernel denoted by $\Sigma X$. If $f\colon X \to Y$ is a morphism in $\E$, then \Cref{prp: Buehler2.12}.\ref{prp: Buehler2.12-push} yields a pushout diagram
	\begin{equation} \label{eqn: std-con} \tag{$D(f)$}
		\begin{tikzcd}[sep={17.5mm,between origins}] 
			X \ar[r, tail, "i"] \ar[d, "f"] \ar[rd, phantom, "\square"] & I(X) \ar[r, two heads] \ar[d, "f'"] & \Sigma X \ar[d, equal] \\
			Y \ar[r, tail, "i'"] & C(f) \ar[r, two heads] & \Sigma X
		\end{tikzcd}
	\end{equation}
	and a short exact sequence
	
	\begin{equation} \label{eqn: std-sequence} \tag{$S(f)$}
		\begin{tikzcd}[sep=large]
			X \ar[r, "\begin{pmatrix} i \\ -f \end{pmatrix}", tail] & I(X) \oplus Y \ar[r, "\begin{pmatrix} f' \hspace{2mm} i' \end{pmatrix}", two heads] & C(f).
		\end{tikzcd}
	\end{equation}
	
	The object $C(f) \in \E$ is called the \textbf{(mapping) cone} of $f$. A \textbf{standard triangle} in $\underline \E$ is any sequence of the form
	\begin{equation} \label{eqn: std-triangle} \tag{$T(f)$}
		\begin{tikzcd}[sep=large] 
			X \ar[r, "f"] & Y \ar[r] & C(f) \ar[r] & \Sigma X.
		\end{tikzcd}
	\end{equation}
	Dually, if $\E$ has enough projectives, the \textbf{(mapping) cocone} $C^\ast(f)$ of $f$ fits into a pullback diagram
	\begin{equation} \label{eqn: std-con-dual} \tag{$D^\ast(f)$}
		\begin{tikzcd}[sep={17.5mm,between origins}]
			\Sigma^{-1}Y \ar[r, tail] \ar[d, equal] & C^\ast(f) \ar[r, two heads] \ar[d] \ar[rd, phantom, "\square"] & X \ar[d, "f"] \\
			\Sigma^{-1}Y \ar[r, tail] & P(Y) \ar[r, two heads, "p"] & Y
		\end{tikzcd}
	\end{equation}
	and into a short exact sequence
	\begin{equation} \label{eqn: std-sequence-dual} \tag{$S^\ast(f)$}
		\begin{tikzcd}
			C^\ast(f) \ar[r, tail] & X \oplus P(Y) \ar[r, two heads, "\begin{pmatrix} -f \hspace{2mm} p \end{pmatrix}"] & Y,
		\end{tikzcd}
	\end{equation}
	where $p=p_Y: P(Y) \twoheadrightarrow Y$ is an admissible epic with $P(Y) \in \Prj(\E)$, and $\Sigma^{-1}Y$ denotes the cokernel of $p$.  Note that $\Sigma X \cong C(X \to 0)$ and $\Sigma^{-1}Y \cong C^\ast(0 \to Y)$ in $\E$.
\end{con}

\begin{thm}[{\cite[Thm.~2.6]{Hap88}}] \label{thm: stable-Frobenius}
	The stable category of a Frobenius category is triangulated. Its suspension functor $\Sigma$ and the quasi-inverse $\Sigma^{-1}$ are defined as in \Cref{con: cone}. The distinguished triangles are those candidate triangles isomorphic to standard triangles. \qed
\end{thm}

\begin{lem}[{\cite[Lem.~2.7]{Hap88}}] \label{lem: ses-triangle} Any short exact sequence \begin{tikzcd}[cramped] X \ar[r, tail, "i"] & Y \ar[r, two heads, "p"] & Z \end{tikzcd}
	in a Frobenius category $\F$ induces a distinguished triangle
	\begin{tikzcd}[cramped] 
			X \ar[r, "i"] & Y \ar[r, "p"] & Z \ar[r] & \Sigma X
		\end{tikzcd}
	in $\underline \F$. \qed
\end{lem}

\begin{rmk} \label{rmk: cocone-cone}
Rotating the distinguished triangle obtained from \eqref{eqn: std-sequence-dual} by \Cref{lem: ses-triangle} and comparing with \eqref{eqn: std-triangle} yields an isomorphism $\Sigma C^\ast(f) \cong C(f)$ in $\underline \F$.
\end{rmk}

\begin{prp}[{\cite[Prop.~7.3]{IKM16}}] \label{prp: stable-functor-triang}
	If $F\colon \F' \to \F$ is an exact functor between Frobenius categories which preserves projective-injectives\footnote{The hypothesis can be restricted to the objects of the form $I(X)$ and $P(Y)$ from \Cref{con: cone}.}, then the induced functor $\underline F\colon \underline \F' \to \underline \F$ is triangulated.\qed
\end{prp}

\begin{lem} \label{lem: sub-Frobenius}
	Let $\E'$ be an exact subcategory of $\E$.
	\begin{enumerate}
		\item \label{lem: sub-Frobenius-a} If $\E'$ has enough $\E$-projectives, then $\Prj(\E') = \Prj(\E) \cap \E'$, and the canonical functor $ \underline \E' \to \underline \E$ is fully faithful.
		
		\item \label{lem: sub-Frobenius-b} If $\F'$ is a sub-Frobenius category of $\F$, then $\F'$ is Frobenius, and the canonical functor $ \underline \F' \to \underline \F$ is fully faithful and triangulated.
	\end{enumerate}
\end{lem}

\begin{proof}\
	\begin{enumerate}[leftmargin=*]
		\item Let $X \in \Prj(\E')$ and pick an admissible epic $p\colon P \twoheadrightarrow X$ in $\E'$ with $P\in \Prj(\E)$. Then $p$ has a right inverse due to \Cref{prp: proj-equiv}.\ref{prp: proj-equiv-lift} applied to $X \in \Prj(\E')$ and $\id_X$. We obtain $X \in \Prj(\E)$ by the particular statement of \Cref{prp: proj-equiv}. The converse inclusion holds trivially by \Cref{prp: proj-equiv}.\ref{prp: proj-equiv-lift}. Now $\E' \subseteq \E$ induces a full functor $ \underline \E' \to \underline \E$. For faithfulness, consider a morphism $f\colon X \to Y$ in $\E'$ which is zero in $\underline \E$. By assumption, there is an admissible epic $p\colon P \twoheadrightarrow Y$ in $\E$ with $P \in \Prj(\E) \cap \E' = \Prj(\E')$. Then $f$ factors through $p$ and is zero in $\underline \E'$, see \Cref{rmk: factor}.\ref{rmk: factor-proj}.
		
		\item By \ref{lem: sub-Frobenius-a}, $\F'$ is Frobenius, and $\F' \subseteq \F$ induces a fully faithful functor $\underline \F \to \underline \F'$. It is triangulated by \Cref{prp: stable-functor-triang}. \qedhere
	\end{enumerate} 
\end{proof}

\subsection{Acyclicity and syzygies}

\begin{dfn}[{\cite[Def.~8.1]{Buh10}}] \label{dfn: admissible}
	An \textbf{admissible} morphism in an exact category is the composition of an admissible epic and an admissible monic, displayed as
	
	\begin{center}
		\begin{tikzcd}[sep={15mm,between origins}]
			X \ar[rr, "\circ" marking] \ar[rd, two heads] && Y \\
			& I. \ar[ru, tail]&
		\end{tikzcd}
	\end{center}
\end{dfn}

\begin{rmk}
	The admissible morphisms, which are monics (epics), are exactly the admissible monics (epics) in \Cref{dfn: exact}.
\end{rmk}

\begin{rmk}[{\cite[Lem.~8.4, Rem.~8.5, Ex.~8.6]{Buh10}}]\ \label{rmk: admissible}
	\begin{enumerate}[leftmargin=*]
		\item \label{rmk: admissible-unique} The defining factorization of an admissible morphism is unique up to a unique isomorphism.
		\item \label{rmk: admissible-analysis} Any admissible morphism $f$ has a so-called \textbf{analysis}
		
		\begin{center}
			\begin{tikzcd}[sep={15mm,between origins}]
				& X \ar[rd, two heads, "e"] \ar[rr, "\circ" marking, "f"] && Y \ar[rd, two heads, "c"] \\
				Z \ar[ru, tail, "k"] &&I \ar[ru, tail, "m"] && C,
			\end{tikzcd}
		\end{center}
		where $k$ is a kernel, $c$ is a cokernel, $e$ is a coimage, and $m$ is an image of $f$. In particular, all these morphisms are uniquely determined by $f$ up to a unique isomorphism.
		
		\item \label{rmk: admissible-Abelian} Only for Abelian categories the class of admissible morphisms is closed under composition.
	\end{enumerate}
	
\end{rmk}

\begin{dfn}[{\cite[Def.~8.8]{Buh10}}] \label{dfn: 2-acyclic} Let $\E$ be an exact category.
	\begin{enumerate}
		\item \label{dfn: 2-acyclic-local} A sequence \begin{tikzcd}[cramped] X' \ar[r]  & X \ar[r] & X'' \end{tikzcd} of two morphisms in $\E$ is called \textbf{acyclic} if both are admissible morphisms and their factorizations
		\begin{center}
			\begin{tikzcd}[sep={15mm,between origins}]
				X' \ar[rr, "\circ" marking] \ar[rd, two heads] && X \ar[rd, two heads] \ar[rr, "\circ" marking] && X''   \\
				& Z \ar[ru, tail] && C \ar[ru, tail]
			\end{tikzcd}
		\end{center}
		give rise to a short exact sequence \begin{tikzcd}[cramped] Z \ar[r, tail] & X \ar[r, two heads] & C \end{tikzcd}.
		
		\smallskip
		
		\item \label{dfn: 2-acyclic-global} A sequence $X = (X^k)_{k \in \ZZ} \in \C(\E)$ is called \textbf{$\boldsymbol 2$-acyclic at position $ n \in \ZZ$} if the sequence \begin{tikzcd}[sep=small, cramped] X^{n-1} \ar[r] & X^n \ar[r] & X^{n+1} \end{tikzcd} is acyclic in the sense of \ref{dfn: 2-acyclic-local}.
		It is called \textbf{2-acyclic} if it is 2-acyclic at all positions $n \in \ZZ$.
	\end{enumerate}
\end{dfn}

\begin{dfn} \label{dfn: syz}
	A \textbf{projective resolution} of an object $X \in \E$ is a sequence $P \in \C(\E)$ with $P^k \in \Prj(\E)$ for $k \in \ZZ_{\leq 0}$ and $P= 0$ for $k \in \ZZ_{> 0}$ which fits into a 2-acyclic sequence
	\begin{center}
		\begin{tikzcd}[sep={12mm,between origins}]
			P\colon \; \cdots  \ar[rr,"\circ" marking] && P^{-n} \ar[rr,"\circ" marking] \ar[rd, two heads] && P^{-n-1} \ar[rr,"\circ" marking] && \cdots \ar[rr,"\circ" marking] && P^{-1} \ar[rr,"\circ" marking] \ar[rd, two heads] && P^{0} \ar[rr, two heads] \ar[rd, two heads] && X \\
			&&& C^{-n} \ar[ru, tail] &&&&&& C^{-1} \ar[ru, tail] && C^{0}. \ar[ru, equal]
		\end{tikzcd}
	\end{center}
	It exists for $X \in \E$ if $\E$ has enough projectives, see {\cite[Prop.~12.2]{Buh10}}. 
	We refer to the object $\syz^n_P(X) := C^{-n}$ as an \textbf{$\boldsymbol n$th syzygy} of $X$, for any $n \in \NN$.
\end{dfn}

As in the Abelian case,  Schanuel's Lemma holds in any exact category, and iterated application yields the well-definedness of syzygies up to projective equivalence.

\begin{lem}[\textbf{Schanuel's Lemma}]
	\label{lem: schnanuel}
	Consider two short exact sequences
	\begin{center}
		\begin{tikzcd}[sep=large]
			Z \arrow[r, tail, "i"]& P  \arrow[r, two heads, "p"]& X & \textit{and} & Z'  \arrow[r, tail, "i'"]& P' \arrow[r, two heads, "p'"]& X
		\end{tikzcd}
	\end{center}
	in an exact category $\E$ with $P, P' \in \Prj(\E)$. Then $Z \oplus P' \cong Z' \oplus P$.
\end{lem}

\begin{proof}
	Due to \Cref{prp: proj-equiv}.\ref{prp: proj-equiv-lift}, we obtain a commutative diagram
	\begin{center}
		\begin{tikzcd}[sep={17.5mm,between origins}]
			Z \arrow[r, tail, "i"] \arrow[d, dashed, "b"] & P \arrow[r, two heads, "p"] \arrow[d, dashed, "a"]& X \arrow[d, equal]\\
			Z' \arrow[r, tail, "i'"]& P' \arrow[r, two heads, "p'"]& X.
		\end{tikzcd}
	\end{center}
	Then \Cref{prp: Buehler2.12}.\ref{prp: Buehler2.12-push} yields a short exact sequence $Z \rightarrowtail Z' \oplus P \twoheadrightarrow P'$ and the claim follows from \Cref{prp: proj-equiv}.\ref{prp: proj-equiv-split} and \Cref{rmk: Buehler7.4}.\footnote{A proof of the dual statement can be found in {\cite[Prop.~3.1]{MR22}}.}
\end{proof}

\begin{prp} \label{prp: syz-welldef}
	Let $\E$ be an exact category. If $P$ and $Q$ are two projective resolutions of $X \in \E$, then $\syz^n_P(X) \oplus \tilde Q^n \cong \syz^n_Q(X) \oplus \tilde P^n$ for any $n \in \NN$ and suitable $\tilde P^n, \tilde Q^n \in \Prj(\E)$. \qed
\end{prp}

\subsection{Categories of $N$-complexes} \label{subsection: N-complexes}

In this subsection, we review some foundational results on $N$-complexes by Iyama, Kato and Miyachi {\cite{IKM17}}.

\begin{dfn} Let $N \in \NN$ with $N \geq 2$.
	A sequence $X \in \C(\A)$ over an additive category $\A$ is called an $\boldsymbol N$\textbf{-complex} if the $N$-fold composition $d_X^{n+N-1} \circ \cdots \circ d_X^n$ is zero for all $n \in \ZZ$. The subcategory of $\C(\A)$ consisting of $N$-complexes is denoted by $\C_N(\A)$.
\end{dfn}

\begin{rmk} \label{rmk: C_N-subcat}
	If $\E$ is an exact category and \begin{tikzcd}[cramped] X \ar[r, tail] & Y \ar[r, two heads] & Z \end{tikzcd} is a termwise short exact sequence in $\C(\E)$ with $Y \in \C_N(\E)$, then also $X, Z \in \C_N(\E)$ by pre- and postcomposition. So, $\C_N(\E)$ is a fully exact subcategory of $\C(\E)$ with the termwise exact structure by \Cref{lem: fullyexact}.\ref{lem: fullyexact-kerclosed}. In particular, $\C_N(\E)$ is a fully exact subcategory of $\C(\E)$ with its default exact structure, see \Cref{rmk: C-exact}.\\
	On the fully exact subcategory $\C_N(\Prj(\E))$ of $\C_N(\E)$, both the termwise and the termsplit exact structure coincide with the exact structure as a fully exact subcategory of $\C(\Prj(\E))$, see \Cref{rmk: Proj-exact}. This holds verbatim for $\C_N(\Inj(\E))$.
\end{rmk}

\begin{ntn} \label{ntn: Hom}
	Given an object $A \in \A$ of an additive category $\A$, the covariant Hom functor $\Hom_\A(A, -): \C(\A) \to \C(\Ab)$ is defined by  $\Hom_\A(A, X)^n := \Hom_\A(A, X^n)$ for $X \in \C(\A)$ and $n \in \ZZ$. Dually, the contravariant Hom functor is defined by $\Hom_\A(X, A)^n :=\Hom_\A(X^{-n}, A)$.
\end{ntn}

\begin{ntn}[{\cite[$\S2$]{IKM17}}] \label{ntn: mu}
	Let $\A$ be an additive category. For $A \in \A$, $s \in \ZZ$, and $t \in \{1,\dots,N-1\}$, the $N$-complex
	\begin{center}
		\begin{tikzcd}
			\mu_t^s(A)\colon \; \cdots \ar[r] & 0 \ar[r] & A^{s-t+1} \ar[r, "\id_A"] & \cdots \ar[r, "\id_A"] & A^{s-1} \ar[r, "\id_A"] & A^s \ar[r] & 0 \ar[r]& \cdots	
		\end{tikzcd}
	\end{center}
	is defined by $A^k=A$ for all $k \in \{s-t+1,\dots, s\}$.
\end{ntn}

\begin{rmk} \label{rmk: mu-dual}
	Let $\A$ be an additive category. For $A,B \in \A$, $s \in \ZZ$, and $t \in \{1,\dots,N-1\}$, we have $\Hom_\A(B, \mu^s_t(A)) = \mu^{s}_t(\Hom_\A(B, A))$ and $\Hom_\A(\mu^s_t(A), B) = \mu^{-s+t-1}_t(\Hom_\A(A,B))$.
\end{rmk}

\begin{rmk}[{\cite[(2.1)]{IKM17}}] \label{rmk: mu-adjunction}
	For any $N$-complex $X \in \C_N(\A)$, $A \in \A$, and $s \in \ZZ$, there are functorial isomorphisms 
	\[ \Hom_\A(A, X^{s}) \cong  \Hom_{\C_N(\A)}(\mu^{s+N-1}_N(A), X) \hspace{5mm} \text{and} \hspace{5mm}   \Hom_\A(X^s, A) \cong \Hom_{\C_N(\A)}(X, \mu^s_N(A))\]
	of Abelian groups given by mapping $f \in \Hom_\A(A, X^{s}) $ and $g \in \Hom_\A(X^s, A)$ to $p^s_f$ and $i^s_g$, respectively, as depicted in the following commutative diagram:
	\begin{center}
		\begin{tikzcd}[sep={17.5mm,between origins}]
			\mu^{s+N-1}_N(A)\colon \ar[d, "p^{s}_f"] & \cdots \ar[r] & 0 \ar[r] \ar[d] & A \ar[d, "f"] \ar[r, "\id_A"] & \cdots \ar[r, "\id_A"] & A \ar[d, "d_X^{\{N-1\}}f"] \ar[r] & 0 \ar[d] \ar[r] & \cdots \\
			
			X\colon  & \cdots \ar[r, "d_X"] & X^{s-1} \ar[r, "d_X"] & X^{s} \ar[r, "d_X"] & \cdots \ar[r, "d_X"] & X^{s+N-1} \ar[r, "d_X"] & X^{s+N} \ar[r, "d_X"] & \cdots \\[-5mm]
			X\colon  \ar[d, "i^s_g"] & \cdots \ar[r, "d_X"] & X^{s-N} \ar[d] \ar[r, "d_X"] & X^{s-N+1} \ar[d, "d_X^{\{N-1\}}g"] \ar[r, "d_X"] & \cdots \ar[r, "d_X"] & X^s \ar[d, "g"] \ar[r, "d_X"] & X^{s+1} \ar[d] \ar[r, "d_X"] & \cdots \\
			
			\mu^s_N(A)\colon & \cdots \ar[r] & 0 \ar[r] & A \ar[r, "\id_A"] & \cdots \ar[r, "\id_A"] & A \ar[r] & 0 \ar[r] & \cdots 
		\end{tikzcd}
	\end{center}
	Note that $\Hom_\A(i^s_g, B) = p^{-s}_{\Hom_\A(g, B)}$ and $\Hom_\A(p^s_f, B)=i^{-s}_{\Hom_\A(f, B)}$ for any $B \in \A$.
\end{rmk}

\begin{lem} [{\cite[Lem.~2.2]{IKM17}}] \label{lem: mu-ProjInj}
	For an additive category $\A$, the $N$-complexes of the form $\mu_t^s(A)$ are projective-injectives of $\C_N(\A)$ for each $A \in \A$, $s \in \ZZ$, and $t \in \{1,\dots,N-1\}$. \qedhere
\end{lem}

\begin{con}[{\cite[(2.2)]{IKM17}}] \label{con: I-and-P}
	  For any additive category $\A$, the exact category $C_N(\A)$ has enough projectives and injectives: For each $X \in \C_N(\A)$, there is an
		\begin{itemize}
			\item admissible monic $i_X := \left(i^k_{\id_{X^k}}\right)_{k \in \ZZ}\colon X \rightarrowtail I(X)$ with $I(X)=I_N(X) := \bigoplus_{k \in \ZZ} \mu^k_N(X^k)$  injective, and an
			\item admissible epic $p_X := \left(p^k_{\id_{X^{k}}}\right)_{k \in \ZZ}\colon P(X) \twoheadrightarrow X$ with $P(X) =P_N(X) := \bigoplus_{k \in \ZZ} \mu^k_N(X^{k-N+1})$ projective,
		\end{itemize}
		see \Cref{rmk: C-direct-sum,rmk: mu-adjunction} and \Cref{lem: mu-ProjInj}. These are our default choices for \Cref{con: cone}.
\end{con}

\begin{rmk} \label{rmk: I-P-functor}
	Note that $I$ and $P$, as defined in \Cref{con: I-and-P}, are functorial: If $f: X \to Y$ is a morphism of $N$-complexes over an exact category $\E$, then the component $I(f)^n$, resp.~$P(f)^n$, of the lift is given on $P^k$ by $f^k$, where $k \in \{n, \dots, n+N-1\}$, resp.~$k \in \{n-N+1, \dots, n\}$.
\end{rmk}

\begin{rmk} \label{rmk: I-P-Hom}
	For an $N$-complex $X \in \C_N(\A)$ over an additive category $\A$ and $A \in \A$, we have in $C_N(\Ab)$:
	\begin{enumerate}[leftmargin=*]
		\item \label{rmk: I-P-Hom-equality}  $\Hom_\A(i_X, A) = p_{\Hom_{\A}(X, A)}$ and $\Hom_\A(p_X, A) = i_{\Hom_{\A}(X, A)}$, see \Cref{rmk: mu-adjunction}.
		
		\item \label{rmk: I-P-Hom-cone} In particular, applying $\Hom_\A(-, A)$ to the sequences \eqref{eqn: std-sequence} and \eqref{eqn: std-sequence-dual} yields \[\Hom_\A(C(f), A) \cong C^\ast(\Hom_\A(f, A)) \text{ and } \Hom_\A(C^\ast(f), A) \cong C(\Hom_\A(f, A)).\]
	\end{enumerate}	
\end{rmk}

\begin{thm}[{\cite[28]{Hap88}}, {\cite[Thm.~2.1]{IKM17}}] \label{thm: C-Frob}
	The exact category $\C_N(\A)$ of $N$-complexes over an additive category $\A$ is Frobenius (with the termsplit exact structure).\qed
\end{thm}

\begin{ntn}
	The stable category of $\C_N(\A)$ is denoted by $\K_N(\A)$. This is a triangulated category, see \Cref{thm: stable-Frobenius,thm: C-Frob}.
\end{ntn}

\begin{rmk} \label{rmk: K-in-K}
	Let $\A'$ be a subcategory of an additive category $\A$. Due to \Cref{rmk: C-exact,rmk: enough-proj}, \Cref{con: I-and-P}, and \Cref{thm: C-Frob}, all assumptions in \Cref{lem: sub-Frobenius}.\ref{lem: sub-Frobenius-b} are satisfied, and $\C_N(\A') \subseteq \C_N(\A)$ induces a fully faithful, triangulated functor $\K_N(\A') \to \K_N(\A)$.
\end{rmk}

\begin{ntn} \label{ntn: d^r} Let $X \in \C(\A)$ be a sequence over an additive category $\A$.
	\begin{itemize}
		\item We write $d_X^{\{r\}} =(d_X^{n+r-1} \circ \cdots \circ d_X^n)_{n \in \ZZ}$ for the $r$th power of $d_X$.
		
		\item The \textbf{shift functor} $\Theta: \C(\A) \to \C(\A)$ is given by $(\Theta X)^k := X^{k+1}$ and $d^k_{\Theta X} := d^{k+1}_{X}$.
		
		\item By $- X$ we denote the sequence with $(- X)^k := X^k$ and $d_{-X}^k := - d^k_{X}$.
	\end{itemize}
\end{ntn}

\begin{con} \label{con: Sigma-explicit}
	Let $f\colon X \to Y$ be a morphism of $N$-complexes over an additive category $\A$.
	\begin{enumerate}[leftmargin=*]
		\item \label{con: Sigma-explicit-formula} In {\cite[693ff.]{IKM17}}, there is an explicit description of the cone $C(f)$ and the cocone $C^\ast(f)$ and their special cases, the suspension functor $\Sigma$ and its quasi-inverse $\Sigma^{-1}$:
		\[(\Sigma X)^n = \bigoplus_{k=1}^{N-1}X^{n+k}, \hspace{5mm} d_{\Sigma X} = \left( \begin{array}{c|ccccc}
		0 & & & E_{N-2} \\ \hline
		-d_X^{\{N-1\}} & -d_X^{\{N-2\}} & -d_X^{\{N-3\}} & \cdots & -d_X^{\{2\}} & -d_X
		\end{array} \right),\]
	
	\[(\Sigma^{-1} X)^n = \bigoplus_{k=-N+1}^{-1}X^{n+k}, \hspace{5mm} d_{\Sigma^{-1} X} = \left( \begin{array}{c|c}
		-d_X   \\
		-d_X^{\{2\}}   \\
		\vdots & E_{N-2} \\
		-d_X^{\{N-3\}}  \\
		-d_X^{\{N-2\}}  \\ \hline
		-d_X^{\{N-1\}}  & 0
	\end{array} \right),\]

	\[C(f)^n = Y^n \oplus \bigoplus_{k=1}^{N-1} X^{n+k}, \hspace{5mm} d_{C(f)} = \left(\begin{array}{c|ccccc}
		d_Y&f&0&\cdots&0 & 0 \\ \hline
		0 & & & d_{\Sigma X}
	\end{array}\right),\]

	\[C^\ast(f)^n = \bigoplus_{k=-N+1}^{-1} Y^{n+k} \oplus X^n, \hspace{5mm} d_{C^\ast(f)} = \left(\begin{array}{c|c}
		& 0 \\
		& 0 \\
		d_{\Sigma^{-1} Y} & \vdots \\
		& 0 \\
		& - f \\ \hline
		0 & d_X
	\end{array}\right). \]
	
	\item \label{con: Sigma-explicit-cone} Note that $C(X \to 0) = \Sigma X$ and $C^\ast(0 \to Y) = \Sigma^{-1} Y$, see \Cref{con: cone}.
	
	\item \label{con: Sigma-explicit-theta} Comparing cone and cocone in the case $N=2$ yields $\Theta C^\ast(f) = - C(f)$, see \Cref{rmk: cocone-cone}.
	\end{enumerate}
\end{con}

\begin{thm}[{\cite[Thm.~2.4]{IKM17}}] \label{thm: IKM2.4}
	There is a functorial isomorphism $\Sigma^2 \cong \Theta^N$ of endofunctors on $\K_N(\A)$, for any additive category $\A$.
\end{thm}

\begin{lem}[{\cite[Lem.~2.6 (i)]{IKM17}}] \label{lem: Sigma-mu}
	Let $\A$ be an additive category. For all $A \in \A$, $k,s \in \ZZ$, $t \in \{1,\dots,N-1\}$, and $l \in \{0,1\}$, we have
	\[\Sigma^{2k+l}\mu^s_t(A) = \begin{cases}
		\mu^{-kN+s}_t(A), & \text{ if } l=0, \\
		\mu^{-kN+s-t}_{N-t}(A), & \text{ if } l=1.
	\end{cases}\]
\end{lem}

\begin{ntn} \label{ntn: Z-B-H}
	Given a sequence $X \in \C(\A)$ over an additive category $\A$, $r \in \NN$, and $n \in \ZZ$, set
	\begin{itemize}
			\item  $Z^n_{(r)} := Z^n_{(r)}(X) :=  \ker\left(d_X^{n+r-1} \circ \cdots \circ d_X^n\right)$,
			\item  $B^n_{(r)} := B^n_{(r)}(X) := \im\left(d_X^{n-1} \circ \cdots \circ d_X^{n-r}\right)$,
			\item  $C^n_{(r)} := C^n_{(r)}(X) := \coker\left(d_X^{n-1} \circ \cdots \circ d_X^{n-r}\right)$,
	\end{itemize}
	if the respective object exists. This is the case, for instance, if the respective $r$-fold composition is an admissible morphism, see \Cref{dfn: admissible}.  If $\A$ is an Abelian category, the homology of an $N$-complex $X \in \C_N(\A)$ is defined as $$H^n_{(r)}:=H^n_{(r)}(X):=Z^n_{(r)}(X) / B^n_{(N-r)}(X).$$
	The lower index $r=1$ is omitted if $N=2$.
\end{ntn}

\begin{rmk} \label{rmk: kernel-mu}
	For an object $A$ of an additive category $\A$,
	\[B^n_{(N-r)}(\mu_N^s(A))=Z^n_{(r)}(\mu_N^s(A)) =
	\begin{cases}
		A, & \text{ if }  s-r+1 \leq n \leq s,\\
		0, & \text{ otherwise,}
	\end{cases}\]

	\[C^n_{(r)}(\mu_N^s(A)) =
	\begin{cases}
		A, & \text{ if }  s-N+1 \leq n \leq s-N+r,\\
		0, & \text{ otherwise,}
	\end{cases}\]
for all $n,s \in \ZZ$ and $r \in \{1, \dots, N-1\}$.
\end{rmk}

\begin{rmk}[{\cite[(3.4)]{IKM17}}] \label{rmk: Hom-mu}
	For an $N$-complex $X \in \C_N(\A)$ over an additive category $\A$, and any $A \in \A$, there are the following isomorphisms of Abelian groups for all $s \in \ZZ$ and $t \in \{1,\dots,N-1\}$:
	\begin{enumerate}[leftmargin=*]
		\item \label{rmk: Hom-mu-a} $\Hom_{\C_N(\A)}(\mu_t^s(A), X)\cong Z^{s-t+1}_{(t)}(\Hom_\A(A, X))$
		\item \label{rmk: Hom-mu-b} $\Hom_{\K_N(\A)}(\mu_t^s(A), X)\cong H^{s-t+1}_{(t)}(\Hom_\A(A, X))$
		\item \label{rmk: Hom-mu-c} $\Hom_{\C_N(\A)}(X, \mu_t^s(A))\cong Z^{-s}_{(t)}(\Hom_\A(X,A))$
		\item \label{rmk: Hom-mu-d} $\Hom_{\K_N(\A)}(X, \mu_t^s(A))\cong H^{-s}_{(t)}(\Hom_\A(X,A))$
	\end{enumerate}
\end{rmk}

\begin{dfn}
	A morphism $f\colon X \to Y$ of $N$-complexes over an additive category $\A$ is called {\textbf{($\boldsymbol N$-)null-homotopic}} if there exists a collection $h=(h^k)_{k \in \ZZ}$ of morphisms $h^k\colon X^k \to Y^{k-N+1}$ in $\A$ such that
	\[f^k = \sum_{r=0}^{N-1} d_Y^{\{N-r-1\}} h^{k+r} d_X^{\{r\}} \]
	for all $k \in \ZZ$. Such an $h$ is referred to   as an \textbf{($\boldsymbol N$-)homotopy}. Two morphisms in $C_N(\A)$ are called \textbf{($\boldsymbol N$-)homotopy equivalent} if their difference is $N$-null-homotopic. The \textbf{homotopy category} of $\C_N(\A)$ has the same objects as $\C_N(\A)$ and $N$-homotopy equivalence classes as morphisms.
\end{dfn}

\begin{thm}[{\cite[Thm.~2.3]{IKM17}}] \label{thm: K-homotopy}
	The homotopy category of $\C_N(\A)$ over an additive category $\A$ is the stable category $\K_N(\A)$.
\end{thm}

\begin{dfn} \label{dfn: null-homotopic}
	An $N$-complex $X \in \C_N(\A)$ over an additive category $\A$ is called \textbf{($\boldsymbol{N}$-)null-homotopic} or \textbf{($\boldsymbol{N}$-)contractible} if $\id_{X}$ is $N$-null-homotopic, or equivalently, if $X=0$ in $\K_N(\E)$.
\end{dfn}

\subsection{Categories of monics}

Categories of morphisms have been studied by Brightbill and Miemietz {\cite[$\S$3]{BM24}}. We collect some of their results, see also {\cite[Ex.~13.12]{Buh10}}.

\begin{ntn}[{\cite[Def.~3.1]{BM24}}, {\cite[Def~4.1]{IKM17}}] \label{dfn: Mor} For an additive category $\A$ and $l \in \NN$, let $\Mor_{l}(\A)$ denote the diagram category $\textup{Func}(T_l, \A)$ where $T_l$ is the linear quiver
	\begin{center}
	 \begin{tikzcd} \underset 1 \bullet \ar[r] & \underset 2 \bullet \ar[r] & \cdots \ar[r] & \underset l \bullet \ar[r] & \underset {l+1} \bullet  \end{tikzcd}
	\end{center}
	 with $l$ arrows. Over an exact category $\E$, we denote by
	
	\begin{itemize}
		\item $\mMor_l(\E)$ the subcategory of $\Mor_l(\E)$ where all arrows map to admissible monics,
		\item $\smMor_l(\E)$ the subcategory of $\mMor_{l}(\E)$ where all arrows map to split monics,
		\item $\eMor_l(\E)$ the subcategory of $\Mor_l(\E)$ where all arrows map to admissible epics and by
		\item $\seMor_l(\E)$ the subcategory of $\eMor_{l}(\E)$ where all arrows map to split epics.
	\end{itemize}
	
\end{ntn}

\begin{rmk}[{\cite[9]{BM24}}]
	The categories $\op{\mMor_l(\E)}$ and $\eMor_{l}(\op\E)$ over an exact category $\E$ agree, see \Cref{rmk: E-op}. Therefore, each statement about $\mMor_l(\E)$ has a dual for $\eMor_{l}(\E)$. 
\end{rmk}

\begin{ntn} \label{ntn: iota} Let $\E$ be an exact category, $l \in \NN$ with $l < N$, and $n \in \ZZ$. Any object \begin{tikzcd}[cramped, sep=small] X^1 \ar[r, tail] & \cdots \ar[r, tail] & X^{l+1} \end{tikzcd}
	of $\mMor_{l}(\E)$ can be considered as a bounded complex with $X^{l+1}$ at position $n$ and zero otherwise, see \Cref{ntn: C}. Homotopies between two such complexes are zero. This gives rise to fully faithful functors
	\begin{center}
	\begin{tikzcd}
		\iota^n = \iota^n_\E\colon \mMor_{l}(\E) \ar[r] & \C_{N}(\E)
	\end{tikzcd}
	\hspace{2.5mm}
	\text{ and }
	\hspace{2.5mm}
	\begin{tikzcd}
		\iota^n = \iota^n_\E\colon \mMor_{l}(\E) \ar[r] & \K_{N}(\E).
	\end{tikzcd}
	\end{center}
\end{ntn}

\Cref{ntn: mu-mMor} is analogous to \Cref{ntn: mu}.

\begin{ntn} \label{ntn: mu-mMor}
	For an object $A$ of an exact category $\E$ and $t \in \{1, \dots, l+1\}$, we define the object
	\begin{center}
		\begin{tikzcd}
			\mu_t(A)\colon \; 0 \ar[r, tail] & \cdots \ar[r, tail] & 0 \ar[r, tail] & A^{l-t+2} \ar[r, tail, "\id_A"] & \cdots \ar[r, tail, "\id_A"] & A^{l+1}
		\end{tikzcd}
	\end{center}
	of $\smMor_{l}(\E)$ by $A^k := A$ for $k \in \{l-t+2, \dots, l+1\}$.
\end{ntn}

\begin{thm}[{\cite[Props.~3.5, 3.9, 3.11]{BM24}}] \label{thm: mMor} Let $\E$ be an exact category.
	
	\begin{enumerate}
		\item \label{thm: mMor-exact} The category $\mMor_l(\E)$ is exact with the termwise exact structure.
		
		\item \label{thm: mMor-Proj} We have $\Prj(\mMor_{l}(\E))=\smMor_l(\Prj(\E))$ and $\Inj(\mMor_{l}(\E))=\smMor_l(\Inj(\E))=\mMor_l(\Inj(\E))$.
		
		\item \label{thm: mMor-enough} If $\E$ has enough projectives resp. injectives, then so has $\mMor_l(\E)$.
	\end{enumerate}
\end{thm}

\begin{cor}[{\cite[Thm.~3.12]{BM24}}] If $\F$ is a Frobenius category, then so is $\mMor_l(\F)$, and $\Prj(\mMor_{l}(\F))=\mMor_l(\Prj(\F))$. In particular, the stable category $\underline \mMor_l(\F)$ is triangulated, see \Cref{thm: stable-Frobenius}. \qed
\end{cor}

\begin{rmk}
	In general, $\Mor_l(\E)$ is not Frobenius even if $\E$ is so, see {\cite[Prop.~3.10]{BM24}}.
\end{rmk}

\begin{lem} \label{lem: stab-mMor-embed}
	Let $\E'$ be a (fully) exact subcategory of $\E$. Then the same holds for the subcategory $\mMor_{l}(\E')$ of $\mMor_{l}(\E)$. If $\E'$ has enough $\E$-projectives, then the canonical functor $\underline \mMor_{l}(\E') \to \underline \mMor_{l}(\E)$ is fully faithful.
\end{lem}

\begin{proof} The first statement holds due to the termwise exact structure. Suppose now that $\E'$ has enough $\E$-projectives. Then $\mMor_{l}(\E')$ has enough projectives due to \Cref{thm: mMor}.\ref{thm: mMor-enough}. By \Cref{lem: sub-Frobenius}.\ref{lem: sub-Frobenius-a} and \Cref{thm: mMor}.\ref{thm: mMor-Proj}, we have
\begin{align*}
	\Prj(\mMor_{l}(\E')) = \smMor_{l}(\Prj(\E')) &= \smMor_{l}(\Prj(\E) \cap \E')\\
	&= \smMor_{l}(\Prj(\E)) \cap \mMor_{l}(\E') = \Prj(\mMor_{l}(\E)) \cap \mMor_{l}(\E').
\end{align*}
Then $\mMor_{l}(\E')$ has enough $\mMor_{l}(\E)$-projectives, see \Cref{rmk: enough-proj}, and \Cref{lem: sub-Frobenius}.\ref{lem: sub-Frobenius-a} yields the claim.
\end{proof}

\subsection{Idempotent complete categories}

\begin{dfn}[{\cite[Def.~6.1]{Buh10}}] \label{dfn: idem-compl} An \textbf{idempotent} $e\colon A \to A$ in an additive category $\A$ is an endomorphism with $e^2 = e$. It is called \textbf{split} if there is a biproduct decomposition $A \cong eA \oplus (1-e)A$ of $A$ into objects $eA, (1-e)A \in \A$ such that $e \cong \begin{pmatrix}
		\id_{eA} & 0 \\
		0 & 0
	\end{pmatrix}$ with respect to this decomposition.
	
	\smallskip
	
	The category $\A$ is called \textbf{idempotent complete} if every idempotent splits or, equivalently, if every idempotent has a kernel, see {\cite[Rem.~6.2]{Buh10}}.
\end{dfn}

\begin{rmk} \label{rmk: eA}
	Let $s\colon B \to A$ and $r\colon A \to B$ be morphisms in an additive category $\A$ with $rs = \id_B$. If the idempotent $e:= sr\colon A \to A$ splits, then $s$ induces an isomorphism $B \cong eA$, whose inverse is $r$ restricted to the direct summand $eA$ of $A$.
\end{rmk}

\begin{rmk} \label{rmk: biproduct-C}
	The conditions on the morphisms defining a biproduct of sequences over an additive category are termwise.
\end{rmk}

\begin{lem} \label{rmk: idem-morph}
	Let $a\colon A' \to A$ be a morphism in an additive category $\A$. Let $(e', e)\colon a \to a$ be an idempotent in the morphism category of $\A$. If both $e'$ and $e$ are split, then so is $(e', e)$. That is, there are unique morphisms 
	$a' = (e', e)a\colon e'A' \to eA$ and $a''=(1-(e',e))a\colon (1-e')A' \to (1-e)A$ forming a biproduct $a \cong a' \oplus a''$ given termwise by $A' \cong e'A' \oplus (1-e') A'$ and $A \cong eA \oplus (1-e) A$.
\end{lem}

\begin{proof} Due to \Cref{rmk: biproduct-C}, it suffices to construct $a'$ and $a''$ to fit in the commutative diagrams
	\[
		\begin{tikzcd}[column sep={25mm,between origins}, row sep={20mm,between origins}]
			e'A' \ar[r, "a'", dashed] \ar[d, tail, "j'"] & eA \ar[d, tail, "j"] \\
			A' \ar[r, "a"] \ar[d, two heads, "q'"] & A \ar[d, two heads, "q"] \\
			(1-e')A' \ar[r, "a''", dashed] & (1-e)A
		\end{tikzcd}
		\; \text{ and } \;
		\begin{tikzcd}[column sep={25mm,between origins}, row sep={20mm,between origins}]
			e'A' \ar[r, "a'", dashed]  & eA  \\
			A' \ar[r, "a"] \ar[u, two heads, "r'"'] & A \ar[u, two heads, "r"'] \\
			(1-e')A' \ar[r, "a''", dashed] \ar[u, tail, "s'"']& (1-e)A, \ar[u, tail, "s"']
		\end{tikzcd}
	\]
	where the vertical morphisms define the respective splittings of $e'$ and $e$. Using $j'r' = e'$ and $jr = e$  we compute $aj'r' = ae' = ea = jra$. Then $qaj'r' = qjra =0$ as $qj=0$, and hence $qaj' = 0$ since $r'$ is an epic. As $j$ is a kernel of $q$  and $q'$ a cokernel of $j'$, the two dashed morphisms in the left diagram exist. Then $ja'r' = aj'r' = jra$ and hence $a'r' = ra$ since $j$ is a monic. So, the upper right square commutes, and analogously the lower one.
\end{proof}

\begin{rmk} \label{rmk: idem-induce} If $(e', e)$ is an idempotent of a morphism, then  $e'$ induces an idempotent of the kernel and $e$ of the cokernel, if the respective object exists.
\end{rmk}

\begin{lem} \label{lem: idem-ses}
	Consider a short exact sequence \begin{tikzcd}[cramped] A' \ar[r, tail, "i"] & A \ar[r, two heads, "p"] & A'' \end{tikzcd} in an exact idempotent complete category $\E$ and an idempotent $(e', e)\colon i \to i$ in the morphism category of $\E$. Then there is an idempotent $(e, e'')\colon p \to p$. The corresponding morphisms from \Cref{rmk: idem-morph} form a biproduct $(i, p) \cong (i', p') \oplus (i'', p'')$ of short exact sequences displayed as follows:
	\begin{center}
		\begin{tikzcd}[column sep={25mm,between origins}, row sep={20mm,between origins}]
			e'A' \ar[r, tail, dashed, "i'"] \ar[d, tail, "j'"] & eA \ar[r, two heads, dashed, "p'"] \ar[d, tail, "j"] & e''A'' \ar[d, tail, dashed, "j''"] \\
			A' \ar[r, tail, "i"] \ar[d, two heads, "q'"] & A \ar[r, two heads,  "p"] \ar[d, two heads, "q"] & A'' \ar[d, two heads, dashed, "q''"] \\
			(1-e')A' \ar[r, tail, dashed, "i''"] & (1-e)A \ar[r, two heads, dashed, "p''"] & (1-e'') A''
		\end{tikzcd}
	\end{center}
	In particular, applying idempotents preserves admissible monics.
\end{lem}

\begin{proof}
	Clearly, $e''\colon A'' \to A''$ can be chosen as the induced morphism between cokernels, see \Cref{rmk: idem-induce}. Fix a choice of splittings of the idempotents $e'$, $e$ and $e''$.
	\Cref{rmk: idem-morph}, applied to the idempotent $(e', e)$ of $i$ and the idempotent $(e, e'')$ of $p$, yields the dashed morphisms and the desired biproduct, see \Cref{rmk: biproduct-C}. The sequences $(i', p')$ and $(i'', p'')$ are short exact by \Cref{cor: ses-summand}.
\end{proof}

\begin{prp}\label{prp: idemp-compl-C-mMor} \
	\begin{enumerate}
		\item \label{prp: idemp-compl-C} If $\A$ is an idempotent complete category, then so are $\C(\A)$ and $\C_N(\A)$.
		
		\item \label{prp: idemp-compl-mMor} If $\E$ is an exact idempotent complete category, then so is $\mMor_l(\E)$.
	\end{enumerate}
\end{prp}

\begin{proof} \
	\begin{enumerate}[leftmargin=*]
		\item This follows from  \Cref{rmk: idem-morph} and \Cref{rmk: C_N-subcat,rmk: biproduct-C}.
		
		\item This follows from \Cref{lem: idem-ses}. \qedhere
	\end{enumerate}
\end{proof}

\subsection{Semiorthogonal decompositions}

In this subsection, we review definitions and results on the topic of \emph{semiorthogonal decompositions} of triangulated categories.

\begin{ntn}
	Consider a triangulated category $\T$ with suspension functor $\Sigma$ and subcategories $\U, \V$.
	Then $\U \ast \V$ denotes the subcategory of $\T$ consisting of objects $T \in \T$ which fit into a distinguished triangle $U \to T \to V \to \Sigma U$ where $U \in \U$ and $V \in \V$.
\end{ntn}

\begin{dfn} \label{dfn: sod}
	A pair $(\U, \V)$ of triangulated subcategories of a triangulated category $\T$ is called a \textbf{semiorthogonal decomposition} of $\T$ if
	\begin{enumerate}
		\item \label{dfn: sod-1} $\Hom_\T(\U, \V) = 0$ and
		\item \label{dfn: sod-2} $\T = \U \ast \V$.
	\end{enumerate}
\end{dfn}

\begin{prp}[{\cite[Prop.~1.2]{IKM11}}]  \label{prp: sod}
	Let $(\U, \V)$ be a semiorthogonal decomposition of a triangulated category $\T$. Then the inclusion functors $i_!\colon \U \to \T$ and $j_\ast\colon \V \to \T$ have a respective right adjoint $i^!\colon \T \to \U$ and left adjoint $j^\ast\colon \T \to \V$. These induce triangle equivalences $\T/\V \simeq \U$ and $\T/\U \simeq \V$. \qed
\end{prp}

\begin{rmk} \label{rmk: cohom}
	Let $\T$ be a triangulated category with suspension functor $\Sigma$. The covariant (contravariant) Hom-functor is (co)homological: Given a distinguished triangle
	\begin{center}
		\begin{tikzcd}
			Y \ar[r, "f"] & Y' \ar[r, "g"] & Y'' \ar[r] & \Sigma Y,
		\end{tikzcd}
	\end{center}
	each of the following two solid  commutative diagrams can be completed by a dashed arrow:
	\begin{center}
		\begin{tikzcd}[sep={22.5mm,between origins}]
			& X  \ar[ld, dashed] \ar[rd, "0"] \ar[d, "f'"'] &&& X & \\
			Y \ar[r, "f"]  & Y' \ar[r, "g"]  & Y'' & Y \ar[r, "f"] \ar[ur, "0"] & Y' \ar[r, "g"] \ar[u, "g'"'] & Y'' \ar[ul, dashed]
		\end{tikzcd}
	\end{center}
\end{rmk}

\begin{rmk}\label{rmk: Bondal-sod}
	The adjoints in \Cref{prp: sod} can be made explicit as follows, see {\cite[Proof of Lem.~3.1.(b) $\Rightarrow$ (c)]{Bon90}}:
	\begin{itemize}
		\item For $T \in \T = \U \ast \V$, fixing a distinguished triangle 
		\begin{center}
			\begin{tikzcd}
				i^!T \ar[r, "u_T"] & T \ar[r] & V_T \ar[r] & \Sigma i^!T
			\end{tikzcd}
		\end{center}
		in $\T$ with $V_T \in \V$ defines $i^!T \in \U$ up to isomorphism.
		\item Let $f\colon T \to T'$ be a morphism in $\T$. Using $\Hom_\T(\U, \V) = 0$ and \Cref{rmk: cohom}, $i^!f$ is uniquely determined by the commutative diagram
		\begin{center}
			\begin{tikzcd}[column sep={22.5mm,between origins}, row sep={20mm,between origins}]
				& i^!T \ar[d, "u_T"] \ar[rdd, "0"] \ar[ldd, dashed, "i^!f"'] & \\
				& T \ar[d, "f"] \\
				i^!T' \ar[r, "u_{T'}"] & T' \ar[r] & V_{T'}.
			\end{tikzcd}
		\end{center}
	\end{itemize}
	The functor $j^\ast$ is given dually.
\end{rmk}

\begin{thm}[{\cite[Lem.~1.1, Thm.~B]{JK15}}] \label{thm: sod} Consider triangulated subcategories $\U, \V$ of a triangulated category $\T$. Then the following conditions are equivalent:
	\begin{enumerate}[label=(\arabic*)]
		\item \label{thm: sod-1} $\V \ast \U \subseteq \U \ast \V$.
		\item \label{thm: sod-2} $\U \ast \V$ is a triangulated subcategory of $\T$.
		\item \label{thm: sod-3} Any morphism in $\Hom_\T(\U, \V)$ factors through an object of $\U \cap \V$.
	\end{enumerate}
	In this case, the following statements hold:
	\pushQED{\qed}
	\begin{enumerate}
		\item \label{thm: sod-a} The pair $(\U/(\U \cap \V), \V/(\U \cap \V))$ is a semiorthogonal decomposition of $(\U \ast \V)/(\U \cap \V)$.
		
		\item \label{thm: sod-b} There are triangle equivalences $\U / (\U \cap \V) \simeq (\U \ast \V)/\V$ and $\V / (\U \cap \V) \simeq (\U \ast \V)/\U$.
		
		\item \label{thm: sod-c} The canonical functors $\U/(\U \cap \V) \to \T/\V$ and $\V/(\U \cap \V) \to \T/\U$ are fully faithful. \qedhere
	\end{enumerate}
\end{thm}

\section{$N$-acyclicity} \label{section: N-acyclicity}

In this section, we define (total) $N$-acyclicity of sequences over an exact category $\E$ (\Cref{subsection: contraction}). We show that it is preserved locally under extensions, cones and suspensions (\Cref{subsection: cones}). This leads to various triangulated subcategories of $\K_N(\E)$ known from the classical case. We describe acyclic $N$-complexes over $\E$ equivalently in terms of so-called \emph{acyclic $N$-arrays} (\Cref{subsection: arrays}). To resolve bounded above $N$-complexes over $\E$, we extend a construction of Keller from the case $N=2$ (\Cref{subsection: resolutions}). If $\E$ is Frobenius, we relate complete resolutions with their $N$-syzygies (\Cref{subsection: N-syzygies}).

\subsection{Contraction and acyclicity} \label{subsection: contraction} In this subsection, we define the concept of $N$-acyclicity by reduction to the special case $N=2$, see \Cref{dfn: 2-acyclic}.

\begin{dfn} \label{dfn: gamma}
	Let $\A$ be an additive category. Given $N \in \NN_{\geq 2}$, $n \in \ZZ$, and $r \in \{1,\dots,N-1\}$, the \textbf{contraction functor} $\gamma^n_{r, N}\colon \C(\A) \to \C(\A)$ is defined by sending $X \in \C(\A)$ to the sequence
	\begin{center}
		\begin{tikzcd}[sep=8mm]
			\cdots \ar[r]  &X^{n-N} \ar[r, "d_X^{\{r\}}"] &X^{n-N+r} \ar[r, "d_X^{\{N-r\}}"] &X^n \ar[r, "d_X^{\{r\}}"] & X^{n+r} \ar[r, "d_X^{\{N-r\}}"] & X^{n+N} \ar[r, "d_X^{\{r\}}"] & X^{n+N+r} \ar[r] & \cdots 
		\end{tikzcd}
	\end{center}
	with $X^n$ at position $n$, see \Cref{ntn: d^r}. A morphism $(f^k)_{k \in \ZZ}$ in $\C(\A)$ is mapped to \[(\dots, f^{n-N+r}, f^n, f^{n+r}, f^{n+N}, \dots).\]
	 By definition, $\gamma^n_{r, N}$ restricts to a functor  $\gamma^n_r\colon \C_N(\A) \to \C_2(\A)$. If $\E$ is an exact category,  $\gamma^n_{r, N}\colon \C(\E) \to \C(\E)$ is exact with respect to the termwise exact structure, see \Cref{rmk: C-exact}.
\end{dfn}

\begin{dfn} \label{dfn: N-acyclic}
	Let $\E$ be an exact category and $N\in \NN_{\geq 2}$. We call a sequence $X \in \C(\E)$ \textbf{$\boldsymbol N$-acyclic at position $ n \in \ZZ$} if $\gamma^n_r(X)$ is 2-acyclic at position $n$ for all $r \in \{1,\dots,N-1\}$, see \Cref{dfn: 2-acyclic}.\ref{dfn: 2-acyclic-global}. We call it \textbf{totally $\boldsymbol N$-acyclic at position $ n \in \ZZ$}  if, in addition, $\Hom_\E(X, P) \in \C_N(\Ab)$ is $N$-acyclic at position $- n$ for all $P \in \Prj(\E)$, see \Cref{ntn: Hom}. We say that $X$ is \textbf{(totally) $\boldsymbol N$-acyclic} if it is (totally) $N$-acyclic at all positions $n \in \ZZ$.\\
	The subcategory of $\C(\E)$ consisting of $N$-acyclic sequences is denoted by $\C_N^{\infty, \varnothing}(\E)$, its subcategory of totally $N$-acyclic sequences by $\C_N^{\infty, \varnothing^\ast}(\E)$. This notation becomes clear in \Cref{ntn: Verdier-notation}. Note that $\C_N^{\infty, \varnothing}(\E)$ is a subcategory of $\C_N(\E)$. For this reason we also use the term \textbf{(totally) acyclic $\boldsymbol N$-complex}.
\end{dfn}

\begin{rmk} \label{rmk: Frobenius-tac}
	If $\F$ is a Frobenius category, then $\C^{\infty, \varnothing}_N(\F) = \C^{\infty, \varnothing^\ast}_N(\F)$ as any $P \in \Prj(\F) = \Inj(\F)$ makes $\Hom_\F(-, P)$ an exact functor.
\end{rmk}

\begin{rmk} \label{rmk: gamma-Hom} \
	\begin{enumerate}[leftmargin=*]
		\item \label{rmk: gamma-Hom-equality} Over an additive category $\A$, we have $\Hom_\A(A, -) \circ \gamma^n_r = \gamma^n_r \circ \Hom_\A(A, -)$ and $\Hom_\A(-, A) \circ \gamma^n_r = \gamma^{-n}_{N-r} \circ \Hom_\A(-, A)$ as functors on $\C(\A)$, for all $A \in \A$, $n \in \ZZ$, and $r \in \{1, \dots, N-1\}$.
		
		\item \label{rmk: gamma-Hom-tac} In particular, an $N$-complex $X \in \C_N(\E)$ over an exact category $\E$ is totally $N$-acyclic at position $n \in \ZZ$ if and only if $\gamma^n_r(X)$ is totally 2-acyclic at $n$, for all $r \in \{1, \dots, N-1\}$.
	\end{enumerate}
\end{rmk}

\begin{rmk} \label{rmk: gamma-homology} Let $\E$ be an exact category, $n \in \ZZ$, and $r \in \{1,\dots,N-1\}$.
	Note that $Z^n_{(r)}(X) = Z^n (\gamma^n_r(X))$ and $B^n_{(r)}(X) = B^n(\gamma^n_{N-r}(X))$ for any sequence $X \in \C(\E)$ if the respective objects exist, see \Cref{ntn: Z-B-H}. In this case, $X$ is $N$-acyclic at position $n$ if and only if $Z^n_{(r)}(X) = B^n_{(N-r)}(X)$ for all $r$.
	In particular, $H^n_{(r)} = H^n \circ \gamma^n_r$ if $\E$ is Abelian and $X\in \C_N(\E)$ an $N$-complex. In this case, $X$ is $N$-acyclic at position $n$ if and only if $H^n_{(r)}(X) =0$ for all $r$.
\end{rmk}

\begin{lem} \label{lem: Zn-exact}
	Consider a term\emph{wise} short exact sequence	\begin{tikzcd}[cramped]
		X \ar[r, tail] & Y \ar[r, two heads] & Z
	\end{tikzcd} of $N$-complexes over an exact category $\E$. If $X$, $Y$, and $Z$ are acyclic at position $n \in \ZZ$, then the induced sequence
	\[\begin{tikzcd}
		Z^n_{(r)}(X) \ar[r, tail] & Z^n_{(r)}(Y) \ar[r, two heads] & Z^n_{(r)}(Z)
	\end{tikzcd}\]
	is short exact for all $r \in \{1, \dots, N-1\}$. The verbatim statement holds for cokernels instead of kernels.
\end{lem}

\begin{proof} Due to \Cref{rmk: gamma-homology}, we may assume that $N=2$. There is a commutative diagram
	\begin{center}
		\begin{tikzcd}[sep={20mm,between origins}]
			X^{n-1} \ar[r, two heads] \ar[d, tail] & Z^n(X) \ar[r, "x'", tail] \ar[d, dashed, "z'"', "y'"] & X^n \ar[d, "v'", tail] \ar[r, "\circ" marking] \ar[rd, "x"] & X^{n+1} \ar[d, tail] \\
			Y^{n-1} \ar[r, two heads, "p"] \ar[d, two heads,] & Z^n(Y) \ar[r, "w'", tail] \ar[d, dashed, "z"'] \ar[rd, "y"] & Y^n \ar[d, "v", two heads] \ar[r, "w", "\circ" marking] & Y^{n+1} \ar[d, two heads] \\
			Z^{n-1} \ar[r, two heads] & Z^n(Z) \ar[r, tail] & Z^n \ar[r, "\circ" marking] & Z^{n+1}, 
		\end{tikzcd}
	\end{center}

	where $z$ and $z'$ are induced on the respective kernels, see \Cref{rmk: help-commute}.\ref{rmk: help-commute-mono}. By \Cref{lem: prep-Zn}, $z'$ is a kernel of $z$. Since $zp$ is an admissible epic, so is $z$ due to \Cref{prp: obscure}.\ref{prp: obscure-epi}. It follows that
	\begin{center}
		\begin{tikzcd}[sep=large]
			Z^n(X) \ar[r, tail, "z'"] & Z^n(Y) \ar[r, two heads, "z"] & Z^n(Z)
		\end{tikzcd}
	\end{center}
	is a short exact sequence in $\E$ as desired.
\end{proof}

\begin{rmk} \label{rmk: help-commute}
	Consider a not necessarily commutative diagram
	
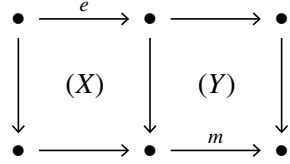
\begin{center}
		\begin{tikzcd}[sep={17.5mm,between origins}]
			\bullet \ar[r, "e"] \ar[d] \ar[rd, phantom, "(X)"] & \bullet \ar[r] \ar[d] \ar[rd, phantom, "(Y)"] & \bullet \ar[d] \\
			\bullet \ar[r] & \bullet \ar[r, "m"] & \bullet
		\end{tikzcd}
	\end{center}
	in an arbitrary category, and suppose that the outer rectangle $(XY)$ commutes. Then:
	\begin{enumerate}[leftmargin=*]
		\item \label{rmk: help-commute-epi} If $e$ is an epic and $(X)$ commutes, then $(Y)$ commutes as well.
		\item \label{rmk: help-commute-mono} If $m$ is a monic and $(Y)$ commutes, then $(X)$ commutes as well.
	\end{enumerate}
\end{rmk}

\begin{lem} \label{lem: prep-Zn}
	In a category with zero morphisms, consider a commutative diagram
	
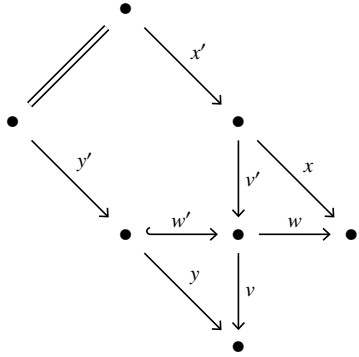
\[
	\begin{tikzcd}[sep={15mm,between origins}]
		& \bullet \ar[rd, "x'"] && \\
		\bullet \ar[ru, equal] \ar[rd, "y'"] && \bullet \ar[rd, "x"] \ar[d, "v'"] \\
		& \bullet \ar[rd, "y"] \ar[r, hook, "w'"] & \bullet \ar[d, "v"] \ar[r, "w"] & \bullet \\
		&& \bullet
	\end{tikzcd}
	\]
	with $ww'=0$ and $x'$ and $v'$ kernels of $x$ and $v$, respectively. Then $y'$ is a kernel of $y$.
\end{lem}

\begin{proof}
	First note that $yy' = vv'x'=0$ since $v'$ is a kernel of $v$. To prove universality, let $a$ be a morphism with $ya = 0$. Then $0=vw'a$, and hence $w'a = v'b$ for a unique $b$ since $v'$ is a kernel of $v$. We compute $xb = wv'b=ww'a = 0$ using $ww'=0$. Hence, $b=x'c$ for a unique $c$ since $x'$ is a kernel of $x$.
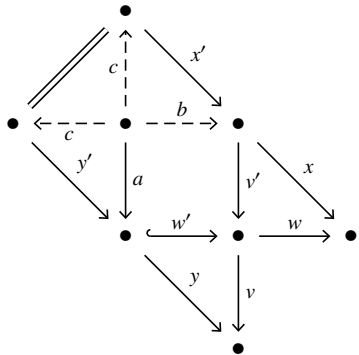
	\[
	\begin{tikzcd}[sep={15mm,between origins}]
		& \bullet \ar[rd, "x'"] && \\
		\bullet \ar[ru, equal] \ar[rd, "y'"] & \bullet \ar[d, "a"] \ar[r, dashed, "b"] \ar[u, dashed, "c"] \ar[l, dashed, "c"] & \bullet \ar[rd, "x"] \ar[d, "v'"] \\
		& \bullet \ar[rd, "y"] \ar[r, hook, "w'"] & \bullet \ar[d, "v"] \ar[r, "w"] & \bullet \\
		&& \bullet
	\end{tikzcd}
	\]
	We show that $c$ is the unique morphism with $y'c=a$ as required: Note that $w'y'c = v'x'c = v'b = w'a$, which implies $y'c=a$ since $w'$ is monic. If $c'$ is an arbitrary morphism with $y'c'=a$, then $v'x'c'=w'y'c' = w'a$. Thus, $x'c'=b$, and hence $c'=c$ by uniqueness of $b$ and $c$, respectively.
\end{proof}

\begin{ntn} \label{ntn: gamma-indices} For $N \in \NN_{\geq 2}$, $n \in \ZZ$ and, $r \in \{1,\dots,N-1\}$, consider the set \[\Z = \Z^n_r:= \{n+aN+br \, \vert \, a \in \ZZ, b \in \{0,1\}\} \subseteq \ZZ\]  of indices selected by the contraction functor $\gamma := \gamma^n_r$. Note that for any $s \in \ZZ$ the interval $[s-N+1,s]$ contains exactly two elements of $\Z$. Denote by $\gamma'\colon \Z \to \ZZ$, $n+aN+br \mapsto n+2a+b$, the reindexing  realized by $\gamma$. By abuse of notation, we also write
	\[\gamma = \gamma^n_r\colon \ZZ \longrightarrow \ZZ, \, s \longmapsto \gamma'(\max\{k\in \Z \, \vert \, k \leq s\}),\] 
	which extends $\gamma'$ to $\ZZ$.
\end{ntn}

\begin{rmk} \label{rmk: mu-acyclic}
	Let $\E$ be an exact category, $A \in \E$, and $s \in \ZZ$.
	\begin{enumerate}[leftmargin=*]
		\item \label{rmk: mu-acyclic-formula} Using \Cref{ntn: gamma-indices}, we have $\gamma^n_r(\mu^s_N(A)) = \mu^{\gamma^n_r(s)}_2(A)$ for all $n \in \ZZ$ and $r \in \{1, \dots, N-1\}$.
		
		\item \label{rmk: mu-acyclic-statement} Suppose that $A \neq 0$. Then the sequence $\mu^s_t(A)$ is (totally) $N$-acyclic if and only if $t = N$: Indeed, $\gamma^n_r(\mu^s_N(A))$ is 2-acyclic for all $n \in \ZZ$ and $r \in \{1, \dots, N-1\}$ due to \ref{rmk: mu-acyclic-formula}. Then also $\Hom_\E(\mu^s_N(A), P) = \mu^{-s+N-1}_N(\Hom_\E(A,P))$ is $N$-acyclic for any $P \in \Prj(\E)$, see \Cref{rmk: mu-dual}. Thus, $\mu^s_N(A)$ is even totally $N$-acyclic. If $t \in \{1, \dots, N-1\}$, then the $2$-complex $\gamma^s_1(\mu^s_t(A))=\mu^s_1(A)$ is clearly not 2-acyclic. So, $\mu^s_N(A)$ is not $N$-acyclic.
	\end{enumerate}
\end{rmk}

\subsection{Cones and extensions} \label{subsection: cones} In this subsection, we describe to which extent the cone $C(f)$ and the cocone $C^\ast(f)$, see \Cref{con: cone}, preserve local $N$-acyclicity of the source and target of $f$ (\Cref{prp: cone-acyclic}). This specializes to the suspension functor $\Sigma$ and its quasi-inverse $\Sigma^{-1}$ (\Cref{cor: Sigma-acyclic}). We examine the preservation of local $N$-acyclicity is preserved under extensions (\Cref{prp: acyclic-ext-closed}). Imposing boundedness and acyclicity conditions, defines various extension-closed Frobenius subcategories of $\C_N(\E)$. Their stable categories are related in a diagram of triangulated subcategories of $\K_N(\E)$ (\Cref{thm: Verdier-notation}).

\begin{lem}[{\cite[Lem.~10.3]{Buh10}}] \label{lem: Buehler10.3}
	Let $f\colon X \to Y$ be a morphism of 2-complexes over an exact category $\E$.  If $X$ and $Y$ are acyclic at positions $n, n+1 \in \ZZ$, then the cone $C(f)$ of $f$ is acyclic at position $n$. 
	Dually, if $X$ and $Y$ are acyclic at positions $n-1, n \in \ZZ$, then the cocone $C^\ast(f)$ of $f$ is acyclic at position $n$.
\end{lem}

\begin{proof}
	Although Bühler's Lemma concerns global acyclicity, his arguments prove our local claim. The dual statement then follows by \Cref{con: Sigma-explicit}.\ref{con: Sigma-explicit-theta}.
\end{proof}

\begin{rmk} \label{rmk: gamma-acyclicity}
	If a sequence $X \in \C(\E)$ over an exact category $\E$ is $N$-acyclic at positions $n, n+1, \dots, n+N-1 \in \ZZ$, resp. $n-N+1, \dots, n-1, n \in \ZZ$, then $\gamma^n_r(X)$ is 2-acyclic at positions $n, n+1$, resp. $n-1, n$, for all $r \in \{1, \dots, N-1\}$. The verbatim statement for total $N$-acyclicity follows due to \Cref{rmk: gamma-Hom}.\ref{rmk: gamma-Hom-tac}.
\end{rmk}

\begin{prp} \label{prp: acyclic-ext-closed}
	Let $X \rightarrowtail Y \twoheadrightarrow Z$ be a termsplit short exact sequence of $N$-complexes $X, Y, Z \in \C_N(\E)$ over an exact category $\E$.
	\begin{enumerate}
		\item \label{prp: acyclic-ext-closed-ac} If $X$ is acyclic at positions $n, n+1, \dots, n+N-1 \in \ZZ$ and and $Z$ is acyclic at positions $n-N+1, \dots, n-1, n$, then $Y$ is acyclic at position $n$.
		
		\item \label{prp: acyclic-ext-closed-tac} If $X$ and $Z$ are totally acyclic at positions $n-N+1, \dots, n, \dots, n+N-1 \in \ZZ$, then $Y$ is totally acyclic at position $n$.
	\end{enumerate}
	
	In particular, the subcategories $\C^{\infty, \varnothing}_N(\E)$ and $\C^{\infty, \varnothing^\ast}_N(\E)$ of (totally) acyclic $N$-complexes are both extension-closed in $\C_N(\E)$.
	
\end{prp}

\begin{proof} \
	\begin{enumerate}[leftmargin=*]
		\item By \Cref{rmk: gamma-acyclicity}, the claim reduces to the case $N=2$. By hypothesis, there are decompositions $Y^k \cong X^k \oplus Z^k$ for all $k$. We can thus write the $k$th morphism in $Y$ as
		\[d_Y^k = \begin{pmatrix}
			d_X^k & f^k \\
			0 & d_Z^k
		\end{pmatrix}\]
		\noindent
		for some morphism $f^k\colon Z^k \to X^{k+1}$ in $\E$. Since $Y$ is a complex, we have
		\[0=d_Y^2 = \begin{pmatrix}
			d_X^2 & d_X f + f d_Z \\
			0 & d_Z^2
		\end{pmatrix}.\]
		\noindent
		Consequently, $f \circ (-d_Z)  = d_X \circ f$ and $\Theta^{-1}f\colon - \Theta^{-1} Z \to  X$ is a morphism in $\C_N(\E)$ with
		\[d^k_{C(\Theta^{-1}f)} = \begin{pmatrix} d^k_X & (\Theta^{-1} f)^{k+1} \\ 0 & -d_{- \Theta^{-1} Z}^{k+1} \end{pmatrix} = \begin{pmatrix}
			d_X^k & f^k \\
			0 & d_Z^k
		\end{pmatrix} = d_Y^k. \]
		\noindent
		Since $-\Theta^{-1}Z$ and $X$ are acyclic at positions $n, n+1$ by hypothesis, $Y \cong C(\Theta^{-1}f)$ is acyclic at position $n$ due to \Cref{lem: Buehler10.3}.
		
		\smallskip
		
		\item This follows from \ref{prp: acyclic-ext-closed-ac} applied to the termsplit short exact sequence
		\[
			\begin{tikzcd}[sep=large]
				\Hom_\E(Z, P) \ar[r, tail] & \Hom_\E(Y, P) \ar[r, two heads] & \Hom_\E(X, P)
			\end{tikzcd}
		\]
		for each $P \in \Prj(\E)$. \qedhere
	\end{enumerate}
\end{proof}

\begin{lem}\label{lem: gauss}
	If \begin{tikzcd}[cramped] A' \ar[r, tail, "a'"] & A \ar[r, two heads, "a"] & A'' \end{tikzcd} is a (split) short exact sequence in an exact category $\E$, then so are the following sequences:
	
	\begin{enumerate}
		\item \label{lem: gauss-a} \begin{tikzcd}[column sep={25mm,between origins}, row sep={20mm,between origins}, ampersand replacement = \&]
			A' \oplus B \ar{r}{\begin{pmatrix} a' & b \\ 0 & \id_B \end{pmatrix}}  \& A \oplus B  \ar{r}{\begin{pmatrix} a & -ab   \end{pmatrix}}  \& A'' 
		\end{tikzcd}
		for any morphism $b\colon B \to A$ in $\E$, and
		\item \label{lem: gauss-b}	\begin{tikzcd}[column sep={25mm,between origins}, row sep={20mm,between origins}, ampersand replacement = \&]
			A'  \ar{r}{\begin{pmatrix} a' \\ -ca' \end{pmatrix}}  \& A \oplus C \ar{r}{\begin{pmatrix} a & 0 \\  c & \id_C \end{pmatrix}}   \& A'' \oplus C 
		\end{tikzcd}
		for any morphism $c\colon A \to C$ in $\E$.
	\end{enumerate}

\end{lem}

\begin{proof} The following isomorphisms of short exact sequences in $\E$ prove the claims:
	\begin{enumerate}[leftmargin=*]
		\item 
		\[
		\begin{tikzcd}[column sep={25mm,between origins}, row sep={20mm,between origins}, ampersand replacement = \&]
			A' \oplus B \ar{r}{\begin{pmatrix} a' & b \\ 0 & \id_B \end{pmatrix}} \ar[d, "\id_{A' \oplus B}"] \& A \oplus B  \ar{r}{\begin{pmatrix} a & -ab   \end{pmatrix}} \ar[d, color=white, "\textcolor{black}{\cong}"'] \ar{d}{\begin{pmatrix} \id_A & -b  \\ 0 & \id_B  \end{pmatrix}} \& A''  \ar[d, "\id_{A''}"] \\
			A' \oplus B  \& A \oplus B  \ar[<-]{l}{\begin{pmatrix} a' & 0 \\ 0 & \id_B \end{pmatrix}} \& A''  \ar[<-]{l}{\begin{pmatrix} a & 0 \end{pmatrix}}
		\end{tikzcd}
		\]
		
		\item 
		\[
		\begin{tikzcd}[column sep={25mm,between origins}, row sep={20mm,between origins}, ampersand replacement = \&]
			A'  \ar{r}{\begin{pmatrix} a' \\ -ca' \end{pmatrix}} \ar[d, "\id_{A'}"] \& A \oplus C \ar{r}{\begin{pmatrix} a & 0 \\  c & \id_C \end{pmatrix}} \ar[d, color=white, "\textcolor{black}{\cong}"'] \ar{d}{\begin{pmatrix} \id_A & 0 \\ c & \id_C \end{pmatrix}} \& A'' \oplus C \ar[d, "\id_{A'' \oplus C}"] \\
			A' \& A  \oplus C \ar[<-]{l}{\begin{pmatrix} a' \\ 0\end{pmatrix}} \& A'' \oplus C \ar[<-]{l}{\begin{pmatrix} a & 0  \\ 0 & \id_C \end{pmatrix}}
		\end{tikzcd}
		\]
	\end{enumerate}
\end{proof}

\begin{con} \label{con: cone-gamma}
	Let $f\colon X \to Y$ be a morphism of $N$-complexes over an exact category $\E$. Using $\gamma = \gamma^n_r$ and $\Z = \Z^n_r$ from \Cref{ntn: gamma-indices}, we have, see \Cref{con: I-and-P} and \Cref{rmk: mu-acyclic}.\ref{rmk: mu-acyclic-formula},
	\[ \gamma I_N(X) = \bigoplus_{k \in \ZZ} \gamma (\mu^k_N(X^k)) = \bigoplus_{k \in \ZZ} \mu^{\gamma(k)}_2(X^k),\]
	 which is a direct sum of $I_2(\gamma X) = \bigoplus_{k \in \Z} \mu^{\gamma (k)}_2(X^k)$ and  $C = \bigoplus_{k \in \ZZ \backslash \Z} \mu^{\gamma(k)}_2(X^k)$. This leads to a split short exact sequence 
	\begin{center}
		\begin{tikzcd}[sep=large]
			I_2(\gamma X) \ar[r, tail, "j_X"] & \gamma I_N(X) \ar[r, two heads, "q_X"]& C
		\end{tikzcd}
	\end{center}
	with $C \in \Prj(\C_N(\E)) \cap \C^{\infty, \varnothing}_N(\E)$, see \Cref{lem: mu-ProjInj} and \Cref{rmk: mu-acyclic}.\ref{rmk: mu-acyclic-statement}. This persists after twisting the maps by $ D^k := \left( \begin{array}{c}
		d_X \\
		\vdots \\
		d_X^{\{k\}} 
	\end{array} \right)$ as follows:
	\[ j_X = \begin{pmatrix}
		1 & 0 \\
		D^{r-1} & 0 \\
		0 & 1 \\
		0 & D^{N-r-1}
	\end{pmatrix}  \hspace{5mm} \text{ and } \hspace{5mm} q_X = \begin{pmatrix}
	-D^{r-1} & E_{r-1} & 0 & \\
	0 & 0 & -D^{N-r-1} & E_{N-r-1}
\end{pmatrix}.\]
	 Indeed, permuting the direct summands yields
	 \[(j_X, p_X) \cong \left( \begin{pmatrix}
	 	D \\ E_2
	 \end{pmatrix}, \begin{pmatrix}
	 E_{N-2} & -D
 \end{pmatrix} \right), \text{ where } D := \begin{pmatrix} D^{r-1} & 0 \\ 0 & D^{N-r-1} \end{pmatrix}.\]

	This is a split short exact sequence due to \Cref{lem: gauss}.\ref{lem: gauss-a} applied to $a' = 0$, $b=D$ and $a = E_{N-2}$. With $i_{\gamma X} = \begin{pmatrix} 1 \\ d_X^{\{r\}} \end{pmatrix}$ and $\gamma i_X = \begin{pmatrix}
		1 \\ D^{N-1}
	\end{pmatrix}$, see \Cref{con: cone,con: I-and-P}, we obtain
	\[j_X \circ i_{\gamma X} = \gamma i_{X}.\]
	We now apply the exact functor $\gamma$ to the sequence \eqref{eqn: std-sequence} and relate the result with $(S(\gamma f))$, see \Cref{con: cone,con: I-and-P}. By \Cref{lem: Buehler3.7}, the commutative square given by the preceding formula can be extended to the following commutative diagram in $\C_N(\E)$ with termsplit short exact rows and split columns:

	\begin{align} \label{diag: cone-contract}
		\begin{tikzcd}[column sep={30mm,between origins}, row sep={25mm,between origins}, ampersand replacement = \&]
			\gamma X \ar[r, tail, "\begin{pmatrix} i_{\gamma X} \\ -\gamma f \end{pmatrix}"] \ar[d, tail, "\id"] \& I_2(\gamma X) \oplus \gamma Y  \ar[r, two heads] \ar[tail]{d}{\begin{pmatrix} j_X & 0 \\ 0 & \id_{\gamma Y} \end{pmatrix}} \& C(\gamma f) \ar[d, dashed, tail, "c_X"] \\
			\gamma X \ar[r, tail, "\begin{pmatrix} \gamma i_{X} \\ \gamma(-f) \end{pmatrix}"] \ar[d, two heads] \& \gamma I_N(X) \oplus \gamma Y \ar[r, two heads] \ar[two heads]{d}{\begin{pmatrix} q_X & 0 \end{pmatrix}} \& \gamma C(f) \ar[d, two heads, dashed] \\
			0 \ar[r, tail] \& C \ar[r, two heads, "\id"] \& C
		\end{tikzcd}
	\end{align}
\end{con}

We summarize this result in the following

\begin{lem}\label{lem: cone-acyclic-prep}
	Let $f\colon X \to Y$ a morphism of $N$-complexes over an exact category $\E$. Consider the contraction functor $\gamma^n_r$, for some $n \in \ZZ$ and $r \in \{1,\dots,N-1\}$. Then there is a split admissible monic $c_X\colon C(\gamma^n_r(f)) \rightarrowtail \gamma^n_r(C(f))$ with cokernel in $\Prj(\C_N(\E)) \cap \C^{\infty, \varnothing^\ast}_N(\E)$. In particular:
	\begin{enumerate}
		\item \label{lem: cone-acyclic-prep-iso} The morphism $c_X$ is an isomorphism in $\K_2(\E)$. 
		
		\item \label{lem: cone-acyclic-prep-2} If $C(\gamma^n_r(f))$ is (totally) 2-acyclic at position $n$, then so is $\gamma^n_r(C(f))$.\qed
	\end{enumerate}
\end{lem}

\begin{prp} \label{prp: cone-acyclic} 
	Let $f\colon X \to Y$ be a morphism of $N$-complexes over an exact category $\E$. If $X$ and $Y$ are (totally) $N$-acyclic at positions $n, n+1, \dots, n+N-1 \in \ZZ$, then the cone $C(f)$ of $f$ is (totally) $N$-acyclic at position $n$. Dually, if $X$ and $Y$ are (totally) $N$-acyclic at positions $n-N+1, \dots, n-1, n \in \ZZ$, then the cocone $C^\ast(f)$ of $f$ is (totally) $N$-acyclic at position $n$.
\end{prp}

\begin{proof} By \Cref{rmk: gamma-acyclicity}, $\gamma^n_r(X)$ and $\gamma^n_r(Y)$ are (totally) 2-acyclic at positions $n$, $n+1$ for all $r \in \{1,\dots,N-1\}$. \Cref{lem: Buehler10.3,lem: cone-acyclic-prep}.\ref{lem: cone-acyclic-prep-2} then yield 2-acyclicity of $C(\gamma^n_r(f))$, and hence of $\gamma^n_r(C(f))$ at position $n$ for all $r$. This means exactly that $C(f)$ is $N$-acyclic at position $n$.
	Assuming total acyclicity, $\Hom_\E(\gamma^n_r(X), P)$ and $\Hom_\E(\gamma^n_r(Y), P)$ are 2-acyclic at positions $-n-1, -n$ for all $r$ and $P \in \Prj(\E)$.
	
	Using \Cref{rmk: I-P-Hom}.\ref{rmk: I-P-Hom-cone}, $\Hom_\A(C(\gamma^n_r(f)), P) \cong C^\ast(\Hom_\A(\gamma^n_r(f), P))$ is then 2-acyclic at position $-n$ by \Cref{lem: Buehler10.3}. This proves total 2-acyclicity of $C(\gamma^n_r(f))$, and hence of $\gamma^n_r(C(f))$ at $n$ for all $r$ due to \Cref{lem: cone-acyclic-prep}.\ref{lem: cone-acyclic-prep-2}. By \Cref{rmk: gamma-Hom}.\ref{rmk: gamma-Hom-tac}, $C(f)$ is now totally $N$-acyclic at position $n$.
\end{proof}

In view of \Cref{con: Sigma-explicit}.\ref{con: Sigma-explicit-cone}, we obtain, as a special case, the following

\begin{cor} \label{cor: Sigma-acyclic}
	Let $X \in \C_N(\E)$ be an $N$-complex over an exact category $\E$. If $X$ is (totally) acyclic at positions $n, n+1, \dots, n+N-1 \in \ZZ$, then $\Sigma X$ is (totally) acyclic at position $n$. Dually, if $X$ is (totally) acyclic at positions $n-N+1, \dots, n-1, n \in \ZZ$, then $\Sigma^{-1} X$ is (totally) acyclic at position $n$. \qed
\end{cor}

\begin{ntn} \label{ntn: Verdier-notation}
	We use Verdier's notation for various subcategories of $N$-complexes:
	\begin{itemize}[leftmargin=*]
		\item  For an additive category $\A$, let $\C^-_N(\A)$, $\C^+_N(\A)$ and $\C^b_N(\A)$ denote the subcategories of $\C_N(\A)$ consisting of $N$-complexes $X \in \C_N(\A)$ with $X^k = 0$ for any sufficiently large, any sufficiently small, and almost all $k \in \ZZ$, respectively. These subcategories are extension-closed and hence fully exact in $\C_N(\A)$, see \Cref{lem: fullyexact}.\ref{lem: fullyexact-extclosed}. To afford the following notation, we set $\C^{\infty}_N(\A) := \C_N(\A)$. For $\# \in \{\infty, +,-,b\}$, the projectively stable category of $\C_N^{\#}(\A)$ is denoted by $\K_N^{\#}(\A)$.
		\item For an exact category $\E$ and $\# \in \{\infty, +,-,b\}$, let $\C^{\#, -}_N(\E)$, $\C^{\#, +}_N(\E)$, $\C^{\#, b}_N(\E)$ and $\C^{\#, \varnothing}_N(\E)$ denote the subcategories of $\C^\#_N(\E) =: \C^{\#, \infty}_N(\E)$ consisting of $N$-complexes which are acyclic at all sufficiently large, sufficiently small, almost all, and all positions, respectively. In addition, we denote by $\C^{\#, \varnothing^\ast}_N(\E)$ the subcategory of $\C^{\#, \varnothing}_N(\E)$ consisting of totally acyclic $N$-complexes. For $\natural \in \{\infty, +, -, b, \varnothing, \varnothing^\ast\}$, the associated projectively stable category is denoted by $\K_N^{\#, \natural}(\A)$.
		
		\item It convenient to use the obvious partial ordering
		
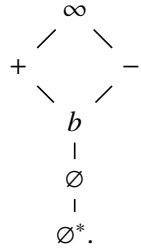
\begin{center}
			\begin{tikzcd}[sep={7.5mm,between origins}]
				& \infty \ar[-, rd] & \\
				+  \ar[-, rd]  \ar[-, ru] && - \\
				& b  \ar[-, ru] \ar[-, d] \\
				& \varnothing  \ar[-, d] \\
				& \varnothing^\ast.
			\end{tikzcd}
		\end{center}
	
		\item Let $\A$ be a subcategory of an exact category $\E$. For $\# \in \{\infty, +,-,b\}$ and $\natural \in \{\infty, +, -, b, \varnothing, \varnothing^\ast\}$, we abbreviate $\C^{\#, \natural_\E}_N(\A) := \C_N(\A) \cap \C^{\#, \natural}_N(\E)$, and denote its projectively stable category by $\K^{\#, \natural_\E}_N(\A)$.  There are two notable special cases:
		\begin{itemize}
			\item $\APC_N(\E) := \C^{\infty, \varnothing_\E}(\Prj(\E))$, the category of acyclic $N$-complexes of projectives, and
			
			\item $\TAPC_N(\E) := \C^{\infty, \varnothing^\ast_\E}(\Prj(\E))$, the category of totally acyclic $N$-complexes of projectives.
		\end{itemize}
		They coincide if $\E$ is Frobenius, see \Cref{rmk: Frobenius-tac}.
	\end{itemize}
\end{ntn}

\begin{thm} \label{thm: Verdier-notation}
	Let $\E$ be an exact category, $\# \in \{\infty, +,-,b\}$, and $\natural \in \{\infty, +, -, b, \varnothing, \varnothing^\ast\}$.
	\begin{enumerate}
		\item \label{thm: Verdier-notation-a} The subcategory $\C_N^{\#, \natural}(\E)$ is extension-closed and hence fully exact in $\C_N(\E)$.
		
		\item \label{thm: Verdier-notation-b} The category $\C_N^{\#, \natural}(\E)$ is a sub-Frobenius category of $\C_N(\E)$.
		
		\item \label{thm: Verdier-notation-c}  There are the following diagrams of canonical fully faithful, triangulated functors:
		
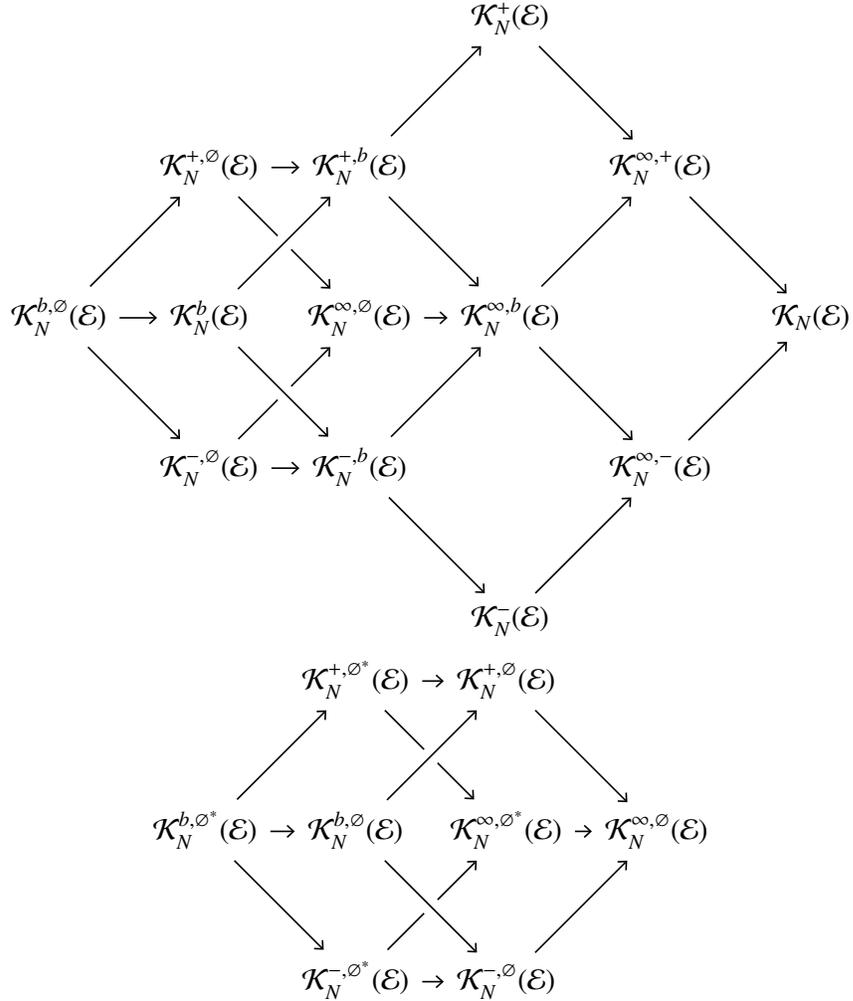
\begin{center}
			\begin{tikzcd}[sep={20mm,between origins}]
				&&& \K^+_N(\E) \ar[rd] && \\
				& \K^{+, \varnothing}_N(\E) \ar[rd] \ar[r] & \K^{+, b}_N(\E) \ar[ru] \ar[rd] && \K^{\infty, +}_N(\E) \ar[rd] \\
				\K^{b, \varnothing}_N(\E) \ar[rd] \ar[ru] \ar[r] & \K^b_N(\E) \ar[ru, crossing over] & \K^{\infty, \varnothing}_N(\E) \ar[r] & \K^{\infty, b}_N(\E) \ar[ru] \ar[rd] && \K_N(\E) \\
				& \K^{-, \varnothing}_N(\E) \ar[ru] \ar[r] & \K^{-, b}_N(\E) \ar[ru] \ar[rd] \ar[lu, <-, crossing over] && \K^{\infty, -}_N(\E) \ar[ru] \\
				&&& \K^-_N(\E) \ar[ru]
			\end{tikzcd}
			\begin{tikzcd}[sep={20mm,between origins}]
				& \K^{+, \varnothing^\ast}_N(\E) \ar[rd] \ar[r] & \K^{+, \varnothing}_N(\E) \ar[rd] & \\
				\K^{b, \varnothing^\ast}_N(\E) \ar[rd] \ar[ru] \ar[r] & \K^{b, \varnothing}_N(\E) \ar[ru, crossing over] & \K^{\infty, \varnothing^\ast}_N(\E) \ar[r] & \K^{\infty, \varnothing}_N(\E) \\
				& \K^{-, \varnothing^\ast}_N(\E) \ar[ru] \ar[r] & \K^{-, \varnothing}_N(\E) \ar[ru]  \ar[lu, <-, crossing over]
			\end{tikzcd}
		\end{center}
	\end{enumerate}
	For a subcategory $\A$ of $\E$, these statements hold more generally with $\E$ replaced by $\A$, and $\natural$ by $\natural_\E$.
\end{thm}

\begin{proof} Part \ref{thm: Verdier-notation-a} follows from \Cref{lem: fullyexact}.\ref{lem: fullyexact-extclosed} and \Cref{prp: acyclic-ext-closed}. For any $X \in \C^{\#, \natural}_N(\E)$, both $I(X)$ and $P(X)$ lie in $\C^{\#, \natural}_N(\E)$ due to \Cref{con: I-and-P} and \Cref{rmk: mu-acyclic}.\ref{rmk: mu-acyclic-statement}, and in $ \Prj(\C_N(\E))$  due to \Cref{lem: mu-ProjInj}. Since also $\Sigma X$ and $\Sigma^{-1}X$ lie in $\C^{\#, \natural}_N(\E)$ due to \Cref{con: Sigma-explicit}.\ref{con: Sigma-explicit-formula} and \Cref{cor: Sigma-acyclic}, the morphisms $i_X$ and $p_X$ from \Cref{con: I-and-P} are admissible in $\C^{\#, \natural}_N(\E)$, and part \ref{thm: Verdier-notation-b} follows. In view of \Cref{thm: C-Frob}, the assumptions of \Cref{lem: sub-Frobenius}.\ref{lem: sub-Frobenius-b} are then satisfied, and part \ref{thm: Verdier-notation-c} follows. Intersecting with $\C_N(\A)$, the preceding arguments also prove the more general claims.
\end{proof}

Similar arguments yield

\begin{prp}  \label{prp: K-in-K} Let $\E'$ be an exact subcategory of $\E$.
	Consider subcategories $\A$ and $\A'$ of $\E$ and $\E'$, respectively, and suppose that $\A'$ is a subcategory of $\A$. Then $\C^{\#, \natural_{\E'}}_N(\A')$ is a fully exact sub-Frobenius category of $\C^{\#, \natural_\E}_N(\A)$, and there is a canonical fully faithful, triangulated functor $\K^{\#, \natural_{\E'}}_N(\A') \to \K^{\#, \natural_\E}_N(\A)$ where $\# \in \{\infty, +,-,b\}$ and $\natural \in \{\infty, +, -, b, \varnothing, \varnothing^\ast\}$. \qed
\end{prp}

\begin{cor} \label{cor: APC}
	Let $\E'$ be an exact subcategory of $\E$ with $\Prj(\E') \subseteq \Prj(\E)$. Then $\APC_N(\E')$ is a fully exact sub-Frobenius category of $\APC_N(\E)$ and there is a canonical fully faithful, triangulated functor $\underline \APC_N(\E') \to \underline \APC_N(\E)$. The verbatim statement holds for $\TAPC$ instead of $\APC$. \qed
\end{cor}

\subsection{Acyclic $N$-arrays} \label{subsection: arrays}

The kernels in an acyclic $N$-complex can be organized in an array of bicartesian squares. Over an Abelian category, such arrays have been used for instance in {\cite{BM24}} and {\cite{IKM17}}. We build \emph{acyclic $N$-arrays} from acyclic $N$-complexes over an exact category and establish an equivalence between the respective categories (\Cref{thm: array}). Our construction is local with respect to the indices of the considered $N$-complex (\Cref{prp: array-local}). As a result, one can define \emph{soft truncations} as in the Abelian case (\Cref{dfn: soft-trunc}).

\begin{dfn}
	A \textbf{(bicartesian) $\boldsymbol N$-array} $X_\bullet^\bullet = (X_\bullet^\bullet, p_\bullet^\bullet, i_\bullet^\bullet)$ over an exact category, where $X^\bullet_\bullet = (X^k_r)^{k \in \ZZ}_{r \in \{0,\dots,N\}}$, $p^\bullet_\bullet = (p^k_r)^{k \in \ZZ}_{r \in \{1,\dots,N\}}$ and $i^\bullet_\bullet = (i^k_r)^{k \in \ZZ}_{r \in \{0,\dots,N-1\}}$, is a diagram consisting of (bicartesian) commutative squares
	\begin{center}
		\begin{tikzcd}[sep={15mm,between origins}]
			& X^{k+1}_{r+1} \ar[rd, "p^{k+1}_{r+1}"] & \\
			X^k_r \ar[rd, "p^k_r"'] \ar[ru, "i^{k}_{r}"] && X^{k+1}_r \\
			&X^{k}_{r-1} \ar[ru, "i^{k}_{r-1}"']
		\end{tikzcd}
	\end{center}
	for $k \in \ZZ$ and $r \in \{1,\dots,N-1\}$. We call $X^\bullet_\bullet$ \textbf{epic}, resp.~\textbf{monic}, if $p^k_r$ is an admissible epic, resp.~$i^k_r$ an admissible monic, for all $k$ and $r$. We call it \textbf{bounded above}, \textbf{bounded below}, or \textbf{bounded}, if $X^k_\bullet = 0$ for any sufficiently large, any sufficiently small, almost all $k \in \ZZ$, respectively.\\
	A morphism $f\colon X \to Y$ between such  arrays is a collection $f^\bullet_\bullet = (f^k_r)^{k \in \ZZ}_{r \in \{0,\dots,N\}}$ of morphisms $f^k_r\colon X^k_r \to Y^k_r$ which establish commutativity. We drop the bullets and write $X=(X, p, i)$ if there is no ambiguity.
\end{dfn}

\begin{rmk} \label{rmk: array-complex} Any $N$-array $X$ over an exact category $\E$ gives rise to morphisms $p^{\{N\}}\colon X^{\bullet}_N \to X^{\bullet}_0$ and $i^{\{N\}}\colon X^\bullet_0 \to X^\bullet_N$ in $\C(\E)$, where $d_{X^\bullet_0}=pi$ and $d_{X^\bullet_N}=ip$. Note that $X^\bullet_0, X^\bullet_N \in \C_N(\E)$ if $p^{\{N\}}=0$ or $i^{\{N\}}=0$.
\end{rmk}

\begin{dfn}
	An \textbf{acyclic $\boldsymbol N$-array} (of $X^\bullet_N \in \C_N(\E)$) over an exact category $\E$ is an epic and monic bicartesian $N$-array with $X^\bullet_0=0$, see also {\cite[Def.~4.4]{BM24}}. We denote the category of acyclic $N$-arrays over $\E$ by $A_N(\E)$.
\end{dfn}

Over an Abelian category, \Cref{thm: array} can be easily verified as mentioned in {\cite[21]{BM24}}. 

\begin{thm} \label{thm: array}
	For an exact category $\E$, the categories $A_N(\E)$ and $\C_N^{\infty, \varnothing}(\E)$ are equivalent. Under this equivalence, an acyclic $N$-complex $X \in \C_N^{\infty, \varnothing}(\E)$ corresponds to an acyclic $N$-array $(X^k_{r})^{k \in \ZZ}_{r = 0, \dots, N}$ of $X$, where $X^k_r = C^k_{(r)}(X) = Z^{k+N-r}_{(r)}(X)$ and, in particular, $X^\bullet_N=X$.
\end{thm}

\begin{ntn}\
	\begin{enumerate}
		\item By $A \hookrightarrow B$ (in contrast to $A \rightarrowtail B$), we denote a monic (which might not be admissible).
		
		\item By \begin{tikzcd}[sep=small, cramped] A \ar[r, epi_mini] & B \end{tikzcd} (in contrast to $A \twoheadrightarrow B$), we denote an epic (which might not be admissible).
	\end{enumerate}
\end{ntn}

\begin{lem} \label{lem: prep-thm-array} Let $\E$ be an exact category.
	\begin{enumerate}
		\item \label{lem: prep-thm-array-upward} Suppose that the following diagram $(XY)$ in $\E$ is bicartesian:
		
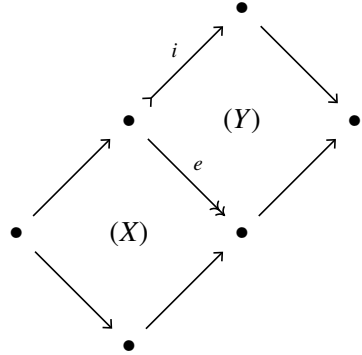
\begin{center}
			\begin{tikzcd}[sep={15mm,between origins}]
				&& \bullet \ar[rd] \ar[dd, phantom, "(Y)"] &\\
				& \bullet \ar[rd, epi, "e"] \ar[ru, tail, "i"] \ar[dd, phantom, "(X)"]  &  & \bullet \\
				\bullet \ar[ru] \ar[rd] && \bullet \ar[ru] \\
				& \bullet \ar[ru]
			\end{tikzcd}
		\end{center}
		Then $(Y)$ is bicartesian. If $e$ is an admissible epic, then also $(X)$ is bicartesian.
		
		\smallskip
		
		\item \label{lem: prep-thm-array-downward} Suppose that the following diagram $(Y'X')$ in $\E$ is bicartesian:
		
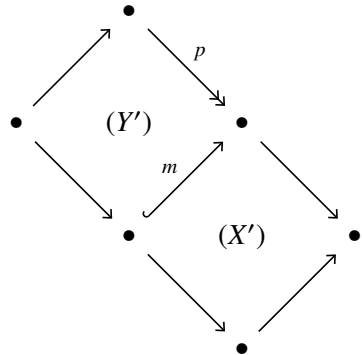
\begin{center}
			\begin{tikzcd}[sep={15mm,between origins}]
				& \bullet \ar[rd, two heads, "p"] \ar[dd, phantom, "(Y')"] && \\
				\bullet \ar[ru] \ar[rd] && \bullet \ar[rd] \ar[dd, phantom, "(X')"] \\
				& \bullet \ar[ru, hook, "m"] \ar[rd]&& \bullet \\
				&& \bullet \ar[ru]
			\end{tikzcd}
		\end{center}
		Then $(Y')$ is bicartesian. If $m$ is an admissible monic, then also $(X')$ is bicartesian.
	\end{enumerate}
\end{lem}

\begin{proof} \
	\begin{enumerate}[leftmargin=*]
		\item As $e$ is an epic, $(Y)$ is a pushout by \Cref{lem: pushpull-sep}.\ref{lem: pushpull-sep-pushout} and hence bicartesian by \Cref{cor: bicart}.\ref{cor: bicart-already}. As $i$ is a monic, $(X)$ is a pullback by \Cref{lem: pushpull-sep}.\ref{lem: pushpull-sep-pullback} and hence bicartesian by \Cref{cor: bicart}.\ref{cor: bicart-already} if $e$ is an admissible epic.
		
		\item is dual to \ref{lem: prep-thm-array-upward}. \qedhere
	\end{enumerate}
\end{proof}

\begin{lem} \label{lem: array-admissible} Over an exact category $\E$, consider the commutative diagram
	
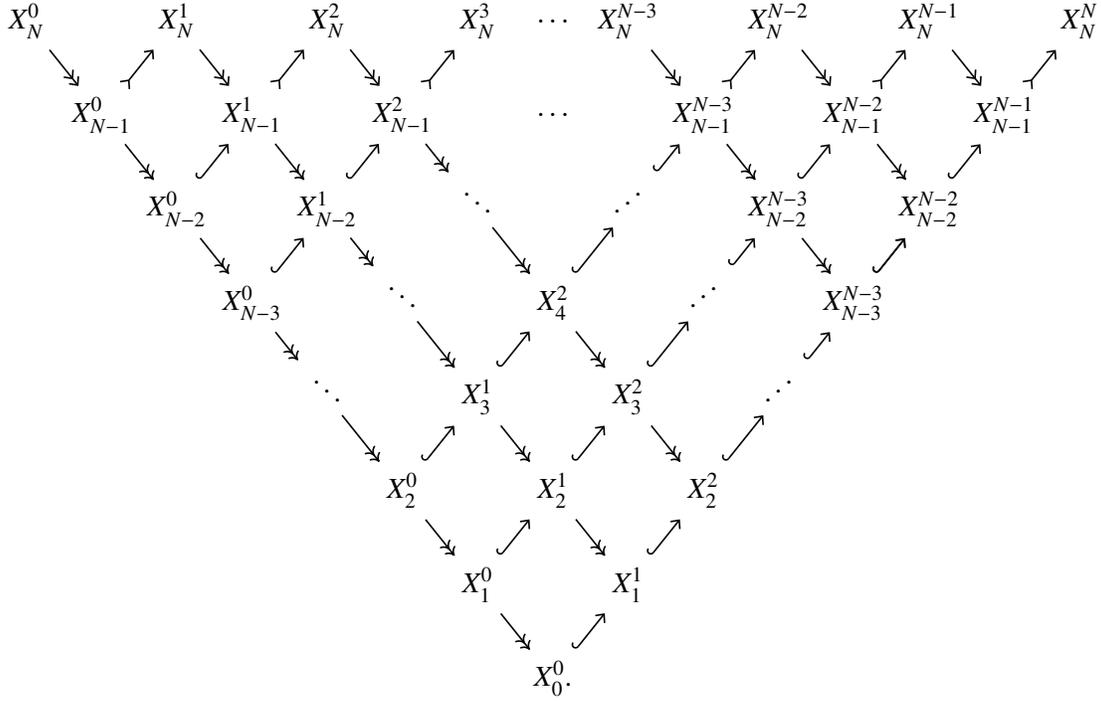
\begin{center}
		\begin{tikzcd}[column sep={10mm,between origins}, row sep={12.5mm,between origins}]
			X^0_N \ar[rd, two heads] && X^1_N \ar[rd, two heads] && X^2_N \ar[rd, two heads] && X^3_N & \cdots & X^{N-3}_N \ar[rd, two heads] && X^{N-2}_N \ar[rd, two heads] && X^{N-1}_N \ar[rd, two heads] && X^N_N \\
			& X^0_{N-1} \ar[rd, epi] \ar[ru, tail] && X^1_{N-1} \ar[rd, epi] \ar[ru, tail] && X^2_{N-1} \ar[rd, epi] \ar[ru, tail] &&\cdots && X^{N-3}_{N-1} \ar[rd, epi] \ar[ru, tail] && X^{N-2}_{N-1} \ar[rd, epi] \ar[ru, tail] && X^{N-1}_{N-1} \ar[ru, tail] \\
			&& X^0_{N-2} \ar[ru, hook] \ar[rd, epi] && X^1_{N-2} \ar[ru, hook] \ar[rd, epi] && \ddots \ar[rd, epi] && \udots \ar[ru, hook] && X^{N-3}_{N-2} \ar[ru, hook] \ar[rd, epi] && X^{N-2}_{N-2} \ar[ru, hook] \\
			&&& X^0_{N-3} \ar[ru, hook] \ar[rd, epi] && \ddots \ar[rd, epi] && X^{2}_{4} \ar[ru, hook] \ar[rd, epi] && \udots \ar[ru, hook] && X^{N-3}_{N-3} \ar[ru, hook] \ar[ru, hook] \\
			&&&& \ddots \ar[rd, epi] && X^{1}_3 \ar[ru, hook] \ar[rd, epi] && X^{2}_3 \ar[ru, hook] \ar[rd, epi] && \udots \ar[ru, hook] \\
			&&&&& X^{0}_2 \ar[ru, hook] \ar[rd, epi] && X^{1}_2 \ar[ru, hook] \ar[rd, epi] && X^2_2 \ar[ru, hook] \\
			&&&&&& X^{0}_1 \ar[ru, hook] \ar[rd, epi] && X^1_1 \ar[ru, hook] \\
			&&&&&&& X^0_0. \ar[ru, hook]
		\end{tikzcd}
	\end{center}
	
	For $n \in \{1, \dots, N-1\}$ and $r \in \{n, \dots, N-1\}$, let $(Y^n_r)$ denote the square
	
	
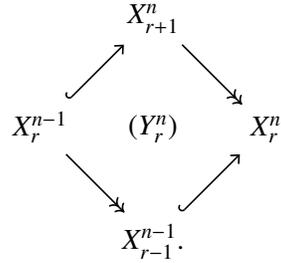
\begin{center}
		\begin{tikzcd}[sep={15mm,between origins}]
			& X^n_{r+1} \ar[rd, epi] \ar[dd, phantom, "(Y^n_r)"]&\\
			X^{n-1}_{r} \ar[ru, hook] \ar[rd, epi] && X^{n}_{r} \\
			& X^{n-1}_{r-1}. \ar[ru, hook]
		\end{tikzcd}
	\end{center}
	
	Suppose that for each $n \in \{1, \dots, N-1\}$ the concatenation of
	
	
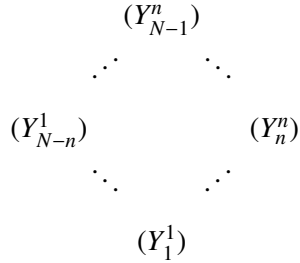
\begin{center}
		\begin{tikzcd}[sep={7.5mm,between origins}]
			&& (Y^{n}_{N-1}) && \\
			&\udots && \ddots \\
			(Y^{1}_{N-n}) &&&& (Y^{n}_n) \\
			&\ddots && \udots\\
			&& (Y^{1}_1)
		\end{tikzcd}
	\end{center}
	is bicartesian. Then each individual square $(Y^{n}_r)$  is bicartesian with admissible morphisms.
\end{lem}

\begin{proof} By \Cref{lem: prep-thm-array}.\ref{lem: prep-thm-array-upward}, the assumption implies that $(Y^{n}_{N-1} \cdots Y^{n}_{n})$ bicartesian. We prove the claim by descending induction on $r= N-1, \dots, n$. Starting with $r =N-1$, the square $(Y^{n}_{N-1})$ is bicartesian due to \Cref{lem: prep-thm-array}.\ref{lem: prep-thm-array-downward}, and its morphisms are admissible by \Cref{cor: bicart}.\ref{cor: bicart-prep}. Let now $n, n' \in \{1, \dots, N-1\}$ and $r, r' \in \{n, \dots, N-2\}$, and suppose that the claim holds for $(Y^{n'}_{r'})$ whenever $r' > r$.
	Then $(Y^{n}_{N-1} \cdots Y^{n}_{r+1})$ is bicartesian with admissible morphisms, see \Cref{lem: pasting-law}, and the upper morphisms of $(Y^{n}_r)$ are admissible. So, $(Y^{n}_{r} \cdots Y^{n}_n)$ is bicartesian by the additional claim of \Cref{lem: prep-thm-array}.\ref{lem: prep-thm-array-downward}. Now the claim for $(Y^{n}_r)$ follows from \Cref{lem: prep-thm-array}.\ref{lem: prep-thm-array-downward} and \Cref{cor: bicart}.\ref{cor: bicart-prep} as before.  
\end{proof}

\begin{prp} \label{prp: array-local}
	Let $X \in \C_N(\E)$ be an $N$-complex over an exact category $\E$ and $n \in \ZZ$. Suppose that all compositions of differentials between positions in $\{n, n+1, \dots, n+N\}$ are admissible.
	
	\begin{enumerate}[leftmargin=*]
		\item \label{prp: array-local-diagram}  For each choice $X^k_{r}$ of (co)images, see \Cref{rmk: admissible}.\ref{rmk: admissible-analysis}, there is a unique commutative diagram:
		
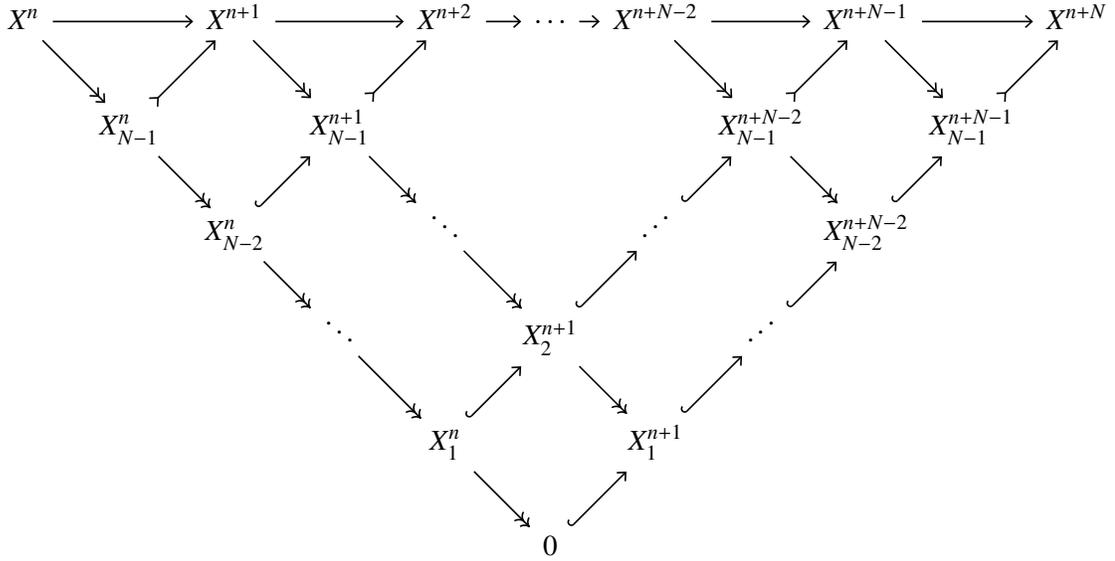
\begin{center}
			\begin{tikzcd}[sep={14mm,between origins}]
				X^{n} \ar[rr] \ar[rd, two heads] && X^{n+1} \ar[rr] \ar[rd, two heads] && X^{n+2} \ar[r] & \cdots \ar[r] & X^{n+N-2} \ar[rr]\ar[rd, two heads]  && X^{n+N-1} \ar[rr] \ar[rd, two heads] && X^{n+N} \\
				& X^{n}_{N-1} \ar[rd, epi] \ar[ru, tail] && X^{n+1}_{N-1} \ar[rd, epi] \ar[ru, tail] && && X^{n+N-2}_{N-1} \ar[rd, epi] \ar[ru, tail] && X^{n+N-1}_{N-1} \ar[ru, tail] \\
				&& X^{n}_{N-2} \ar[rd, epi] \ar[ru, hook] && \ddots \ar[rd, epi] && \udots \ar[ru, hook] && X^{n+N-2}_{N-2} \ar[ru, hook] \\
				&&& \ddots\ar[rd, epi]  && X^{n+1}_{2}\ar[rd, epi] \ar[ru, hook] && \udots \ar[ru, hook] \\
				&&&& X^{n}_{1} \ar[rd, epi] \ar[ru, hook] && X^{n+1}_{1} \ar[ru, hook] \\
				&&&&& 0 \ar[ru, hook]
			\end{tikzcd}
		\end{center}
		If $X$ acyclic at position $k \in \{n, \dots, n+N\}$, we have  $X^k_r = C^k_{(r)}(X)$ and $X^{k-N+r}_r = Z^{k}_{(r)}(X)$, for all occurring $r$. If this holds for all $k$, then every morphism of such $N$-complexes induces a unique morphism of the associated diagrams.
	
		\item \label{prp: array-local-admissible} If $\gamma^{n+k}_{N-k}(X)$ is 2-acyclic at position $n+k$ for each $k \in \{1, \dots, N-1\}$, then all squares of the diagram in in \ref{prp: array-local-diagram} are bicartesian with admissible morphisms. In particular, this holds if $X$ is acyclic at positions $n+1, \dots, n+N-1$.
	\end{enumerate}
\end{prp}

\begin{proof}\
	\begin{enumerate}[leftmargin=*]
		\item The morphisms in the diagram exist due to the universal property of (co)images. It commutes by repeated application of \Cref{rmk: help-commute}. The morphisms $f^k_r\colon X^k_{r} \to  Y^k_{r}$ induced by $f$ on cokernels form a morphism of the respective diagrams associated to $X$ and $Y$. Indeed, using \Cref{rmk: help-commute} multiple times, one verifies that the following diagrams commute for all occurring indices $k$, $r$ and $s$: 
		
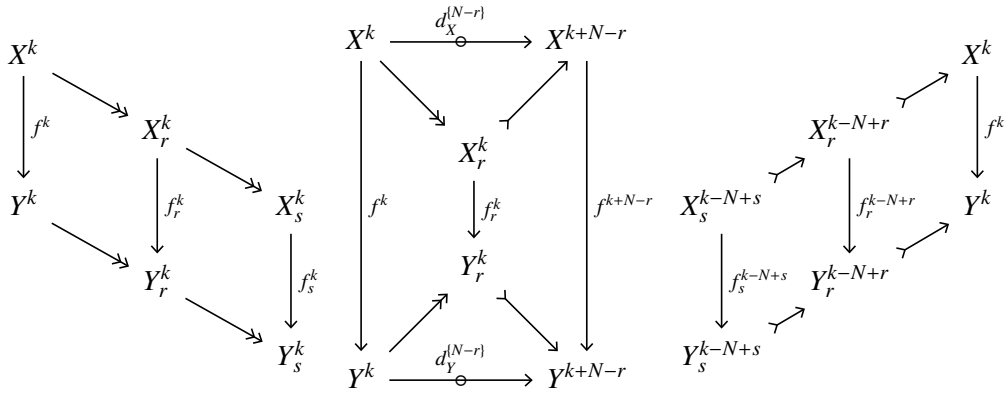
		\begin{center}
			\begin{tikzcd}[column sep={14mm,between origins}, row sep={10mm,between origins}]
				X^k \ar[dd, "f^k"] \ar[rd, two heads] \\
				& X^{k}_{r} \ar[dd, "f^{k}_{r}"] \ar[rd, two heads] \\
				Y^k \ar[rd, two heads]&& X^{k}_{s} \ar[dd, "f^{k}_{s}"] \\
				& Y^{k}_{r} \ar[rd, two heads] \\
				&& Y^{k}_{s}
			\end{tikzcd}
			\begin{tikzcd}[sep={15mm,between origins}]
				X^k \ar[rr, "\circ" marking, "d_X^{\{N-r\}}"] \ar[rd, two heads] \ar[ddd, "f^k"] && X^{k+N-r} \ar[ddd, "f^{k+N-r}"] \\			
				& X^{k}_{r} \ar[ru, tail] \ar[d, "f^{k}_{r}"] \\
				& Y^{k}_{r} \ar[rd, tail]  \\
				Y^k \ar[rr, "d_Y^{\{N-r\}}", "\circ" marking] \ar[ru, two heads]  & \textcolor{white}{\bullet} & Y^{k+N-r}
			\end{tikzcd}
			\begin{tikzcd}[column sep={17mm,between origins}, row sep={10mm,between origins}]
				&& X^k \ar[dd, "f^k"] \\
				& X^{k-N+r}_{r} \ar[dd, "f^{k-N+r}_r"]  \ar[ru, tail] \\
				X^{k-N+s}_{s} \ar[dd, "f^{k-N+s}_{s}"] \ar[ru, tail] && Y^k\\
				& Y^{k-N+r}_{r} \ar[ru, tail] \\
				Y^{k-N+s}_{s} \ar[ru, tail]
			\end{tikzcd}
		\end{center}
	
	\item By hypothesis, \begin{tikzcd}[cramped] X^n_{N-k} \ar[r, tail] & X^{n+k} \ar[r, two heads] & X^{n+k}_k \end{tikzcd} is short exact, and hence
	
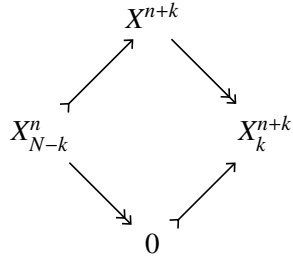
\begin{center}
		\begin{tikzcd}[sep={15mm,between origins}]
			& X^{n+k} \ar[rd, two heads]& \\
			X^n_{N-k} \ar[ru, tail] \ar[rd, two heads] && X^{n+k}_k \\
			& 0  \ar[ru, tail]
		\end{tikzcd}
	\end{center}
	 is bicartesian for all $k \in \{1, \dots, N-1\}$ by \Cref{prp: Buehler2.12}. \Cref{lem: array-admissible} yields the claim. \qedhere
	\end{enumerate}
\end{proof}

\begin{proofof}{\Cref{thm: array}}
	 Given an acyclic $N$-array $X$, the sequences \begin{tikzcd}[cramped, sep=small] X^{n-r}_{N-r} \ar[r, tail] & X^n_N \ar[r, two heads] & X^{n}_{r} \end{tikzcd} are exact for all $n$ and $r$ by \Cref{prp: Buehler2.12}. So, $X_N^\bullet$ is an acyclic $N$-complex, see \Cref{rmk: array-complex}, with $C^n_{(r)}(X_N^\bullet) = Z^{n+N-r}_{(r)}(X_N^\bullet) = X^n_r$ for all $n$ and $r$. This yields a functor $A_N(\E) \to \C^{\infty, \varnothing}_N(\E)$ which sends a morphism $f\colon X \to Y$ in $A_N(\E)$ to the morphism $f^\bullet_N\colon X_N^\bullet \to Y_N^\bullet$ in $\C_N(\E)$.\\
	 To show essential surjectivity, let $X \in \C^{\infty, \varnothing}_N(\E)$ be an acyclic $N$-complex. Apply \Cref{prp: array-local} for all $n \in \ZZ$ with a fixed choice of all (co)images. The respective diagrams patch together to form an acyclic $N$-array, sent to $X$ under the above functor. Fullness follows from \Cref{prp: array-local} as well. To see that the functor is faithful, suppose that $f^n_N=0$ for all $n \in \ZZ$. Then the commutative diagram
	
	\begin{center}
		\begin{tikzcd}[sep={17.5mm,between origins}]
			X^{n}_{N} \ar[r, two heads] \ar[d, "f^n_N"] \ar[rd, "0"] & X^n_{r} \ar[d, "f^{n}_{r}"] \\
			Y^{n}_{N} \ar[r, two heads] & Y^n_r
		\end{tikzcd}
	\end{center}
	shows that $f^n_r =0$ for all $n$ and $r$.\qedhere
\end{proofof}

Based on \Cref{prp: array-local}, we define soft truncations. For convenience, we impose slightly stronger acyclicity hypotheses than necessary.

\begin{dfn} \label{dfn: soft-trunc}
	Let $X \in \C_N(\E)$ be an $N$-complex over an exact category  $\E$ and $n \in \ZZ$.
	
	\begin{enumerate}
		\item \label{dfn: soft-trunc-left} If $X$ is acyclic at positions $n,\dots, n+N-2$, its \textbf{(left) soft truncation} is the $N$-complex
		
		\begin{center}
			\begin{tikzcd}[sep=small]
				\sigma^{\geq n} X\colon & C^{n}_{(1)}  \ar[r, tail] & \cdots  \ar[r, tail] & C^{n+N-2}_{(N-1)}  \ar[r, tail] & X^{n+N-1} \ar[r] & X^{n+N} \ar[r] & X^{n+N+1} \ar[r] & \cdots.
			\end{tikzcd}
		\end{center}
		Note that $C^{n+r-1}_{(r)}(X)=Z^{n+N-1}_{(r)}(X)$ for all $r \in \{1, \dots, N-1\}$ if $X$ is acyclic also at position $n+N-1$.
		\item \label{dfn: soft-trunc-right} If $X$ is acyclic at positions $n-N+2, \dots, n$, its \textbf{(right) soft truncation} is the $N$-complex
		\begin{center}
			\begin{tikzcd}[sep=small]
				\sigma^{\leq n} X\colon & \cdots \ar[r] & X^{n-N-1} \ar[r] & X^{n-N} \ar[r] & X^{n-N+1} \ar[r, two heads] & Z^{n-N+2}_{(N-1)} \ar[r, two heads] & \cdots \ar[r, two heads] & Z^{n}_{(1)}.
			\end{tikzcd}
		\end{center}
		Note that $Z^{n-r+1}(X)=C^{n-N+1}_{(r)}(X)$ for all $r \in \{1, \dots, N-1\}$ if $X$ is acyclic also at position $n-N+1$.
	\end{enumerate}
\end{dfn}

\begin{cor} \label{cor: soft-trunc-seq} Let $X \in \C_N(\E)$ be an $N$-complex over an exact category  $\E$ and $n \in \ZZ$.
	\begin{enumerate}
		\item \label{cor: soft-trunc-seq-a} If $X$ is acyclic at all positions greater than or equal to $n$, then $\sigma^{\geq n} X$ is acyclic. Dually, if $X$ is acyclic at all positions up to $ n$, then $\sigma^{\leq n} X$ is acyclic.
		
		\item \label{cor: soft-trunc-seq-b} If $X$ is acyclic at positions $n-N+1, \dots, n+1 \in \ZZ$, then there is a term\emph{wise} short exact sequence of $N$-complexes
		
		\begin{center}
			\begin{tikzcd}[column sep=small]
				\sigma^{\leq n} X\colon \ar[d] & \cdots \ar[r] & X^{n-N+1} \ar[r, two heads] \ar[d, tail, "\id"] & X^{n-N+1}_{N-1} \ar[r, two heads] \ar[d, tail] & \cdots \ar[r, two heads] & X^{n-N+1}_{1}  \ar[r] \ar[d, tail] & 0 \ar[r] \ar[d, tail] & \cdots \\
				X\colon \ar[d] & \cdots \ar[r] & X^{n-N+1} \ar[r] \ar[d, two heads] & X^{n-N+2} \ar[r] \ar[d, two heads] & \cdots \ar[r] & X^{n} \ar[r] \ar[d, two heads] & X^{n+1} \ar[r] \ar[d, two heads, "\id"]  & \cdots \\
				\sigma^{\geq n-N+2}X\colon& \cdots \ar[r] & 0 \ar[r] & X^{n-N+2}_{1} \ar[r, tail] & \cdots \ar[r, tail] & X^{n}_{N-1} \ar[r, tail]  & X^{n+1} \ar[r] & \cdots,
			\end{tikzcd}
		\end{center}
		where $X^k_r = C^k_{(r)}(X) = Z^{k+N-r}_{(r)}(X)$ for all occurring indices $k$ and $r$.
	\end{enumerate}
\end{cor}

\begin{proof}
	 Part \ref{cor: soft-trunc-seq-b} is obvious from the definitions. To see \ref{cor: soft-trunc-seq-a}, patch all diagrams that can be obtained from \Cref{prp: array-local}, as in the proof of \Cref{thm: array}. Then modify the resulting diagram at the left end  as follows, and extend it by zero to create an acyclic $N$-array of $\sigma^{\geq n} X$:
	
	\begin{center}
		\begin{footnotesize}
			\begin{tikzcd}[sep={11mm,between origins}]
				X^{n}_{1} \ar[rr, tail] \ar[rd, equal] && X^{n+1}_{2} \ar[rr, tail] \ar[rd, equal] && X^{n+2}_{3} \ar[>-, r] & \cdots \ar[r] & X^{n+N-3}_{N-2} \ar[rr, tail] \ar[rd, equal]  && X^{n+N-2}_{N-1} \ar[rr, tail] \ar[rd, equal] && X^{n+N-1} \ar[rd, two heads] \ar[rr] && X^{n+N} \ar[rd, two heads] \ar[r] & \cdots\\
				& X^{n}_{1} \ar[rd, equal] \ar[ru, tail] && X^{n+1}_{2} \ar[rd, equal] \ar[ru, tail] && && X^{n+N-3}_{N-2} \ar[rd, equal] \ar[ru, tail] && X^{n+N-2}_{N-1} \ar[rd, two heads] \ar[ru, tail] && X^{n+N-1}_{N-1} \ar[ru, tail] \ar[rd, two heads] && \ddots \\
				&& X^{n}_{1} \ar[rd, equal] \ar[ru, tail] && \ddots \ar[rd, equal] && \udots \ar[ru, tail] && X^{n+N-3}_{N-2} \ar[rd, two heads] \ar[ru, tail] && X^{n+N-2}_{N-2} \ar[ru, tail] \ar[rd, two heads] && \ddots \\
				&&& \ddots\ar[rd, equal]  && X^{n+1}_{2}\ar[rd, equal] \ar[ru, tail] && \udots \ar[ru, tail] && X^{n+N-3}_{N-3} \ar[ru, tail] \ar[rd, two heads] && \ddots \\
				&&&& X^{n}_{1} \ar[rd, equal] \ar[ru, tail] && X^{n+1}_{2} \ar[rd, two heads] \ar[ru, tail] && \udots \ar[ru, tail] && \ddots \\
				&&&&& X^n_{1} \ar[ru, tail] \ar[rd, two heads] && X^{n+1}_{1} \ar[ru, tail] \ar[rd, two heads] && \udots \\
				&&&&&& 0 \ar[ru, tail] && 0 \ar[ru, tail]
			\end{tikzcd}
		\end{footnotesize}
	\end{center}
	
	The claimed acyclicity follows from \Cref{thm: array}.
\end{proof}

\subsection{Resolutions of $N$-complexes} \label{subsection: resolutions}

Keller described (injective) resolutions of 2-complexes over an exact category $\E$ which are bounded on one side, see {\cite[4.1, Lemma]{Kel90}}. In this subsection we generalize his approach to construct projectively resolving $N$-arrays, which then yield projective $N$-resolutions (\Cref{cor: resolution-exist}). For elements of $\mMor_{N-2}(\E)$, see \Cref{ntn: iota}, such resolutions take the form of one-sided acyclic $N$-arrays (\Cref{cor: syz-resolution}).

\begin{dfn} Let $\E$ be an exact category.
	\begin{enumerate}
		\item We refer to an epic bicartesian $N$-array $(X, p, i)$ over $\E$ with $i^{\{N\}}=0$ as a \textbf{resolving $\boldsymbol N$-array} (of $X^\bullet_0 \in \C_N(\E)$). We call it \textbf{projectively resolving} if $X^\bullet_N \in \C_N(\Prj(\E))$, see \Cref{rmk: array-complex}.
		
		\item We refer to a monic bicartesian $N$-array $(X, p, i)$ over $\E$ with $p^{\{N\}}=0$ as a \textbf{coresolving $\boldsymbol N$-array} (of $X^\bullet_0 \in \C_N(\E)$). We call it \textbf{injectively coresolving} if $X^\bullet_0 \in \C_N(\Inj(\E))$, see \Cref{rmk: array-complex}.
	\end{enumerate}
\end{dfn}

In the following, we consider only (projectively) resolving $N$-arrays. However, there are obvious dual statements on (injectively) coresolving $N$-arrays.

\begin{prp} \label{prp: resolution-exist}
	Let $\E$ be an exact category with enough projectives. Then any bounded above $N$-complex over $\E$ admits a bounded above, projectively resolving $N$-array.
\end{prp}

The proof relies on the following argument of Keller:

\begin{lem} \label{lem: help-resolution} Consider a commutative diagram
	
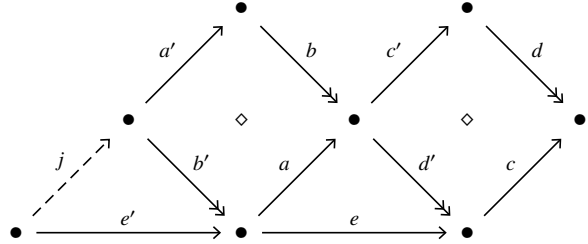
\[
	\begin{tikzcd}[sep={15mm,between origins}]
		&& \bullet \ar[rd, two heads, "b"] && \bullet \ar[rd, two heads, "d"] &\\
		& \bullet \ar[rd, two heads, "b'"] \ar[ru, "a'"] & \diamond & \bullet \ar[rd, two heads, "d'"] \ar[ru, "c'"] &\diamond& \bullet \\
		\bullet \ar[rr, "e'"] \ar[ru, dashed, "j"]&& \bullet \ar[ru, "a"] \ar[rr, "e"] && \bullet \ar[ru, "c"]
	\end{tikzcd}
	\]
	 of two bicartesian squares in an additive category. If $c'a = 0$ and $ee'=0$, then there exists a morphism $j$ completing the diagram with $a'j = 0$.
\end{lem}

\begin{proof} Interpreting the bicartesian squares as short exact sequences as in \Cref{prp: Buehler2.12}.\ref{prp: Buehler2.12-pull}, the hypotheses yield
	\[\begin{pmatrix} c' \\ d'\end{pmatrix} \begin{pmatrix} -b & a\end{pmatrix} \begin{pmatrix} 0 \\ e' \end{pmatrix} = \begin{pmatrix} c'ae' \\ d'ae'\end{pmatrix} = \begin{pmatrix} c'ae' \\ ee'\end{pmatrix} = 0,\]
	and hence $\begin{pmatrix} -b & a\end{pmatrix} \begin{pmatrix} 0 \\ e' \end{pmatrix} = 0$, since $\begin{pmatrix} c' \\ d'\end{pmatrix}$ is a monic. Then $j$ exists by the universal property of pullbacks.
\end{proof}

\begin{proofof}{\Cref{prp: resolution-exist}}
	Let $X\in \C^{-}_N(\E)$, and set $m := \max\{ k \in \ZZ \, \vert \, X^k \neq 0\}$. We construct a projectively resolving $N$-array $X^\bullet_\bullet$ of $X$ as follows:
	
	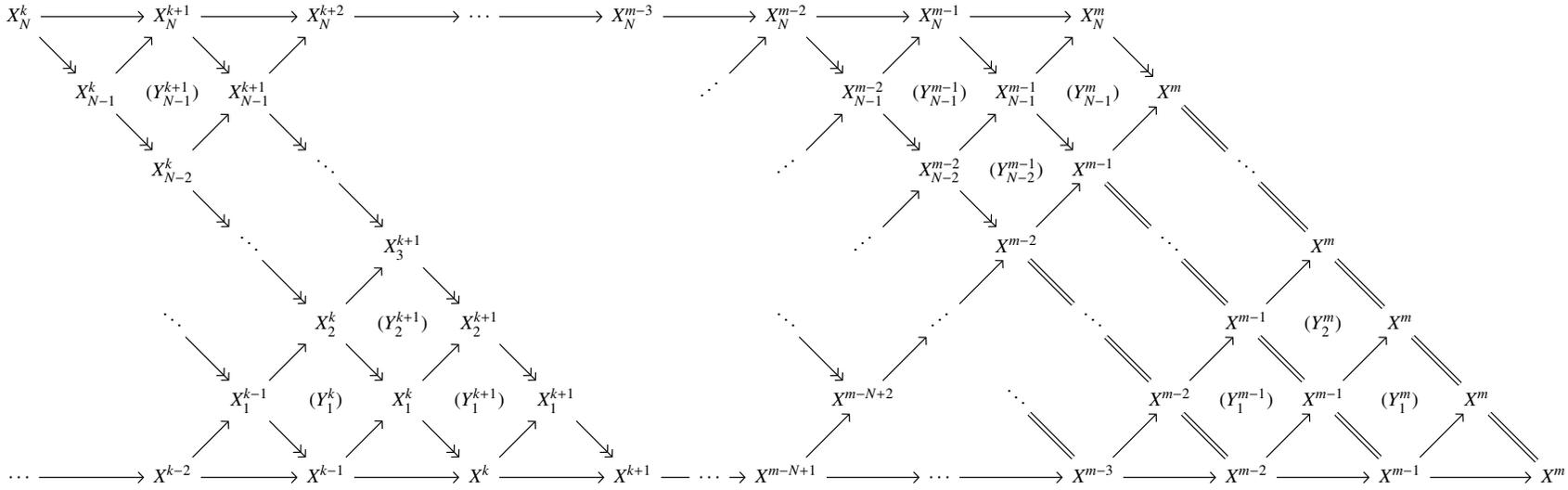

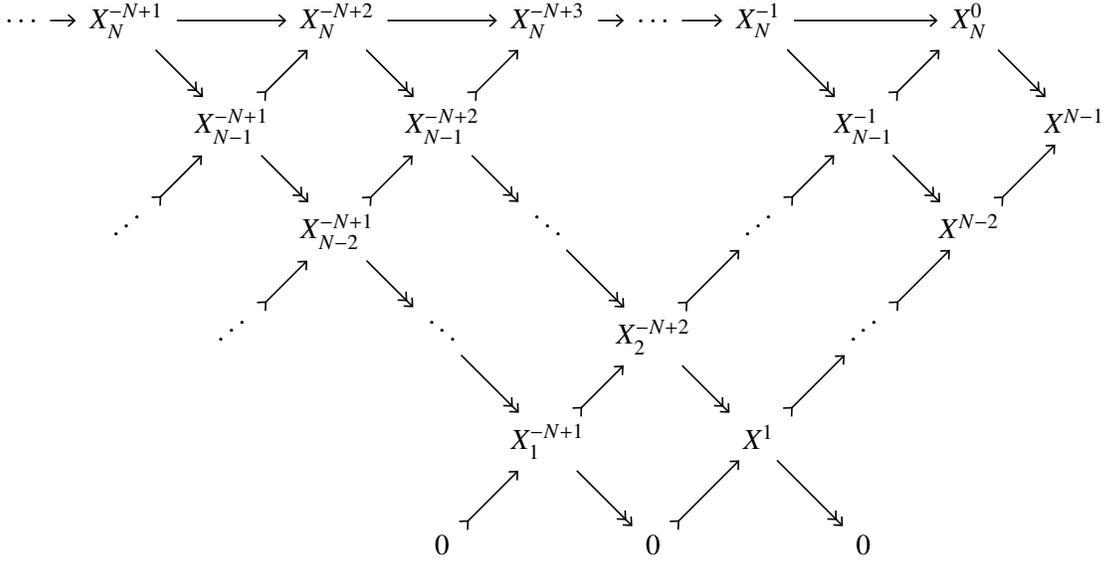
\begin{figure}
		\begin{scriptsize}
			\rotatebox{90}{
				\begin{tikzcd}[sep={11mm,between origins}, ampersand replacement = \&]
					X_N^k \ar[dr, two heads] \ar[rr] \& \&  X_N^{k+1} \ar[dr, two heads] \ar[rr] \& \&  X_N^{k+2} \ar[rr]  \&\& \cdots \ar[rr] \& \& X_N^{m-3} \ar[rr] \& \&  X_N^{m-2} \ar[rd, two heads] \ar[rr]\& \& X_N^{m-1} \ar[rd, two heads] \ar[rr] \& \& X_N^m \ar[rd, two heads] \& \& \& \& \& \&\\
					\&  X^{k}_{N-1} \ar[dr, two heads] \ar[ru] \& (Y^{k+1}_{N-1}) \&  X^{k+1}_{N-1} \ar[dr, two heads] \ar[ru] \& \&  \& \& \& \& \udots \ar[ru] \& \& X^{m-2}_{N-1} \ar[rd, two heads] \ar[ru] \& (Y^{m-1}_{N-1}) \& X^{m-1}_{N-1} \ar[ru] \ar[rd, two heads] \& (Y^{m}_{N-1}) \& X^m  \ar[rd, equal] \\
					\& \&  X^{k}_{N-2} \ar[ru] \ar[rd, two heads] \&  \&  \ddots \ar[rd, two heads] \& \& \& \& \& \& \udots \ar[ru] \& \&  X^{m-2}_{N-2} \ar[ru] \ar[rd, two heads] \& (Y^{m-1}_{N-2}) \& X^{m-1} \ar[ru] \ar[rd, equal] \&\& \ddots \ar[rd, equal]  \\
					\& \& \&  \ddots \ar[dr, two heads] \& \&  X^{k+1}_3 \ar[rd, two heads] \& \& \& \& \& \& \udots \ar[ru]  \&\& X^{m-2} \ar[ru] \ar[rd, equal] \&\& \ddots \ar[rd, equal] \&\& X^m  \ar[rd, equal]\\
					\& \&  \ddots \ar[dr, two heads]  \& \&  X^k_2  \ar[rd, two heads] \ar[ru] \& (Y^{k+1}_2) \&  X^{k+1}_{2} \ar[rd, two heads] \& \& \& \&\ddots \ar[rd, two heads] \& \& \udots \ar[ru] \&\& \ddots \ar[rd, equal] \&\&  X^{m-1} \ar[rd, equal] \ar[ru] \&(Y^m_2)\& X^m \ar[rd, equal]  \\
					\& \& \&  X^{k-1}_1 \ar[rd, two heads] \ar[ru] \& (Y^{k}_1) \&  X^k_1 \ar[rd, two heads] \ar[ru] \& (Y^{k+1}_1) \&  X^{k+1}_1 \ar[rd, two heads]  \& \& \& \&  X^{m-N+2} \ar[ru] \&\& \ddots \ar[rd, equal] \&\&  X^{m-2} \ar[rd, equal] \ar[ru] \& (Y^{m-1}_1) \& X^{m-1} \ar[rd, equal] \ar[ru] \& (Y^m_1) \& \ar[rd, equal] X^m  \\
					\cdots \ar[rr] \& \&  X^{k-2} \ar[ru] \ar[rr] \& \&  X^{k-1} \ar[ru] \ar[rr] \& \&  X^k \ar[ru] \ar[rr] \& \&  X^{k+1} \ar[r,-] \& \cdots \ar[r] \& X^{m-N+1} \ar[rr]\ar[ru] \&\& \cdots \ar[rr] \&\& X^{m-3} \ar[rr] \ar[ru]  \&\& X^{m-2} \ar[ru]  \ar[rr] \&\& X^{m-1} \ar[rr] \ar[ru] \&\& X^m
				\end{tikzcd}
				}
			\end{scriptsize}
		\caption{Keller's resolution for general $N$}
	\end{figure}
	
	For all $k > m$ and $r \in \{1, \dots, N-1\}$, set $X^k_r := 0$. Fix $n \in \ZZ_{\leq m}$ and assume that $(Y^k_r)$ has already been constructed for $r \in \{1, \dots, N\}$ and $k > n$. Consider the following diagram joining the concatenated bicartesian squares $(Y^{n+1}_1 \cdots Y^{n+N-1}_{N-1})$ and $(Y^{n+N}_{N-1} \cdots Y^{n+N}_{1})$, see \Cref{lem: pasting-law}:
	
	\begin{center}
		
		\begin{tikzcd}[sep={17.5mm,between origins}]
			&&&X_N^{n+N-1} \ar[rd, "p^{n+N-1}_N"'] && X_N^{n+N} \ar[rrdd, two heads, "\prod_{l=2}^N p^{n+N}_l"] & \\
			&&&& X^{n+N-1}_{N-1} \ar[ru, "i^{n+N-1}_{N-1}"'] \ar[rrdd, two heads, "\prod_{l=1}^{N-1} p^{n+N-1}_l"]&\\ 
			& X^n_1 \ar[rruu, "\prod_{l=N-1}^{1} i^{n-1+l}_l"] \ar[rd, two heads, "p^n_1"] &&&&&& X^{n+N}_{1}\\
			X^{n-1} \ar[ru, "i^{n-1}_0", dashed] \ar[rr, "d^{n-1}"] && X^{n} \ar[rrrr, "d^{\{N-1\}}"] \ar[rruu, "\prod_{l=N-2}^{0} i^{n+l}_l"]&&&& X^{n+N-1} \ar[ru, "i^{n+N-1}_{0}"'] &
		\end{tikzcd}
	\end{center}
	
	Then the morphism $i^{n-1}_0$ exists due to \Cref{lem: help-resolution}. The squares $(Y^{n}_r)$ can now be defined by successive pullbacks for increasing $r = 1,\dots, N-1$. To complete the induction step, we pick an admissible epic $p^{n-1}_N\colon X_N^{n-1} \twoheadrightarrow X^{n-1}_{N-1}$ with $X_N^{n-1} \in \Prj(\E)$. \qedhere
\end{proofof}

\begin{rmk}
	Let $X$ be a (projectively) resolving $N$-array over an exact category $\E$. Pick $n \in \ZZ$ and $r \in \{1,\dots,N-1\}$. Then suitably composing morphisms in $X$ yields a (projectively) resolving 2-array of $\gamma^n_r(X_0^\bullet)$, as constructed by Keller:
	\begin{center}
		\begin{tikzcd}[sep={12.5mm,between origins}]
			\cdots \ar[rr] && X_N^{n-N+r} \ar[rd, two heads] \ar[rr, "d^{\{N-r\}}_{X_N}"] && X_N^n \ar[rd, two heads] \ar[rr, "d^{\{r\}}_{X_N}"] && X_N^{n+r} \ar[rd, two heads] \ar[rr, "d^{\{N-r\}}_{X_N}"] && X_N^{n+N} \ar[rr] && \cdots \\
			& \cdots && X^{n-N+r}_r \ar[ru] \ar[rd, two heads] && X^{n}_{N-r} \ar[ru] \ar[rd, two heads] && X^{n+r}_r \ar[ru] \ar[rd, two heads]   && \cdots \\
			\cdots\ar[rr]  && X_0^{n-N} \ar[ru] \ar[rr, "d_{X^\bullet_0}^{\{r\}}"] && X_0^{n-N+r} \ar[ru] \ar[rr, "d_{X^\bullet_0}^{\{N-r\}}"] && X_0^{n} \ar[ru] \ar[rr, "d_{X^\bullet_0}^{\{r\}}"] && X_0^{n+r} \ar[rr] && \cdots
		\end{tikzcd}
	\end{center}
\end{rmk}

\begin{rmk} \label{rmk: array-lift}
	Let $X^\bullet_\bullet$ and $Y^\bullet_\bullet$ be resolving $N$-arrays of $N$-complexes $X$ and $Y$, respectively, over an exact category $\E$. Suppose $X^\bullet_\bullet$ is projectively resolving and one of $X^\bullet_\bullet$ and $Y^\bullet_\bullet$ is bounded above. Then any morphism $X \to Y$ in $\C_N(\E)$ lifts to a morphism $X^\bullet_\bullet \to Y^\bullet_\bullet$, which in turn induces a morphism $X^\bullet_N \to Y^\bullet_N$ in $\C_N(\E)$: Due to boundedness, the morphism $X^k_r \to Y^k_r$ must be zero, for any sufficiently large $k \in \ZZ$. The remaining morphisms exist exist due to functoriality of pullbacks for $r \in \{1, \dots, N-1\}$, and due to projectivity for $r=N$.
\end{rmk}

Combining the proof of \Cref{prp: resolution-exist} with \Cref{prp: Buehler2.15,prp: Buehler2.12} yields

\begin{cor} \label{cor: syz-resolution} Let $\E$ be an exact category with enough projectives and $X \in \mMor_{N-2}(\E)$. Then $\iota^0 X$, see \Cref{ntn: iota}, admits a projectively resolving $N$-array which can be modified into a one-sided $N$-acyclic array as follows:
		\begin{center}
			\pushQED{\qed}
			\begin{tikzcd}[sep={14mm,between origins}]
				\cdots \ar[r] & X_N^{-N+1} \ar[rr] \ar[rd, two heads] && X_N^{-N+2} \ar[rr] \ar[rd, two heads] && X_N^{-N+3} \ar[r] & \cdots \ar[r] & X_N^{-1} \ar[rr]\ar[rd, two heads]  && X_N^{0} \ar[rd, two heads] && \\
				&& X^{-N+1}_{N-1} \ar[rd, two heads] \ar[ru, tail] && X^{-N+2}_{N-1} \ar[rd, two heads] \ar[ru, tail] && && X^{-1}_{N-1} \ar[rd, two heads] \ar[ru, tail] && X^{N-1} \\
				&\udots \ar[ru, tail] && X^{-N+1}_{N-2} \ar[rd, two heads] \ar[ru, tail] && \ddots \ar[rd, two heads] && \udots \ar[ru, tail] && X^{N-2} \ar[ru, tail] \\
				&& \udots \ar[ru, tail] && \ddots\ar[rd, two heads]  && X^{-N+2}_2\ar[rd, two heads] \ar[ru, tail] && \udots \ar[ru, tail] \\
				&&&&& X^{-N+1}_1 \ar[rd, two heads] \ar[ru, tail] && X^1 \ar[ru, tail] \ar[rd, two heads] \\
				&&&&0 \ar[ru, tail] && 0 \ar[ru, tail] && 0
			\end{tikzcd}
		\end{center}
		In particular, $X$ can be recovered as a cokernel in $\Mor_{N-2}(\E)$ of the morphism
		
		\begin{center}
			\begin{tikzcd}[sep={17.5mm,between origins}]
				X^{-N+1}_N \ar[d] \ar[r, equal]& X^{-N+1}_N \ar[d] \ar[r, equal] & \cdots \ar[r, equal] &  X^{-N+1}_N \ar[d] \ar[r, equal] & X^{-N+1}_N \ar[d] \\  X^{-N+2}_N \ar[r] & X^{-N+3}_N \ar[r] & \cdots \ar[r] & X^{-1}_N \ar[r] & X^{0}_N.
			\end{tikzcd}
		\end{center}
		The obvious dual statements hold for elements of $\eMor_{N-2}(\E).$\qed
\end{cor}

\begin{rmk} \label{rmk: mMor-lift} In the situation of \Cref{cor: syz-resolution}, the induced morphism $X_N^\bullet \to Y_N^\bullet$ in \Cref{rmk: array-lift} can be considered as a lift of a morphism $X \to Y$ in $\mMor_{N-2}(\E)$.
\end{rmk}

\begin{prp} \label{lem: lifting-bijection} 
	Let $\E$ be an exact category with enough projectives. For each $X \in \mMor_{N-2}(\E)$, choose an $N$-complex $X^\bullet_N \in \C^{-}_N(\Prj(\E))$ as in \Cref{cor: syz-resolution}. Then morphisms lift uniquely up to homotopy from $\mMor_{N-2}(\E)$ to $\C^-_N(\Prj(\E))$, see \Cref{rmk: mMor-lift}. This defines a fully faithful functor $\rho: \mMor_{N-2}(\E) \to \K^{-}_N(\Prj(\E))$, see \Cref{thm: K-homotopy}.
\end{prp}

\begin{proof}
	Let $f\colon X \to Y$ be a morphism in $\mMor_{N-2}(\E)$. Denote by $X^\bullet_\bullet$ and $Y^\bullet_\bullet=(Y^\bullet_\bullet, q^\bullet_\bullet, j^\bullet_\bullet)$ the respective projectively resolving admissible $N$-arrays from \Cref{cor: syz-resolution} with $X^\bullet_N = \rho(X) =: P$ and $Y^\bullet_N = \rho(Y) =: Q$. Consider a lift $g\colon P \to Q$ of $f$ in $\C_N(\E)$. To prove its uniqueness up to homotopy, we suppose that $f=0$ and show that $g$ is null-homotopic. To this end, set $h^k := 0$ for $k > 0$, fix $n \leq 0$, and assume for all $k>n$ that $h^k\colon P^k \to Q^{k-N+1}$ is already defined such that $g^k = \sum_{r=0}^{N-1} d_Q^{\{N-r-1\}} h^{k+r} d_P^{\{r\}}$.\\
	To construct $h^n$ set $\tilde g^n := g^n - \sum_{r=1}^{N-1} d_Q^{\{N-r-1\}} h^{n+r} d_P^{\{r\}}$. We claim that $q^n_N \tilde g^n =0$. If $n=0$, then $q^n_N \tilde g^n = q^n_N g^n = f^{N-1} p^n_N = 0$ as $f=0$. Otherwise, $j^{n}_{N-1}$ is a monic, and we verify the equivalent claim $d_Q \tilde g^n = j^{n}_{N-1} q^n_N \tilde g^n = 0$:
	
	\begin{align*}
		d_Q \left( \sum_{r=1}^{N-1} d_Q^{\{N-r-1\}} h^{n+r} d_P^{\{r\}} \right) &= \sum_{r=1}^{N-1} d_Q^{\{N-r\}} h^{n+r} d_P^{\{r\}} 	
		= \left(\sum_{r=0}^{N-2} d_Q^{\{N-r-1\}} h^{(n+1)+r} d_P^{\{r\}} \right) d_P \\
		&= \left(g^{n+1} - h^{r+N}d_P^{\{N-1\}} \right) d_P = g^{n+1} d_P = d_Q g^n
	\end{align*} 
	
	It follows that $\tilde g^n$ factors through the kernel $j^{n-1}_{N-1} \cdots j^{n-N+1}_1$ of $q^n_N$, that is, $\tilde g^n = j^{n-1}_{N-1} \cdots j^{n-N+1}_1 \tilde h^n$ for some $\tilde h^n\colon P^n \to Y^{n-N+1}_1$. Projectivity of $P^n$ yields  a lift $h^n\colon P^n \to Q^{n-N+1}$ such that $q^{n-N+1}_2 \cdots q^{n-N+1}_N h^n =\tilde h^n$. In conclusion, $g$ is null-homotopic as desired:
	
	\begin{align*}
		g^n & = \tilde g^n + \sum_{r=1}^{N-1} d_Q^{\{N-r-1\}} h^{n+r} d_P^{\{r\}} =  j^{n-1}_{N-1} \cdots j^{n-N+1}_1 q^{n-N+1}_2 \cdots q^{n-N+1}_N h^n + \sum_{r=1}^{N-1} d_Q^{\{N-r-1\}} h^{n+r} d_P^{\{r\}} \\
		&= d_Q^{\{N-1\}}h^n + \sum_{r=1}^{N-1} d_Q^{\{N-r-1\}} h^{n+r} d_P^{\{r\}} = \sum_{r=0}^{N-1} d_Q^{\{N-r-1\}} h^{n+r} d_P^{\{r\}}
	\end{align*}
	
	Therefore, up to homotopy the unique lift of the sum of two morphisms equals the sum of the respective unique lifts, and the unique lift of a composition equals the composition of the unique lifts. This makes $\rho$ a full (additive) functor. For faithfulness, suppose that $g$ factors through $I(P)$, see \Cref{rmk: factor}.\ref{rmk: factor-proj}. Due to \Cref{cor: syz-resolution}, $f^r\colon X^r \to Y^r$ factors through the cokernel of  $d_{I(P)}^{\{r\}}: I(P)^{-N+1} \to I(P)^{r-N+1}$. This cokernel is zero, since $P^k=0$ for $k > 0$, see \Cref{con: I-and-P}.
\end{proof}

\begin{dfn}
	Let $f\colon X \to Y$ be a morphism of $N$-complexes over an exact category $\E$. We call both $f$ and $X$ a \textbf{resolution} (of $Y$) if the cone $C(f)$ is $N$-acyclic, and \textbf{projective} if $X \in \C_N(\Prj(\E))$. We call $f$ and $Y$ a \textbf{coresolution} (of $X$) if the cocone $C^\ast(f)$ is $N$-acyclic, and \textbf{injective} if $Y \in \C_N(\Inj(\E))$. 
\end{dfn}

\begin{prp} \label{prp: cone-resolution}
	Let $(X, p, i)$ be a resolving $N$-array over an exact category $\E$. Then $p^{\{N\}}\colon X^\bullet_N \to X^\bullet_0$ is a resolution, see \Cref{rmk: array-complex}. The acyclic $N$-array $(C, q, j)$ of its cone $ C(p^{\{N\}})$ is given by $C^n_r := X^{n+N-r}_{N-r} \oplus \bigoplus_{k=N-r+1}^{N-1} X_N^{n+k}$,
	\[ q^n_r = \left( \begin{array}{c|c|c}
		i & p^{\{r-1\}} & 0 \\ \hline 0 & 0 & E_{r-2}
	\end{array} \right)\colon C^n_{r} \twoheadrightarrow C^{n}_{r-1}, \; j^n_r = \left( \begin{array}{c|ccc} p & 0 & \cdots & 0 \\ \hline 0 && E_{r-1} \\ \hline -i^{\{r\}} & -d_{X_N^\bullet}^{\{r-1\}} & \cdots & -d_{X_N^\bullet} \end{array} \right)\colon C^n_{r} \rightarrowtail C^{n+1}_{r+1},\]
	written in the spirit of \Cref{ntn: C,ntn: d^r}.
\end{prp}

\begin{proof}
	Note that $C^n_{N} = C(p^{\{N\}})^n$ for all $n$, see \Cref{con: Sigma-explicit}.\ref{con: Sigma-explicit-formula}. For the commutativity of the squares
	\[
	\begin{tikzcd}[sep={15mm,between origins}]
		& C^{n+1}_{r+1} \ar[rd, "q^{n+1}_{r+1}"] & \\
		C^{n}_{r} \ar[ru,  "j^{n}_r"] \ar[rd, "q^{n}_{r}"] && C^{n+1}_{r}  \\
		& C^{n}_{r-1}, \ar[ru,  "j^{n}_{r-1}"]
	\end{tikzcd}
	\]
	
	we verify the equation $q^{n+1}_{r+1} j^{n}_r = j^{n}_{r-1} q^{n}_r$ in matrix form for all $n$ and $r$, using the commutativity of the $N$-array $X$ and $d := d_{X^\bullet_N}=ip$, see \Cref{rmk: array-complex}:
	
	\begin{align*}
		\left( \begin{array}{c|c|ccc}
			ip & p^{\{r\}} & 0 & \cdots & 0 \\ \hline
			0 & 0  \\
			\vdots & \vdots & & E_{r-2} \\
			0 & 0  \\ \hline
			-i^{\{r\}} & -d^{\{r-1\}} & -d^{\{r-2\}} & \cdots & -d
		\end{array} \right)  = \left( \begin{array}{c|c|ccc}
			pi & p^{\{r\}} & 0 & \cdots & 0 \\ \hline
			0 & 0  \\
			\vdots & \vdots & & E_{r-2} \\
			0 & 0  \\ \hline
			-i^{\{r\}} & -i^{\{r-1\}} p^{\{r-1\}} & -d^{\{r-2\}} & \cdots & -d
		\end{array} \right).
	\end{align*}
	
	For $r=N$, this specializes to $d_{C(p^{\{N\}})}^n = j^{n}_{N-1} q^n_N$, see  \Cref{con: Sigma-explicit}.\ref{con: Sigma-explicit-formula}. Due to \Cref{prp: array-local}.\ref{prp: array-local-admissible}, it remains to prove that the sequence \begin{tikzcd} C^{n-N+r}_{r} \ar[r, tail, "j^{\{N-r\}}"] & C(p^{\{N\}})^n \ar[r, two heads, "q^{\{r\}}"] & C^{n}_{N-r} \end{tikzcd}
	in $\E$ is short exact for all $n$ and $r$. In explicit terms, this sequence reads
	\begin{center}
		\begin{tikzcd}
			X^n_{N-r} \oplus \bigoplus_{k=1}^{r-1} X_N^{n+k} \ar[r, "j^{\{N-r\}}"] & X_0^n \oplus \bigoplus_{k=1}^{N-1} X_N^{n+k} \ar[r, "q^{\{r\}}"] & X^{n+r}_r \oplus \bigoplus_{k=r+1}^{N-1} X_N^{n+k}.
		\end{tikzcd}
	\end{center}
	
	Using the commutativity of $X$ again, one sees that
	\[ j^{\{N-r\}} = \prod_{k=N-1}^r j^{n-N+k}_k =  \left( \begin{array}{c|ccc} p^{\{N-r\}} & 0 & \cdots & 0 \\ \hline
		0 \\
		\vdots && E_{r-1} \\
		0 \\ \hline
		-i^{\{r\}} & -d^{\{r-1\}} & \cdots & -d \\ \hline
		0 & 0 & \cdots & 0 \\
		\vdots & \vdots & \ddots & \vdots \\
		0 & 0 &  \cdots & 0
	\end{array} \right).\]
	
	Indeed, by induction, we compute
	\begin{align*}
		j^{\{N-r\}} &= j^{\{N-r-1\}} j^{n-N+r}_r = \left( \begin{array}{c|ccc} p^{\{N-r-1\}} & 0 & \cdots & 0 \\ \hline
			0 \\
			\vdots && E_{r} \\
			0 \\ \hline
			-i^{\{r+1\}} & -d^{\{r\}} & \cdots & -d \\ \hline
			0 & 0 & \cdots & 0 \\
			\vdots & \vdots & \ddots & \vdots \\
			0 & 0 &  \cdots & 0
		\end{array} \right) \left( \begin{array}{c|ccc} p & 0 & \cdots & 0 \\ \hline 0 && E_{r-1} \\ \hline -i^{\{r\}} & -d^{\{r-1\}} & \cdots & -d \end{array} \right) \\
		&= \left( \begin{array}{c|ccc} p^{\{N-r\}} & 0 & \cdots & 0 \\ \hline
			0 \\
			\vdots && E_{r-1} \\
			0 \\ \hline
			-i^{\{r\}} & -d^{\{r-1\}} & \cdots & -d \\ \hline
			-i^{\{r+1\}}p + di^{\{r\}} & -d^{\{r\}} + d d^{\{r-1\}} & \cdots & -d^{\{2\}} + d d\\ \hline
			0 & 0 & \cdots & 0 \\
			\vdots & \vdots & \ddots & \vdots \\
			0 & 0 &  \cdots & 0
		\end{array} \right) = \left( \begin{array}{c|ccc} p^{\{N-r\}} & 0 & \cdots & 0 \\ \hline
			0 \\
			\vdots && E_{r-1} \\
			0 \\ \hline
			-i^{\{r\}} & -d^{\{r-1\}} & \cdots & -d \\ \hline
			0 & 0 & \cdots & 0 \\
			\vdots & \vdots & \ddots & \vdots \\
			0 & 0 &  \cdots & 0
		\end{array} \right),
	\end{align*}
	since $di^{\{r\}} = ipi^{\{r\}} = i^{\{r+1\}}p$. Similarly, we obtain
	\[q^{\{r\}} = \prod_{k=r-1}^0 q^{n}_{N-k} = \left( \begin{array}{c|c|ccc|c|ccc}
		i^{\{r\}} & p^{\{N-r\}}d^{\{r-1\}} & p^{\{N-r\}}d^{\{r-2\}} & \cdots & p^{\{N-r\}}d & p^{\{N-r\}} & 0 & \cdots & 0 \\ \hline
		0 & 0 & 0 & \cdots & 0 & 0 \\
		\vdots & \vdots & \vdots & \ddots & \vdots & \vdots && E_{N-r-1} \\
		0 & 0 & 0 & \cdots & 0 & 0
	\end{array} \right).\]
	
	The context for the morphisms in the these matrices is provided by the following excerpt of $X$, see \Cref{lem: pasting-law}:
	\begin{center}
		\begin{tikzcd}[sep={17.5mm,between origins}]
			X_N^n \ar[rd, "p^{\{r\}}", two heads] \ar[r, "d"] & \cdots \ar[r, "d"] & X_N^{n+r} \ar[rd, "p^{\{N-r\}}", two heads] \ar[r, "d"] & \cdots \ar[r, "d"] & X_N^{n+N} \\
			& X^n_{N-r} \ar[rd, "p^{\{N-r\}}", two heads] \ar[ru, "i^{\{r\}}"] &\diamond& X^{n+r}_r \ar[ru, "i^{\{N-r\}}"] \\
			&& X_0^n \ar[ru, "i^{\{r\}}"]
		\end{tikzcd}
	\end{center}
	By \Cref{prp: Buehler2.12}.\ref{prp: Buehler2.12-pull}, $a' := \begin{pmatrix} p^{\{N-r\}} \\ -i^{\{r\}} \end{pmatrix}$ and $a := \begin{pmatrix} i^{\{r\}} & p^{\{N-r\}} \end{pmatrix}  $ define a short exact sequence $(a', a)$ with
	\[-ab = \begin{pmatrix}
		p^{\{N-r\}}d^{\{r-1\}} & p^{\{N-r\}}d^{\{r-2\}} & \cdots & p^{\{N-r\}}d
	\end{pmatrix} \hspace{5mm} \text{ for } \hspace{5mm} b :=\begin{pmatrix} 0 & \cdots & 0 \\ -d^{\{r-1\}} & \cdots & -d \end{pmatrix}.\]
	Now move the $(r+1)$st row of $j^{\{N-r\}}$ and the $(r+1)$st column of $q^{\{r\}}$ to the respective second position. Then applying \Cref{lem: gauss}.\ref{lem: gauss-a}, followed by \Cref{lem: gauss}.\ref{lem: gauss-b} with $c=0$, yields the claim.
\end{proof}

Combining \Cref{prp: resolution-exist,prp: cone-resolution} yields

\begin{cor} \label{cor: resolution-exist}
	Every $N$-complex $X \in C^-_N(\E)$ over an exact category $\E$ admits a projective resolution $P \to X$ with $P \in C^-_N(\Prj(\E))$.  \qed
\end{cor}

\subsection{$N$-syzygies} \label{subsection: N-syzygies}

Syzygies of 2-complexes consist of single objects of the base exact category $\E = \mMor_{0}(\E)$. It seems natural to define the syzygy $\Omega^n X$ of an $N$-complex $X$ at position $n$ as an element of $\mMor_{N-2}(\E)$. In this subsection, we prove that these $N$-syzygies induce a triangle equivalence $\underline \Omega^n\colon \underline \APC_N(\F) \to \underline \mMor_{N-2}(\F)$ of stable categories, for any Frobenius category $\F$. The required essential surjectivity is given by a complete resolution. This is constructed by patching the one-sided $N$-array from \Cref{cor: syz-resolution} together with a dual one-sided $N$-array.

\smallskip

Brightbill and Miemietz {\cite{BM24}} establish this equivalence in a setup, where $\F$ is a Frobenius subcategory of an Abelian category $\E$ with $\Prj(\F)=\Prj(\E)$, see \Cref{asn: setting}. Propositions \ref{prp: Omega-exact}, \ref{prp: Omega-surj-full}, and \Cref{thm: Omega-triangle-equiv} correspond to {\cite[Props.~4.7, 4.3, 4.5, Thm.~4.12]{BM24}}.

\begin{dfn} \label{dfn: Omega}
	Let $X \in \C_N(\E)$ be an $N$-complex over an exact category $\E$ and $n \in \ZZ$. Suppose that $X$ is acyclic at positions $n-N+1, \dots, n-1$. We define the \textbf{syzygy of $X$ at position $n$} as the unique object $\Omega^n X := \Omega_\E^n X \in \mMor_{N-2}(\E)$ with $\iota^{n-1}\Omega^n X = \tau^{\leq n-1} \sigma^{\geq n-N+1} X$, see \Cref{dfn: soft-trunc}.\ref{dfn: soft-trunc-left}.
	In explicit terms, this reads
	\begin{center}
		\begin{tikzcd}
			\Omega^n X\colon \; X^{n-N+1}_{1} \ar[r, tail] & X^{n-N+2}_{2} \ar[r, tail] & \cdots \ar[r, tail] & X^{n-1}_{N-1},
		\end{tikzcd}
	\end{center}
	where $X^k_r = C^k_{(r)}(X)$ for all occurring indices $k$ and $r$. Recall that $X^k_r= Z^{n}_{(r)}(X)$ for $r \in \{1, \dots, N-1\}$ if $X$ is acyclic at position $n$.  This gives rise to a functor $\C^{\infty, \varnothing}_N(\E) \to \mMor_{N-2}(\E).$
\end{dfn}

 To show exactness of $\Omega^n$, the proof of {\cite[Prop.~4.7]{BM24}} uses the snake lemma in an Abelian category $\E$, which is not available in general exact categories. However, \Cref{lem: Zn-exact} serves as a replacement.

\begin{prp} \label{prp: Omega-exact}
	Let $\E$ be an exact category. The functor $\Omega^n: \C^{\infty, \varnothing}_N(\E) \to \mMor_{N-2}(\E)$ is exact for all $n \in \ZZ$. 
\end{prp}

\begin{proof}
	Due to the termwise exact structure of $\mMor_{N-2}(\E)$, exactness of $\Omega^n$ is equivalent to exactness of $C^{n-r}_{(N-r)}$, for all $r \in \{1, \dots, N-1\}$. Hence, the claim follows from \Cref{lem: Zn-exact}
\end{proof}

\begin{dfn}
	Let $\E$ be an exact category and $X \in \mMor_{N-2}(\E)$. A \textbf{complete resolution} of $X$ is an object $P \in \TAPC_N(\E)$ with $\Omega^1(P)=X$.
\end{dfn}

Over a Frobenius category $\F$, $\TAPC_N(\F)=\APC_N(\F)$, see \Cref{ntn: Verdier-notation}, and \Cref{prp: Omega-surj-full} establishes the existence of complete resolutions. The use of \Cref{prp: Buehler2.15} in \Cref{cor: syz-resolution} allows us to extend the arguments of {\cite[Prop.~4.5]{BM24}} from the case where $\E$ is Abelian.

\begin{prp} \label{prp: Omega-surj-full}
	If $\F$ is a Frobenius category, then $\Omega^n$ restricts to an essentially surjective and full functor $\APC_N(\F) \to \mMor_{N-2}(\F)$, for all $n \in \ZZ$.
\end{prp}

\begin{proof} We may assume that $n = 1$. For essential surjectivity, consider an arbitrary object
	\begin{center}
		\begin{tikzcd}
			X^1 \ar[r, tail] & X^2 \ar[r, tail] & \cdots \ar[r, tail] & X^{N-2} \ar[r, tail] & X^{N-1}
		\end{tikzcd}
	\end{center}
	of $\mMor_{N-2}(\F)$. In view of \Cref{thm: array}, it suffices to construct an acyclic $N$-array $X^\bullet_\bullet \in A_N(\F)$ with $X^k_N \in \Prj(\F)=\Inj(\F)$, for all $k \in \ZZ$, and $X^{-N+r+1}_{r} = X^r$, for all $r \in \{1,\dots,N-1\}$. \Cref{cor: syz-resolution} yields the \enquote{left half} of the desired array. To complete the array we apply the dual construction to the element
	\begin{center}
		\begin{tikzcd}
			X^{-N+2}_{N-1} \ar[r, two heads] & X^{-N+2}_{N-2} \ar[r, two heads] & \cdots \ar[r, two heads] & X^{-N+2}_2 \ar[r, two heads] & X^1
		\end{tikzcd}
	\end{center}

	of $\eMor_{N-2}(\F)$ using the already constructed bicartesian squares and projective-injectives. In view of \ref{thm: array}, fullness follows from \Cref{rmk: array-lift} and its dual.
\end{proof}

\begin{rmk} \label{rmk: Omega-mu}
	For an object $A$ of an exact category $\E$,
	\[\Omega^n \mu_N^s(A) = \begin{cases} 0, & \text{ if } n > s \text{ or } n \leq s-N+1, \\ \mu_{N+n-s-1}(A), & \text{ otherwise,} \end{cases}  \]
	which is an object of $ \smMor_{N-2}(\E)$, see \Cref{rmk: kernel-mu} and \Cref{ntn: mu-mMor}.
\end{rmk}

\begin{lem} \label{lem: Omega-stab} Let $\E$ be an exact category and $Q \in \APC_N(\E)$. Then $\Omega^n P(Q)\in \Prj(\mMor_{N-2}(\E))$ for all $n \in \ZZ$, see \Cref{con: I-and-P}. In particular, $\Omega^n = \Omega^n_\E$ induces a unique functor $\underline \Omega^n = \underline \Omega_\E^n$ of stable categories:
	\begin{center}
		\begin{tikzcd}[column sep={30mm,between origins}, row sep={17.5mm,between origins}]
			\APC_N(\E) \ar[r, "\Omega^n_\E"] \ar[d] & \mMor_{N-2}(\E) \ar[d] \\
			\underline \APC_N(\E) \ar[r, dashed, "\underline \Omega^n_\E"] & \underline \mMor_{N-2}(\E)
		\end{tikzcd}
	\end{center}
\end{lem}

\begin{proof}
	Using \Cref{rmk: Omega-mu}, we find that $\Omega^n P(Q)  \in \smMor_{N-2}(\Prj(\E)) =\Prj\left(\mMor_{N-2}(\E)\right),$ see \Cref{con: I-and-P} and \Cref{thm: mMor}.\ref{thm: mMor-Proj}. The particular claim follows using \Cref{rmk: factor}.\ref{rmk: factor-proj} and \Cref{con: I-and-P}.
\end{proof}

Brightbill and Miemietz show in {\cite[\S 4.2]{BM24}} that $\underline \Omega^n_\F$ is faithful in their setup. They construct a null-homotopy for a morphism $f: P \to Q$ in $\APC_N(\F)$ by lifting a factorization of $\Omega^n(f)$ through an element of $\Prj(\mMor_{N-2}(\F))$. The quotient objects in their proof, formed in the ambient Abelian category, occur in the acyclic $N$-arrays of $P$ and $Q$, see \Cref{thm: array}. So, their argument works almost verbatim in our setup Combining this with \Cref{prp: Omega-surj-full,prp: Omega-exact}, \Cref{lem: Omega-stab} and \Cref{prp: stable-functor-triang} yields

\begin{thm} \label{thm: Omega-triangle-equiv} If $\F$ is Frobenius category, then the functor $\underline \Omega^n\colon \underline \APC_N(\F) \to \underline \mMor_{N-2}(\F)$ is a triangle equivalence for all $n \in \ZZ$.  \hfill $\square$
\end{thm}

\section{Stabilized $N$-derived categories}

In this section, we complete the diagram in \Cref{thm: buchweitz-intro} by the (stabilized) $N$-derived category and establish the two remaining (stabilized) functors $\underline \iota^0$ and $\underline \tau^{\leq 0}$. Along the way we provide analogues of known results on 2-derived categories for general $N$. The proof that the stabilized truncation $\underline \tau^{\leq 0}$ is a triangle equivalence occupies a substantial part of this section.

\subsection{$N$-derived categories} \label{subsection: N-derived}

In this subsection, we consider the $N$-derived category $\D_N(\E)$ of an exact idempotent complete category $\E$ and its subcategories $\D^{\#, \natural}_N(\E)$ given by boundedness conditions. These are defined as Verdier quotients of the corresponding homotopy categories $\K^{\#, \natural}_N(\E)$ by their respective triangulated subcategory $\K^{\#, \varnothing}_N(\E)$ of acyclic $N$-complexes. We establish several fundamental properties known from the classical case, where $N=2$ and $\E$ is Abelian. The case $N=2$ was considered by Keller {\cite{Kel96}}, the Abelian case by Iyama, Kato and Miyachi {\cite{IKM17}}. Like Keller, we impose idempotent completeness of $\E$ to ensure that $N$-acyclicity is preserved under homotopy equivalence. Notably, we obtain canonical inclusions $\mMor_{N-2}(\E) \subseteq \D^b_N(\E) \subseteq \D^\#_N(\E) \subseteq \D_N(\E)$ as  triangulated subcategories and triangle equivalences $\D^-_N(\E) \simeq \K^-_N(\Prj(\E))$ and $\D^b_N(\E) \simeq \D^{-, b}_N(\E)\simeq \K^{-, b_\E}_N(\Prj(\E))$. Finally, we relate syzygies to truncations of $N$-complexes in $\D^{-, b}_N(\E)$, a key ingredient for the proof of \Cref{thm: buchweitz-intro}.

\smallskip

Bühler proves \Cref{prp: idemp-compl}, proposed by Keller, for $N=2$, see {\cite[Prop.~10.9]{Buh10}}. We adapt his arguments for general $N$. 

\smallskip

\begin{prp} \label{prp: idemp-compl} The following conditions are equivalent for any exact category $\E$:
	\begin{enumerate}[label=(\arabic*)]
		\item \label{prp: idemp-compl-1} Every null-homotopic $N$-complex in $\C_N(\E)$ is acyclic.
		\item \label{prp: idemp-compl-2} The category $\E$ is idempotent complete.
		\item \label{prp: idemp-compl-3} If a morphism $X \to Y$ in $\K_N(\E)$ admits a left-inverse and $Y$ is acyclic at position $n \in \ZZ$, then so is $X$.
	\end{enumerate}
	In particular, the full, but not essential, image of $\K^{\infty, \varnothing}_N(\E)$ in $\K_N(\E)$ is thick if and only if $\E$ is idempotent complete.
\end{prp}

\begin{proof} \
	
	\begin{itemize}[leftmargin=16mm]
		\item[\ref{prp: idemp-compl-3} $\Rightarrow$ \ref{prp: idemp-compl-1}] If $X \in \C_N(\E)$ is null-homotopic, then there is an isomorphism $X \to 0$ in $\K_N(\E)$, see \Cref{dfn: null-homotopic}. As the zero $N$-complex is acyclic, \ref{prp: idemp-compl-3} implies that $X$ is so as well.
		
		\item[\ref{prp: idemp-compl-1} $\Rightarrow$ \ref{prp: idemp-compl-2}] For an idempotent $e\colon A \to A$ in $\E$, consider the $N$-complex
		\begin{center}
			\begin{tikzcd}[sep=6mm]
				X\colon \; \cdots \ar[r, "\id"] & X^0 \ar[r, "e"] & X^1 \ar[r, "1-e"] & X^2 \ar[r, "\id"] & \cdots \ar[r, "\id"] & X^N \ar[r, "e"] & X^{N+1} \ar[r, "1-e"] & X^{N+2} \ar[r, "\id"] & \cdots
			\end{tikzcd}
		\end{center}
		where $X^i = A$ for all $i \in \ZZ$. One can see that the morphism $\id_{X}$ is null-homotopic with homotopy $h$ defined by $h^i = \id_A$. So, $X$ is acyclic by \ref{prp: idemp-compl-1} and $e$ splits with kernel $Z^0_{(1)}(X)$, see \Cref{dfn: idem-compl}.
		
		\item[\ref{prp: idemp-compl-2} $\Rightarrow$ \ref{prp: idemp-compl-3}] Let $s\colon X \to Y$ be a morphism in $\C_N(\E)$ with left-inverse $r\colon Y \to X$ in $\K_N(\E)$.
		This means that $r s - \id_{X}$ factors in $\C_N(\E)$ through $i = i_X\colon X \rightarrowtail I(X)$, see \Cref{rmk: factor}.\ref{rmk: factor-proj} and \Cref{con: I-and-P}:
		
		\begin{center}
			\begin{tikzcd}[sep={15mm,between origins}]
				X \ar[rr, "r s \, - \, \id_{X}"] \ar[rd, "i", tail] && X \\
				& I(X) \ar[ru, "f"]
			\end{tikzcd}
		\end{center}
		\noindent
		Equivalently, the morphism $\begin{pmatrix} i \\ s \end{pmatrix}\colon X \to I(X) \oplus Y$, has the left-inverse $\begin{pmatrix} - f & r \end{pmatrix}\colon I(X) \oplus Y \to X$ in $\C_N(\E)$.
		Suppose that $Y$ and hence $I(X) \oplus Y$ is acyclic at position $n$, see \Cref{rmk: mu-acyclic}.\ref{rmk: mu-acyclic-statement}  and \Cref{prp: Buehler2.9}. Replacing $s$ by $\begin{pmatrix} i \\ s \end{pmatrix}$ we may assume that $s$ has a left inverse $r\colon Y \to X$ in $\C_N(\E)$. Due to \ref{prp: idemp-compl-2}, \Cref{lem: eX-acyclic} applies to the idempotent $e := sr: Y \to Y$ and shows that  $X \cong_{\C_N(\E)} eY$ is acyclic at position $n$, see \Cref{rmk: eA}.
	\end{itemize}
	
	The particular claim follows since thickness follows from \ref{prp: idemp-compl-3} and implies \ref{prp: idemp-compl-1}.\qedhere
\end{proof}

\begin{lem} \label{lem: eX-acyclic}
	Consider an idempotent $e\colon X \to X$ of an $N$-complex $X$ over an exact idempotent complete category $\E$. If $X$ is acyclic at position $n \in \ZZ$, then so is $eX$, see \Cref{prp: idemp-compl-C-mMor}.\ref{prp: idemp-compl-C}.
\end{lem}

\begin{proof} By applying the contraction functors $\gamma^n_r$ for $r \in \{1,\dots,N-1\}$, we may assume that $N=2$.
	There are induced idempotents $e^n$ on $Z^n=Z^n(X)$ and $C^n = C^n(X)$, see \Cref{rmk: idem-induce}, and the diagram
	
	\begin{center}
		\begin{tikzcd}[sep={15mm,between origins}]
			X^{n-1} \ar[rr, "d^{n-1}", "\circ" marking] \ar[rd, two heads] \ar[ddd, "e^{n-1}"]&& X^n \ar[rr, "d^{n}", "\circ" marking] \ar[ddd, "e^{n}"] \ar[rd, two heads] && X^{n+1} \ar[ddd, "e^{n+1}"]\\
			& Z^n \ar[ru, tail]\ar[d, "e^{n}", dashed] && C^n \ar[ru, tail]\ar[d, "e^{n}", dashed]\\
			& Z^n \ar[rd, tail] && C^n \ar[rd, tail]\\
			X^{n-1} \ar[rr, "d^{n-1}", "\circ" marking] \ar[ru, two heads] && X^n \ar[rr, "d^{n}", "\circ" marking] \ar[ru, two heads] && X^{n+1}
		\end{tikzcd}
	\end{center}
	commutes, see \Cref{rmk: help-commute}. Applying \Cref{lem: idem-ses} and its dual yields a commutative diagram
	
	\begin{center}
		\begin{tikzcd}[column sep={20mm,between origins}, row sep={12.5mm,between origins}]
			&&& e^{n}C^{n} \ar[rd, tail, dashed] \ar[dd, tail, dashed, crossing over] & \\
			e^{n-1}X^{n-1} \ar[rr] \ar[rd, two heads, dashed] \ar[dd, tail] && e^{n}X^{n} \ar[rr, crossing over] \ar[ru, two heads, dashed] \ar[dd, tail] && e^{n+1}X^{n+1} \ar[dd, tail] \\
			& e^{n}Z^{n} \ar[ru, tail, dashed] && C^{n}  \ar[dd, two heads, dashed] \ar[rd, tail] \\
			X^{n-1} \ar[rr] \ar[rd, two heads] \ar[dd, two heads] && X^{n} \ar[rr, crossing over] \ar[ru, two heads] \ar[dd, two heads] && X^{n+1} \ar[dd, two heads]\\
			&  Z^{n} \ar[ru, tail] \ar[uu, <-< , dashed, crossing over] && (1-e^{n}) C^{n}  \ar[rd, tail, dashed]  \\
			(1-e^{n-1})X^{n-1} \ar[rr] \ar[rd, two heads, dashed] && (1-e^{n})X^{n} \ar[rr] \ar[ru, dashed, two heads] && (1-e^{n+1}) X^{n+1}  \\
			& (1-e^{n}) Z^{n} \ar[ru, tail, dashed] \ar[uu, <<- , dashed, crossing over] &&  
		\end{tikzcd}
	\end{center}
	\noindent
	such that all sequences \begin{tikzcd}[cramped] \bullet  \ar[r, tail] & \bullet \ar[r, two heads] & \bullet \end{tikzcd} are short exact. This implies the claim.
\end{proof}

\begin{cor} \label{cor: resolution-cohomology}
	If $f\colon X \to Y$ is a resolution over an exact idempotent complete category $\E$ with $X \in \C^-_N(\Prj(\E))$ and $Y \in \C^{-, b}_N(\E)$, then $X \in \C^{-, b_\E}_N(\Prj(\E))$. 
\end{cor}

\begin{proof}
	As $\K^{-, b}_N(\E)$ is a triangulated subcategory of $\K^-_N(\E)$, see \Cref{thm: Verdier-notation}.\ref{thm: Verdier-notation-c}, and $C(f) \in \C^{-, \varnothing}_N(\E) \subseteq \C^{-, b}_N(\E)$, this implies that $X$ is isomorphic in $\K_N(\E)$ to a complex in $\C^{-, b}_N(\E)$ and hence lies in $\C^{-, b_\E}_N(\Prj(\E))$ by \Cref{prp: idemp-compl}.\ref{prp: idemp-compl-3}.
\end{proof}

\begin{cor} \label{cor: K-cap-K}
	Let $\E$ be an exact idempotent complete category and $\#, \#' \in \in \{\infty,-,b\}$ and $\natural, \natural' \in \{\infty,-,b, \varnothing\}$ with $(\#', \natural') \leq (\#, \natural)$. Then $\K^{\#', \natural'}_N(\E)  = \K^{\#', \natural}_N(\E) \cap \K^{\#, \natural'}_N(\E)$ as subcategories of $\K^{\#, \natural}_N(\E)$.
\end{cor}

\begin{proof}
	Let $X \in \K^{\#', \natural}_N(\E) \cap \K^{\#, \natural'}_N(\E)$. This means that $X$ is isomorphic in $\K_N(\E)$ to both an object $Z \in \C^{\#', \natural}_N(\E)$ and $Y \in \C^{\#, \natural'}_N(\E)$. \Cref{prp: idemp-compl}.\ref{prp: idemp-compl-3} applied to $Z \cong_{\K_N(\E)} Y$ yields that $Z \in \C^{\#', \natural}_N(\E) \cap \C^{\infty, \natural'}_N(\E) = \C^{\#', \natural'}_N(\E)$. So, $X \cong_{\K_N(\E)} Z$ lies in the subcategory $\K^{\#', \natural'}_N(\E)$ of $\K_N(\E)$. The converse inclusion is obvious.
\end{proof}

\begin{dfn} \label{dfn: D}
	The \textbf{$\boldsymbol N$-derived categories} of an exact idempotent complete category $\E$ are defined as the Verdier quotients $\D^{\#, \natural}_N(\E) := \K^{\#, \natural}_N(\E)/\K^{\#, \varnothing}_N(\E)$, for $\# \in \{\infty, +,-,b\}$ and $\natural \in \{\infty, +, -, b\}$. A morphism in $\C_N(\E)$ is called an \textbf{($\boldsymbol N$-)quasi-isomorphism} if it is an isomorphism in $\D_N(\E)$.
\end{dfn}

\begin{lem} \label{lem: D-in-D}
	Let $\E'$ be an exact subcategory of an exact idempotent complete category $\E$. Then there is a canonical triangle functor $\D^{\#, \natural}_N(\E') \to \D^{\#, \natural}_N(\E)$, for $\# \in \{\infty, +,-,b\}$ and $\natural \in \{\infty, +, -, b\}$.
\end{lem}

\begin{proof}
	This follows from the universal property of the Verdier quotient, since the fully faithful functor triangle $\K^{\#, \natural}_N(\E') \to \K^{\#, \natural}_N(\E)$ from \Cref{prp: K-in-K} sends $\K^{\#, \varnothing}_N(\E')$ to $\K^{\#, \varnothing}_N(\E)$.
\end{proof}

Due to the particular statement of \Cref{prp: idemp-compl}, \Cref{prp: quasi-characterization} is a special case of a general fact on Verdier quotients, see {\cite[Prop.~2.1.35]{Nee01}}.

\begin{prp} \label{prp: quasi-characterization}
	Over an exact idempotent complete category, resolutions of $N$-complexes agree with $N$-quasi-isomorphisms. \qed
\end{prp}

Henrard and van Roosmalen {\cite{HR20}} consider 2-complexes over (more general) deflation-exact categories. We reduce the claim of \Cref{lem: termwise-triangle} to their argument by contraction.

\begin{lem} \label{lem: termwise-triangle}
	Over an exact idempotent complete category $\E$, any term\emph{wise} short exact sequence \begin{tikzcd}[cramped]
		X \ar[r, "f"] & Y \ar[r, "g"] & Z
	\end{tikzcd} in $\C_N(\E)$, see \Cref{rmk: C-exact}, induces a distinguished triangle in $\D_N(\E)$:
	\begin{center}
		\begin{tikzcd}[cramped]
			X \ar[r, "f"] & Y \ar[r, "g"] & Z \ar[r] & \Sigma X
		\end{tikzcd}
	\end{center}
	
\end{lem}

\begin{proof}
	The morphism $f$ fits into the distinguished triangle \hyperref[eqn: std-triangle]{$T(f)$}
	\begin{center}
		\begin{tikzcd}
			X \ar[r, "f"] & Y \ar[r] & C(f) \ar[r, "w"] & \Sigma X
		\end{tikzcd}
	\end{center}
	in $\K_N(\E)$ and hence in $\D_N(\E)$, see \Cref{con: cone}. Since $g f = 0$, there is a morphism
	\begin{align} \label{diag: HR20}
		\begin{tikzcd}[row sep={17.5mm,between origins}, ampersand replacement=\&]
			X \ar[r, tail, "\begin{pmatrix}
				i_X \\ f
			\end{pmatrix}"] \ar[d, "\id_X"] \& I(X) \oplus Y \ar[r, two heads] \ar[two heads]{d}{\begin{pmatrix} 0 & \id_Y \end{pmatrix}} \& C(f) \ar[d, dashed, "h_f"]\\
			X \ar[r, "f"] \& Y \ar[r, "g"] \& Z
		\end{tikzcd}
	\end{align}
	of term\emph{wise} short exact sequences in $\E$, see \Cref{con: cone,con: I-and-P}. We claim that the cone $C(h_f)$ is $N$-acyclic. Then $h_f$ is an $N$-quasi-isomorphism, see \Cref{prp: quasi-characterization}, and the isomorphism
	
	\begin{center}
		\begin{tikzcd}[sep={17.5mm,between origins}]
			X \ar[r, "f"] \ar[d, "\id"] & Y \ar[r] \ar[d, "\id"] & C(f) \ar[r, "w"] \ar[d, "h_f"] & \Sigma X \ar[d, "\id"] \\
			X \ar[r, "f"] & Y \ar[r, "g"] & Z \ar[r, dashed, "w \,  \circ \, h^{-1}_f"] & \Sigma X
		\end{tikzcd}
	\end{center}
	of candidate triangles in $\D_N(\E)$ yields the claim. Henrard and van Roosmalen prove this statement in {\cite[Prop.~3.23]{HR20}} for $N=2$. In particular, $C(h_{\gamma f})$ is 2-acyclic, where $\gamma := \gamma^n_r$ for arbitrary $n \in \ZZ$ and $r \in \{1,\dots,N-1\}$. We show that $C(h_{\gamma f}) \cong C(\gamma h_f)$ in $\K_2(\E)$. Then 2-acyclicity of $\gamma C(h_f)$ and thus $N$-acyclicity of $C(h_f)$ follow from \Cref{lem: cone-acyclic-prep}.\ref{lem: cone-acyclic-prep-2} and \Cref{prp: idemp-compl}.\ref{prp: idemp-compl-3}. \\
	To this end, we apply $\gamma$ to \eqref{diag: HR20} and combine the result with \eqref{diag: cone-contract}, see \Cref{con: cone-gamma}. This leads to a commutative diagram
	\begin{center}
		\begin{tikzcd}[sep={17.5mm,between origins}]
			&\gamma X \ar[dd, "\id" near start] \ar[rr, tail] \ar[dl, "\id"] && I(\gamma X) \oplus \gamma Y \ar[dd, two heads] \ar[rr, two heads] \ar[dl, tail] && C(\gamma f) \ar[ld, tail, "c_X"] \ar[dd, "h_{\gamma f}"] \\
			\gamma X \ar[rr, tail, crossing over] \ar[rd, "\id"] && \gamma I(X) \oplus \gamma Y \ar[rr, two heads, crossing over] \ar[rd, two heads] && \gamma C(f) \ar[rd, "\gamma h_f"] \\
			&\gamma X \ar[rr, "\gamma f"] && \gamma Y \ar[rr, "\gamma g"] && \gamma Z
		\end{tikzcd}
	\end{center}
	in $\C_2(\E)$ with termwise short exact rows. Indeed, $h_{\gamma f} = \gamma h_f \circ c_X$ by the universal property of cokernel. The isomorphism $c_X$ in $\K_2(\E)$ from \Cref{lem: cone-acyclic-prep}.\ref{lem: cone-acyclic-prep-iso} then fits into an isomorphism 
	\begin{center}
		\begin{tikzcd}[row sep={17.5mm,between origins}]
			C(\gamma f) \ar[r, "h_{\gamma f}"] \ar[d, "c_X"] & \gamma Z \ar[d, "\id"] \ar[r] & C(h_{\gamma f}) \ar[d, dashed, "\cong"] \ar[r] & \Sigma C(\gamma f) \ar[d, "\Sigma c_X"]\\
			\gamma C(f) \ar[r, "\gamma h_{f}"] & \gamma Z \ar[r] & C(\gamma h_{f}) \ar[r] & \Sigma (\gamma C(f))
		\end{tikzcd}
	\end{center}
	of distinguished triangles in $\K_2(\E)$, which yields the claim.
\end{proof}

We now prove \Cref{thm: diamond,thm: IKM-resolution}.

\begin{thm} \label{thm: D-diamond}
	For an exact idempotent complete category $\E$, there is a diagram of canonical fully faithful, triangle functors and equivalences:
	
	\begin{center}
		\begin{tikzcd}[sep={15mm,between origins}]
			&& \D^+_N(\E) \ar[rd, "\simeq"] && \\
			& \D^{+, b}_N(\E) \ar[ru] \ar[rd, "\simeq"] && \D^{\infty, +}_N(\E) \ar[rd] \\
			 \D^b_N(\E) \ar[ru, "\simeq"] && \D^{\infty, b}_N(\E) \ar[ru] \ar[rd] && \D_N(\E) \\
			& \D^{-, b}_N(\E) \ar[ru, "\simeq"'] \ar[rd] \ar[lu, <-, "\simeq"] && \D^{\infty, -}_N(\E) \ar[ru] \\
			&& \D^-_N(\E) \ar[ru, "\simeq"']
		\end{tikzcd}
	\end{center}
\end{thm}

\begin{proof}
	By symmetry, it suffices to consider the arrows pointing upwards. Let $\#, \#', \natural, \natural' \in \{\infty,-,b\}$ with $(\#', \natural') \leq (\#, \natural)$. We use \Cref{thm: sod}.\ref{thm: sod-c} to establish the functor $\D^{\#', \natural'}_N(\E) \to \D^{\#, \natural}_N(\E)$. By \Cref{thm: Verdier-notation}.\ref{thm: Verdier-notation-c} and \Cref{cor: K-cap-K},  $\U := \K^{\#', \natural'}_N(\E)$, $\V:=\K^{\#, \varnothing}_N(\E)$, and $\U \cap \V= \K^{\#', \varnothing}_N(\E)$ are triangulated subcategories of $\T := \K^{\#, \natural}_N(\E)$. Condition \ref{thm: sod-3} in \Cref{thm: sod} holds trivially if $\#' = \#$, since $\U \cap \V = \V$. For the functor $\D^-_N(\E) \to \D^{\infty, -}_N(\E)$, consider a morphism $X \to Y$ in $\T$ with $X \in \C^-_N(\E)$ and $Y \in \C^{\infty, \varnothing}_N(\E)$. For any sufficiently large $n \in \ZZ$, it factors through the canonical morphism $\sigma^{\leq n} Y \to Y$ with $\sigma^{\leq n} Y \in \C^{-, \varnothing}_N(\E)$ by \Cref{cor: soft-trunc-seq}. To show its essential surjectivity, let $X \in \C^{\infty, -}_N(\E)$. Due to \Cref{cor: soft-trunc-seq} and \Cref{lem: termwise-triangle}, for sufficiently large $n \in \ZZ$, there is a distinguished triangle
	\begin{center}
		\begin{tikzcd}
			\sigma^{\leq n} X \ar[r] & X \ar[r] & \sigma^{\geq n-N+2} X \ar[r] & \Sigma \sigma^{\leq n} X
		\end{tikzcd}
	\end{center}
	in $\D^{\infty, -}_N(\E)$ with $\sigma^{\geq n-N+2} X = 0 \in \D^{-, b}_N(\E)$ and hence $X \cong \sigma^{\leq n} X \in \D^{-}_N(\E)$. The preceding arguments restrict to $\D^{-, b}_N(\E) \to \D^{\infty, b}_N(\E)$ and $\D^{b}_N(\E) \to \D^{+, b}_N(\E)$.
\end{proof}

\begin{lem} \label{lem: Verdier-crit}
	Let $X, Y \in \C_N(\E)$ be $N$-complexes over an exact idempotent complete category $\E$. Then the canonical functor $\K_N(\E) \to \D_N(\E)$ induces an isomorphism $\Hom_{\K_N(\E)}(X, Y) \cong \Hom_{\D_N(\E)}(X, Y)$ of Abelian groups if $X \in \C^{-}_N(\Prj(\E))$ or $Y \in \C^{+}_N(\Inj(\E))$.
\end{lem}

\begin{proof} If $\E$ is Abelian, $\Hom_{\K_N(\E)}(X, \K^{\infty, \varnothing}_N(\E)) = 0$ if $X \in \C^{-}_N(\Prj(\E))$ and $\Hom_{\K_N(\E)}(\K^{\infty, \varnothing}_N(\E), Y) = 0$ if $Y \in \C^{+}_N(\Inj(\E))$ by {\cite[Lem.~3.3]{IKM17}}. However, the proof works in general without the use of homology. Then Verdier's criterion {\cite[Ch.~I, \S 2, n\textsuperscript{o} 5,  5-3 Prop.]{Ver77}} yields the claim.
\end{proof}

\begin{thm} \label{thm: K-sod} Let $\E$ be an exact idempotent complete category with enough projectives.
	\begin{enumerate}
		\item \label{thm: K-sod--} The pair $(\K^{-}_N(\Prj(\E)), \K^{-, \varnothing}_N(\E))$ is a semiorthogonal decomposition of $\K^-_N(\E)$, which gives rise to a triangle equivalence $\D^-_N(\E) \simeq \K^-_N(\Prj(\E))$.
		\item \label{thm: K-sod-b} The pair $(\K^{-, b_\E}_N(\Prj(\E)), \K^{-, \varnothing }_N(\E))$ is a semiorthogonal decomposition of $\K^{-, b}_N(\E)$, which gives rise to a triangle equivalence $\D^{b}_N(\E) \simeq \K^{-, b_\E}_N(\Prj(\E))$.
	\end{enumerate}
\end{thm}

\begin{proof}  \Cref{lem: Verdier-crit} yields condition \ref{dfn: sod-1} in \Cref{dfn: sod}, condition \ref{dfn: sod-2} follows from the standard triangle of a resolution, see \Cref{cor: resolution-exist,cor: resolution-cohomology}. The triangle equivalences are then due to \Cref{prp: sod} using the triangle equivalence $\D^b_N(\E) \simeq \D^{-, b}_N(\E)$ from \Cref{thm: D-diamond}. 
\end{proof}

\begin{rmk} \label{rmk: K-sod}
	In view of \Cref{rmk: Bondal-sod}, the triangle equivalences in \Cref{thm: K-sod} send an $N$-complex $X$ to $P$ for a projective resolution $P \to X$, and lift morphisms.
\end{rmk}

\begin{prp} \label{prp: mMor-Db}
	Let $\E$ be an exact idempotent complete category. Then the composed functor $\mMor_{N-2}(\E) \xrightarrow{\iota^n} \C^b_N(\E) \to \D^b_N(\E)$, see \Cref{ntn: iota}, is fully faithful and also denoted by $\iota^n$.
\end{prp}

\begin{proof}
	We may assume that $n=0$. Due to \Cref{rmk: K-sod}, postcomposing $\iota^0$ with the equivalence $\D^b_N(\E) \to \K^{-, b_\E}_N(\Prj(\E))$ and the fully faithful functor $\K^{-, b_\E}_N(\Prj(\E)) \to \K^{-}_N(\Prj(\E))$, see Theorems \ref{thm: Verdier-notation}.\ref{thm: Verdier-notation-c} and \ref{thm: K-sod}.\ref{thm: K-sod-b}, yields the functor $\rho$ from \Cref{lem: lifting-bijection}. Then $\iota^0$ is fully faithful since $\rho$ is so.
\end{proof}

\begin{dfn} For an additive category $\A$, the \textbf{(left hard) truncation} $\tau^{\leq n}$ at $n \in \ZZ$ is the exact functor $\C_N(\A) \to \C^-_N(\A)$ sending $X \in \C_N(\A)$ to
	\begin{center}
		\begin{tikzcd}
			\tau^{\leq n} X\colon & \cdots \ar[r,"d_X^{n-3}"] & X^{n-2} \ar[r,"d_X^{n-2}"] & X^{n-1} \ar[r,"d_X^{n-1}"] & X^n \ar[r] & 0 \ar[r] & \cdots.
		\end{tikzcd}
	\end{center}
	The \textbf{(right hard) truncation} $\tau^{\geq n} $ is defined analogously.
\end{dfn}

\begin{rmk} \label{rmk: tau-triangle}
	For any $N$-complex $X \in \C_N(\E)$ over an exact category $\E$ and $n \in \ZZ$, there is a termsplit short exact sequence \begin{tikzcd}[cramped] \tau^{\geq n} X \ar[r, tail] & X \ar[r, two heads] & \tau^{\leq n-1} X  \end{tikzcd}. It induces a distinguished triangle
	\begin{center}
		\begin{tikzcd} \tau^{\geq n} X \ar[r] & X \ar[r] & \tau^{\leq n-1} \ar[r] & \Sigma \tau^{\geq n} X \end{tikzcd}
	\end{center}
	in $\K_N(\E)$, see \Cref{lem: ses-triangle}, and hence in $\D_N(\E)$ if $\E$ is idempotent complete.
\end{rmk}

Reversing the construction in \Cref{cor: syz-resolution}, we relate truncations and syzygies of acyclic $N$-complexes. This is generalizes the quasi-isomorphism between an object and its resolution, known from the case $N=2$.

\begin{lem} \label{lem: quasi}
	Let $X \in \C_N(\E)$ be an $N$-complex  over an exact idempotent complete category $\E$. Suppose that $X$ is acyclic at all positions up to $ n \in \ZZ$. Then the canonical morphism 
	\begin{center}
		\begin{tikzcd}[sep={22.5mm,between origins}]
			\tau^{\leq n} X\colon \ar[d] & \cdots \ar[r] &  X^{n-N+1} \ar[r, "d_X"] \ar[d, two heads]& X^{n-N+2 } \ar[r, "d_X"] \ar[d, two heads]&\cdots \ar[r, "d_X"]  & X^{n-1} \ar[r, "d_X"] \ar[d, two heads] & X^{n}  \ar[d, two heads] \\
			\iota^n \Omega^{n+1}X\colon & \cdots \ar[r] &  0 \ar[r] & C^{n-N+2}_{(1)}(X) \ar[r, tail] & \cdots \ar[r, tail] & C^{n-1}_{(N-2)}(X) \ar[r, tail] & C^{n}_{(N-1)}(X)
		\end{tikzcd}
	\end{center}
	is a resolution. In particular, there is an isomorphism of functors $\tau^{\leq n} \cong \iota^n \circ \Omega^{n+1} \colon \APC_N(\E) \to \D^{-, b}_N(\E)$.
\end{lem}

\begin{proof}
	Patch all diagrams that can be obtained from \Cref{prp: array-local} as in the proof of \Cref{thm: array}. Then modify the resulting diagram at the right end as follows, and extend by zero to create a resolving $N$-array $X^\bullet_\bullet$ of $X^\bullet_0 =\iota^n \Omega^{n+1}X$ with $X^\bullet_N = \tau^{\leq n} X$:
	\begin{center}
		\begin{footnotesize}
			\begin{tikzcd}[sep={1.1cm,between origins}]
				\cdots \ar[r] & X^{n-N+2} \ar[rr] \ar[rd, two heads] && X^{n-N+3} \ar[-, r] & \cdots \ar[r] & X^{n-1} \ar[rr]\ar[rd, two heads]  && X^{n} \ar[rd, two heads] && &&& \\
				 \udots  \ar[ru, tail] && C^{n-N+2}_{(N-1)} \ar[rd, two heads] \ar[ru, tail] && && C^{n-1}_{(N-1)} \ar[rd, two heads] \ar[ru, tail] && C^{n}_{(N-1)} \ar[rd, equal] \\
				& \udots \ar[ru, tail]  && \ddots \ar[rd, two heads] && \udots \ar[ru, tail] && C^{n-1}_{(N-2)} \ar[ru, tail] \ar[rd, equal] && C^{n}_{(N-1)} \ar[rd, equal] \\
				&& \ddots\ar[rd, two heads]  && C^{n-N+2}_{(2)}\ar[rd, two heads] \ar[ru, tail] && \udots \ar[ru, tail] && C^{n-1}_{(N-2)} \ar[rd, equal] \ar[ru, tail] && \ddots \ar[rd, equal] \\
				&&& C^{n-N+1}_{(1)} \ar[rd, two heads] \ar[ru, tail] && C^{n-N+2}_{(1)} \ar[ru, tail] \ar[rd, equal] && \udots \ar[ru, tail] && \ddots \ar[rd, equal] && C^{n}_{(N-1)}\ar[rd, equal]\\
				\cdots \ar[rr] && 0 \ar[ru, tail] \ar[rr] && 0 \ar[ru, tail] \ar[rr, tail] && C^{n-N+2}_{(1)} \ar[rr, tail] \ar[ru, tail] && \cdots \ar[rr, tail] && C^{n-1}_{(N-2)} \ar[rr, tail] \ar[ru, tail] && C^n_{(N-1)}
			\end{tikzcd}
		\end{footnotesize}
	\end{center}
	Then \Cref{prp: cone-resolution} and yield the desired resolution. Naturality is clear.
\end{proof}

\subsection{Perfect $N$-complexes} \label{subsection: perfect}

In this subsection, we consider the subcategory $\D^{\perf}_N(\E)$ of perfect $N$-complexes of $\D^b_N(\E)$. We characterize it by means of an Ext condition due to Buchweitz in the classical case. The corresponding Verdier quotient is the $N$-singularity category $\underline \D^b_N(\E)$. We show that its objects are obtained by embedding $\underline \mMor_{N-2}(\E)$ at various positions, generalizing a statement of Orlov's (\Cref{lem: Orlov1.10}).

\smallskip

If $\E$ is idempotent complete, then so are its subcategories $\Prj(\E)$ and $\Inj(\E)$, see \Cref{rmk: Proj-exact}. This allows for the following

\begin{dfn} \label{dfn: perfect}
	 The category of \textbf{perfect} $N$-complexes over an exact idempotent complete category $\E$ is defined as $\D^{\perf}_N(\E) := \D^b_N(\Prj(\E))$.
\end{dfn}
	
\begin{lem} \label{lem: perf-functor} Let $\E$ be an exact idempotent complete category with enough projectives.
	\begin{enumerate}
		\item \label{lem: perf-functor-equiv} The canonical triangle functor $\K^b_N(\Prj(\E)) \to \D^{\perf}_N(\E)$ is an equivalence.
		
		\item \label{lem: perf-functor-embed} The canonical triangle functor $\D^{\perf}_N(\E) \to \D^b_N(\E)$ is fully faithful.
	\end{enumerate}
\end{lem}

\begin{proof} \
	\begin{enumerate}[leftmargin=*]
		\item The claim follows from \Cref{lem: Verdier-crit} applied to $\Prj(\E)=\Prj(\Prj(\E))$.
			
		\item The functor exists due to \Cref{lem: D-in-D}. By precomposition with the equivalence from \ref{lem: perf-functor-equiv} the claim follows from \Cref{lem: Verdier-crit}. \qedhere
	\end{enumerate}
\end{proof}

\begin{rmk} \label{rmk: Perf-Ext} Consider a perfect $N$-complex \begin{tikzcd}[sep=small, cramped] P\colon P^m \ar[r] & \cdots \ar[r] & P^n \end{tikzcd} $\in \D^{\perf}_N(\E)$ over an exact idempotent complete category, then there is a distinguished triangle,
	\begin{center}
		\begin{tikzcd}
			\mu_1^n(P^n) \ar[r] & P \ar[r] & \tau^{\leq n-1} P \ar[r] & \Sigma \mu_1^n(P^n),
		\end{tikzcd}
	\end{center}
	in $\D_N(\E)$. Hence, any such $P$ is an iterated extension of the object $\mu_1^s(P^s)$, where $s \in \{m, \dots, n\}$. In particular, as $\Hom$ is a (co)homological functor,
	\[\Hom_{\D_N(\E)}(-, P) =0 \; \Longleftrightarrow \; \Hom_{\D_N(\E)}(-, \mu_1^s(P^s)) = 0 \text{ for all } s \in \{m,\dots,n\},\]
	\[\Hom_{\D_N(\E)}(P, -) =0 \; \Longleftrightarrow \; \Hom_{\D_N(\E)}(\mu_1^s(P^s), -)= 0 \text{ for all } s \in \{m,\dots,n\}.\]
\end{rmk}

\begin{dfn}
	The \textbf{triangulated hull} $\tri_\T(\cS)$ of subcategory $\cS$ of a  triangulated category $\T$ is the smallest triangulated subcategory of $\T$ containing all objects of $\cS$.
\end{dfn}

\begin{lem} \label{lem: perf-tri} The category of perfect complexes over an exact idempotent complete category $\E$ is given by $\D^{\perf}_N(\E) = \tri\{\mu_1^s(P) \mid P \in \Prj(\E), s \in \{1,\dots,N-1\} \}$, see also \cite[Lem.~2.6.(ii)]{IKM17}.
\end{lem}

\begin{proof} By \Cref{lem: Sigma-mu} and \Cref{rmk: Perf-Ext}, we have $$\mu^0_1(P)=\Sigma \mu^{N-1}_{N-1}(P) \in \tri\{\mu_1^s(P) \mid P \in \Prj(\E), s \in \{1,\dots,N-1\}\}$$ and the claim follows from \Cref{rmk: Perf-Ext} using \Cref{thm: IKM2.4}.
\end{proof}

\begin{dfn}
	We define the \textbf{stabilized $\boldsymbol N$-derived category}, or  \textbf{$\boldsymbol N$-singularity category}, of an exact idempotent complete category $\E$ as the Verdier quotient $\underline \D_N^b(\E):=\D_N^b(\E)/\D^{\perf}_N(\E)$.
\end{dfn}

\begin{dfn} \label{ntn: iota-stab}
	Let $\E$ be an exact idempotent complete category. The functor $\iota^n = \iota^n_\E$, see \Cref{ntn: iota}, sends $\Prj(\mMor_{N-2}(\E))$ to $\D^{\perf}_N(\E)$, see \Cref{thm: mMor}.\ref{thm: mMor-Proj}, and hence factors uniquely through the stable category as $\underline \iota^n = \underline \iota^n_\E$:
	\begin{center}
		\begin{tikzcd}[row sep={20mm,between origins}, column sep={25mm,between origins}]
			\mMor_{N-2}(\E) \ar[d]  \ar[r, "\iota^n_\E"] & \D^{b}_N(\E) \ar[d] \\
			\underline \mMor_{N-2}(\E) \ar[r, dashed, "\underline \iota^n_\E"]& \underline \D^{b}_N(\E)
		\end{tikzcd}
	\end{center}
\end{dfn}

\begin{dfn}
	For $n \in \ZZ$, we define the \textbf{$\boldsymbol n$th extension group} of two $N$-complexes $X, Y \in \C_N(\E)$ over an exact idempotent complete category $\E$ as $\Ext_\E^n(X, Y):=\Hom_{\D_N(\E)}(X, \Sigma^nY)$.
\end{dfn}

Part \ref{lem: Orlov1.10-orig} of \Cref{lem: Orlov1.10} yields in particular a statement of Orlov's in our setting, see {\cite[Lem.~1.10]{Orl09}}. It also serves to generalize Buchweitz's characterization of perfect 2-complexes of modules in part \ref{lem: Orlov1.10-perf}, see {\cite[Lem.~1.2.1]{Buc21}}.

\begin{lem} \label{lem: Orlov1.10}
	Let $\E$ be an exact idempotent complete category with enough projectives.
	\begin{enumerate}
		\item \label{lem: Orlov1.10-orig} If $P \in \C^{-, b_\E}_N(\Prj(\E))$ is acyclic at all positions up to $ n+1 \in \ZZ$, then $P \cong \iota^{n} \Omega^{n+1} P$ in $\underline \D^b_N(\E)$. In particular, any object of $\underline \D^b_N(\E)$ lies in $\underline \iota^n(\underline\mMor_{N-2}(\E))$, for any sufficiently small $n \in \ZZ$.
		
		\item \label{lem: Orlov1.10-perf} An $N$-complex $X \in \D^{b}_N(\E)$ lies in $\D^{\perf}_N(\E)$ if and only if $\Ext_{\E}^n(X, \iota^0(\mMor_{N-2}(\E)))=0$ for any sufficiently large $n \in \ZZ$. In particular, the subcategory $\D^{\perf}_N(\E)$ of $\D^b_N(\E)$ is thick.
	\end{enumerate}
\end{lem}

\begin{proof}\
	\begin{enumerate}[leftmargin=*]
		\item By \Cref{cor: soft-trunc-seq} and \Cref{lem: termwise-triangle}, there is a distinguished triangle
		
		\begin{center}
			\begin{tikzcd}
				\sigma^{\leq n} P \ar[r] & P \ar[r] & \sigma^{\geq n-N+2} P \ar[r] & \Sigma\sigma^{\leq n} P
			\end{tikzcd}
		\end{center}
		in $\D^{-, b}_N(\E)$ with $\sigma^{\leq n} P$ acyclic. It follows that $P \cong \sigma^{\geq n-N+2} P$ in $\D_N(\E)$. Setting $C := \Omega^{n+1} P$ we have $\tau^{\leq n} \sigma^{\geq n-N+2} P = \iota^{n} C$ and $\tau^{\geq n+1} \sigma^{\geq n-N+2} P = \tau^{\geq n+1} P$. Applying \Cref{rmk: tau-triangle} to $X=\sigma^{\geq n-N+2} P \cong P$ yields a distinguished triangle
		\begin{align} \label{eqn: Orlov-triangle}
			\begin{tikzcd}[ampersand replacement=\& ]
				\tau^{\geq n+1} P \ar[r] \& P \ar[r] \&  \iota^{n} C \ar[r] \& \Sigma\tau^{\geq n+1} P
			\end{tikzcd}
		\end{align}
		in $\D^b_N(\E)$ with $\tau^{\geq n+1} P \in \D^{\perf}_N(\E)$. It follows that $P \cong \iota^{n} C$ in $\underline \D^b_N(\E)$, which proves \ref{lem: Orlov1.10-orig}. The particular claim follows by \Cref{thm: K-sod}.\ref{thm: K-sod-b}.
		
		\item Assume first that $X \in \D^{\perf}_N(\E)$. It follows from \Cref{thm: IKM2.4} and \Cref{lem: Verdier-crit} that, for any sufficiently large $n \in \ZZ$, and $s \in \{0,1\}$,
		\[\Hom_{\D_N(\E)}(X, \Sigma^{2n+s}\iota^0(\mMor_{N-2}(\E))) \cong \Hom_{\K_N(\E)}(\Sigma^{-s} X, \Theta^{Nn}\iota^0(\mMor_{N-2}(\E))) = 0,\]
		because these two complexes have disjoint support. To show the converse suppose that $X \in \D^b_N(\E)$ satisfies the vanishing of extension groups. \Cref{cor: resolution-exist,cor: resolution-cohomology} yield a projective resolution $P \to X$ with $P \in \C^{-, b_\E}_N(\Prj(\E))$ and $X \cong P$ in $\D_N(\E)$. Pick $m \in \ZZ$ sufficiently large such that $P$ is acyclic at all positions up to $n + 1$ where $n= -Nm$.
		Due to the distinguished triangle \eqref{eqn: Orlov-triangle} from the proof of \ref{lem: Orlov-prep-1}, it suffices to show that $\iota^{n} C \in \D^{\perf}_N(\E)$. Then $X$ lies in the triangulated subcategory $\D^{\perf}_N(\E)$ of $\D^b_N(\E)$. Since $\iota^n C \cong \Sigma^{2m}\iota^0 C \in \Sigma^{2m} \iota^0 \mMor_{N-2}(\E)$ by \Cref{thm: IKM2.4}, we may assume by hypothesis that $\Hom_{\D_N(\E)}(P, \iota^n C) = 0$. Then \eqref{eqn: Orlov-triangle} yields an induced morphism, see \Cref{rmk: cohom},
		\begin{center}
			\begin{tikzcd}[sep={22.5mm,between origins}]
				P \ar[r, "0"] \ar[rd, "0"] &  \iota^n C \ar[d, equal] \ar[r] & \Sigma\tau^{\geq n+1} P \ar[ld, dashed, "f"] \\
				& \iota^n C
			\end{tikzcd}
		\end{center}
		in $\D_N(\E)$.  \Cref{lem: Verdier-crit} shows that
		$$\Hom_{\D_N(\E)}(\tau^{\geq n+1} \Sigma \tau^{\geq n+1} P, \iota^n C) \cong  \Hom_{\K_N(\E)}(\tau^{\geq n+1} \Sigma \tau^{\geq n+1} P, \iota^{n} C) = 0,$$ because these two complexes have disjoint support. Thus, there is an induced morphism, see \Cref{rmk: tau-triangle} and \Cref{rmk: cohom},
		
		\begin{center}
			\begin{tikzcd}[sep={25mm,between origins}]
				\tau^{\geq n+1} \Sigma \tau^{\geq n+1} P \ar[r] \ar[rd, "0"] &  \Sigma \tau^{\geq n+1} P \ar[d, "f"] \ar[r] & \tau^{\leq n} \Sigma \tau^{\geq n+1} P \ar[ld, dashed] \\
				& \iota^n C.
			\end{tikzcd}
		\end{center}
		in $\D_N(\E)$. By \Cref{lem: tau-Sigma-tau}.\ref{lem: tau-Sigma-tau-b}, $f$ and hence $\id_{\iota^n C} = \iota^{n} \id_C$ thus factors through $\tau^{\leq n} \Sigma \tau^{\geq n+1} P \in \iota^{n}\Prj(\mMor_{N-2}(\E))$, see \Cref{thm: mMor}.\ref{thm: mMor-Proj}. By \Cref{prp: mMor-Db} this implies that $\id_C$ and hence $C$ is zero in the stable category $\underline \mMor_{N-2}(\E)$. It follows that $C \in \Prj(\mMor_{N-2}(\E))$ due to \Cref{rmk: factor}.\ref{rmk: factor-zero} and thus $\iota^{n} C \in \D^{\perf}_N(\E)$ as claimed. The particular claim follows as  $\Ext^n_\E(-,-)$ commutes with direct sums in the first argument. \qedhere
	\end{enumerate}
\end{proof}

\begin{lem} \label{lem: tau-Sigma-tau} For an $N$-complex $X \in \C_N(\A)$ over an additive category $\A$, and $n \in \ZZ$ we have:
	\begin{enumerate}
		\item \label{lem: tau-Sigma-tau-a} $\tau^{\geq n-N+1} \Sigma \tau^{\leq n} X = \bigoplus_{k=1}^{N-1} \mu^{n-k}_{N-k}(X^{n-k+1})$
		\item \label{lem: tau-Sigma-tau-b} $\tau^{\leq n-1} \Sigma \tau^{\geq n} X \cong_{\C_N(\A)} \bigoplus_{k=1}^{N-1} \mu^{n-1}_{N-k}(X^{n+k-1})$
	\end{enumerate}
	
\end{lem}

\begin{proof}  We use the explicit description of the considered $N$-complexes given in \Cref{con: Sigma-explicit}.\ref{con: Sigma-explicit-formula}.
	\begin{enumerate}[leftmargin=*]
		\item The direct summand corresponding to the last row of the matrix of the differential of $\tau^{\geq n-N+1} \Sigma \tau^{\leq n} X$ is zero, which implies the claimed equality. 
		\item The differentials of $\tau^{\leq n-1} \Sigma \tau^{\geq n} X$ read
		\[d^k_{\tau^{\leq n-1} \Sigma \tau^{\geq n} X} = \left( \begin{array}{ccc}
			
			& \large{E_{k-n+N}} &  \\ \hline
			-d_X^{\{k-n+N\}} & \cdots & -d_X
		\end{array} \right)\]
		for $k \in \{n-N+1, \dots, n-2\}$.
			The desired isomorphism $f\colon \tau^{\leq n-1} \Sigma \tau^{\geq n} X \to \bigoplus_{k=1}^{N-1} \mu^{n-k}_{N-k}(X^{n-k+1})$ is given by
			\[f^k = \begin{pmatrix}
				1 & 0 & \cdots & 0\\
				d_X & 1 & \ddots & \vdots \\
				\vdots & \ddots & \ddots & 0 \\
				d_X^{\{k-n+N-1\}} & \cdots & d_X & 1
			\end{pmatrix}\]
			for $k \in \{n-N+1, \dots, n-1\}$ and zero elsewhere. \qedhere
	\end{enumerate}
	\end{proof}

\subsection{Stabilized syzygies} \label{subsection: stabilized-syzygies}

In this subsection, we approach the question whether full faithfulness of $\iota^n\colon \mMor_{N-2}(\E) \to \D^b_N(\E)$, see \Cref{prp: mMor-Db}, persists under stabilization. We bypass this problem by restricting the functor $\underline \iota^n$ to $\underline \Omega^{n+1}(\underline \TAPC_N(\E))$, see \Cref{prp: Orlov1.11}. Following Buchweitz, this suffices to prove \Cref{thm: buchweitz-intro}. Avramov, Briggs, Iyengar and Letz offer an alternative argument, see {\cite[133]{Buc21}}, which depends on the existence of arbitrary products in the category of modules. Our proof, inspired by Orlov, see {\cite[Prop.~1.11]{Orl09}}, does not require this assumption.

\begin{prp} \label{prp: Orlov1.11}
	Let $\E$ be an exact idempotent complete category and $n \in \ZZ$. Then the functor $\underline \iota^n\colon \underline \mMor_{N-2}(\E) \rightarrow \underline \D^b_N(\E)$ becomes fully faithful when restricted to the image $\underline \Omega^{n+1}(\underline \TAPC_N(\E))$.
\end{prp}

\begin{proof}
	We may assume that $n=0$. Consider $X , Y\in \mMor_{N-2}(\E)$, and write $X = \Omega^1 \tilde P$ and $Y = \Omega^1 \tilde Q$ where $\tilde P, \tilde Q \in \TAPC_N(\E)$. By \Cref{lem: quasi}, we have isomorphisms \[\iota^0 Y \cong \tau^{\leq 0} \tilde Q =: Q \in \C^{-, b_\E}_N(\Prj(\E)) \hspace{5mm} \text{ and } \hspace{5mm} \tau^{\leq m-1} Q \cong \iota^{m-1} \Omega^m  Q\] in $\D_N(\E)$ for any sufficiently small $m \in \ZZ$. Then \Cref{rmk: tau-triangle} yields a distinguished triangle
	
	\begin{align} \label{eqn: Orlov-triangle-2}
		\begin{tikzcd}[sep=large, ampersand replacement = \&] 
			\tau^{\geq m}  Q \ar[r, "v"] \& \iota^0 Y \ar[r, "t"] \& \iota^{m-1} \Omega^m  Q \ar[r] \& \Sigma \tau^{\geq m}  Q
		\end{tikzcd}
	\end{align}
	in $\D_N(\E)$ with $C(t) \cong \Sigma \tau^{\geq m} Q \in \D^{\perf}_N(\E)$. We use this to prove that 
	\[\Hom_{\underline \mMor_{N-2}(\E)}(X, Y) \xlongrightarrow[\cong]{\underline \iota^0} \Hom_{\underline \D^b_N(\E)}(\iota^0 X, \iota^0 Y)\]
	is an isomorphism. 
	
	\smallskip
	
	\textbf{Surjectivity:} Consider a morphism $\iota^0 X \to \iota^0 Y$ in $\underline \D^b_N(\E)$ given by a roof
	
	\begin{center}
		\begin{tikzcd}
			s^{-1}g \colon \; \; \iota^0 X \ar[r, "g"] & W \ar[r, <-, "s"] & \iota^0 Y
		\end{tikzcd}
	\end{center}
	of morphisms $g, s$ in $\D^b_N(\E)$ with $C(s) \in \D^{\perf}_N(\E)$. We may assume that $\Sigma^{-1} C(s)$ and  $\iota^{m-1}\Omega^m Q$ have disjoint supports and hence
	\begin{align*}
		&\Hom_{\D_N(\E)}(\Sigma^{-1} C(s), \iota^{m-1}\Omega^m Q) \cong \Hom_{\K_N(\E)}(\Sigma^{-1} C(s), \iota^{m-1}\Omega^m Q) =0,\\
		&\Hom_{\D_N(\E)}(\iota^0 X, \Sigma \tau^{\geq m} Q) \cong  \Hom_{\D_N(\E)}(\tau^{\leq 0} \tilde P, \Sigma \tau^{\geq m} Q) = 0
	\end{align*}
	 by \Cref{lem: Verdier-crit,lem: quasi,lem: Orlov-prep}.\ref{lem: Orlov-prep-2}. Using the distinguished triangles \hyperref[eqn: std-triangle]{$T(s)$} from \Cref{con: cone} and \eqref{eqn: Orlov-triangle-2} together with \Cref{prp: mMor-Db}, we obtain the following induced morphisms in $\D_N(\E)$, where $f \in \Hom_{\mMor_{N-2}(\E)}(X, Y)$, see \Cref{rmk: cohom}:
	\begin{center}
		\begin{tikzcd}[sep={22.5mm,between origins}]
			& \iota^{m-1} \Omega^m Q  &&& \iota^0 X \ar[d, "r \, \circ \, g"] \ar[rd, "0"] \ar[ld, dashed, "\iota^0 f"'] & \\
			 \Sigma^{-1} C(s) \ar[r] \ar[ru, "0"] & \iota^0 Y \ar[r, "s"] \ar[u, "t"']& W \ar[lu, dashed, "r"']& \iota^0 Y \ar[r, "t"] & \iota^{m-1} \Omega^m Q \ar[r] & \Sigma \tau^{\geq m} Q 
		\end{tikzcd}
	\end{center}

	These fit into an equivalence of roofs $s^{-1}g \to \iota^0 f$ in $\D_N(\E)$:
	
	\begin{center}
		\begin{tikzcd}[sep={22.5mm,between origins}]
			& W \ar[d, "r"] &\\
			\iota^0 X \ar[ru, "g"] \ar[r, "r \, \circ \, g"] \ar[rd, "\iota^0 f"'] & \iota^{m-1} \Omega^m Q & \iota^0 Y \ar[lu, "s"'] \ar[l, "t"'] \\
			& \iota^0 Y \ar[u, "t"] \ar[ru, equal]
		\end{tikzcd}
	\end{center}
	
	\textbf{Injectivity:} Suppose that $f \in \Hom_{\mMor_{N-2}(\E)}(X, Y)$ is a morphism with $\iota^0 f = 0$ in $\underline \D^b_N(\E)$. We  prove that $f$ factors through an object of $\Prj(\mMor_{N-2}(\E))$. By assumption, there is a morphism $s\colon \iota^0 Y \to W$ in $\D^b_N(\E)$ with cone $C(s) \in \D^{\perf}_N(\E)$ such that
	\[
		\begin{tikzcd}[sep={22.5mm,between origins}]
			& \iota^0 Y \ar[d, "s"] & \\
			\iota^0 X \ar[ru, "\iota^0 f"] \ar[r, "0"] &W & \iota^0 Y \ar[lu, equal] \ar[l, "s"]
		\end{tikzcd}
	\]
	commutes in $\D_N(\E)$. As before, we may assume that $\Hom_{\D_N(\E)}(\Sigma^{-1} C(s), \iota^{m-1}\Omega^m Q) =0$
	and use the distinguished triangles \hyperref[eqn: std-triangle]{$T(s)$} from \Cref{con: cone} and \eqref{eqn: Orlov-triangle-2} to obtain the following induced morphisms in $\D_N(\E)$:
	\begin{center}
		\begin{tikzcd}[sep={22.5mm,between origins}]
			& \iota^0 X \ar[ld, dashed, "g"'] \ar[d, "\iota^0 f"] \ar[rd, "0"] &&& \Sigma^{-1} C(s) \ar[ld, dashed, "w"'] \ar[d, "u"] \ar[rd, "0"] & \\
			\Sigma^{-1} C(s) \ar[r, "u"] & \iota^0 Y \ar[r, "s"]  & W & \tau^{\geq m} Q \ar[r, "v"] & \iota^0 Y \ar[r, "t"] & \iota^{m-1} \Omega^m Q 
		\end{tikzcd}
	\end{center}
	These fit into a commutative diagram
	\begin{center}
		\begin{tikzcd}[sep={22.5mm,between origins}]
			& \tau^{\geq m} Q \ar[rd, "v"] & \\
			\iota^0 X \ar[r, "g"] \ar[ru, "w \, \circ \, g"] \ar[rd, "\iota^0 f"] & \Sigma^{-1} C(s) \ar[u, "w"]  \ar[r, "u"] \ar[d, "u"] & \iota^0 Y \\
			& \iota^0 Y \ar[ru, equal]
		\end{tikzcd}
	\end{center}
	in $\D_N(\E)$, and hence $\iota^0 f$ factors through $\tau^{\geq m} Q \in \D^{\perf}_N(\E)$. The claimed injectivity follows from \Cref{prp: mMor-Db} and \Cref{lem: part-Orlov}.
\end{proof}

\begin{lem} \label{lem: Orlov-prep}
	Consider $N$-complexes $X \in \C^{\varnothing^\ast}_N(\E)$, $\tilde P \in \TAPC_N(\E)$, $Q \in \C^{b}_N(\Prj(\E))$ over an exact idempotent complete category $\E$ and $n \in \ZZ$. We have
	\begin{enumerate}
		\item \label{lem: Orlov-prep-1} $\Hom_{\K_N(\E)}(X, Q) = 0$ and, in particular, $\Hom_{\D_N(\E)}(\tau^{\leq n} \tilde P, \tau^{\leq n-N+1 } Q) = 0$,
		\item \label{lem: Orlov-prep-2} $\Hom_{\K_N(\E)}(\tau^{\leq n} X, \Sigma \tau^{\leq n} Q) = 0$ and, in particular, $\Hom_{\D_N(\E)}(\tau^{\leq n} \tilde P, \Sigma \tau^{\leq n} Q) = 0$.
	\end{enumerate}
\end{lem}

\begin{proof} \
	\begin{enumerate}[leftmargin=*]
		\item We may assume that $Q \neq 0$ and set $m := \max\{ k \in \ZZ \, \vert \, Q^k \neq 0 \}$. Since $X$ is totally acyclic, $\Hom_{\K_N(\E)}(X, \mu^m_1(Q^m)) = H^{-m}_{(1)}(\Hom_\E(X, Q^m)) = 0$ due to \Cref{rmk: Hom-mu}.\ref{rmk: Hom-mu-d}. By induction on the length of $Q$ we may suppose that $\Hom_{\K_N(\E)}(X, \tau^{\leq m-1} Q)=0$. Then $\Hom_{\K_N(\E)}(X, Q) =0$ follows by applying $\Hom_{\K_N(\E)}(X, -)$ to the distinguished triangle
		\begin{center}
			\begin{tikzcd}
				\mu^m_1(Q^m) \ar[r] & Q \ar[r] & \tau^{\leq m-1} Q \ar[r] & \Sigma \mu^m_1(Q^m)
			\end{tikzcd}
		\end{center}
		in $\K_N(\E)$, see \Cref{rmk: tau-triangle}. The second claim follows using \Cref{lem: Verdier-crit}:
		
		$$
		\Hom_{\D_N(\E)}(\tau^{\leq n} \tilde P, \tau^{\leq n-N+1} Q)  \cong \Hom_{\K_N(\E)}(\tau^{\leq n} \tilde P, \tau^{\leq n-N+1} Q) = \Hom_{\K_N(\E)}(\tilde P, \tau^{\leq n-N+1} Q) = 0.
		$$
		
		\item We may assume that $n=0$ and $Q = \tau^{\leq 0} Q$. We have $\tau^{\geq -N+1} \Sigma Q = \bigoplus_{k=1}^{N-1} \mu^{-k}_{N-k}(Q^{-k+1})$ by \Cref{lem: tau-Sigma-tau}.\ref{lem: tau-Sigma-tau-a}, see \Cref{con: Sigma-explicit}.\ref{con: Sigma-explicit-formula}. Using \Cref{rmk: Hom-mu}.\ref{rmk: Hom-mu-d}, this implies that
		\begin{align*}
			\Hom_{\K_N(\E)}(\tau^{\leq 0} X, \tau^{\geq -N+1} \Sigma Q) & \cong \bigoplus_{k=1}^{N-1} \Hom_{\K_N(\E)}(\tau^{\leq 0} X, \mu^{-k}_{N-k}(Q^{-k+1})) \\
			& \cong \bigoplus_{k=1}^{N-1} H^{k}_{(N-k)}(\Hom_\E(\tau^{\leq 0}X, Q^{-k+1})) \\
			& = \bigoplus_{k=1}^{N-1} H^{k}_{(N-k)}(\tau^{\geq 0}\Hom_\E(X, Q^{-k+1}))\\
			& = \bigoplus_{k=1}^{N-1} H^{k}_{(N-k)}(\Hom_\E(X, Q^{-k+1}))\\
			&= 0,
		\end{align*}
		since $X$ is totally acyclic. 
		Due to part \ref{lem: Orlov-prep-1} for $\tau^{\leq -N} \Sigma Q \in \C^b_N(\Prj(\E))$, we have
		\[\Hom_{\K_N(\E)}(\tau^{\leq 0} X, \tau^{\leq -N} \Sigma Q) = \Hom_{\K_N(\E)}(X, \tau^{\leq -N} \Sigma Q) = 0.\]
		Then $\Hom_{\K_N(\E)}(\tau^{\leq 0} X, \Sigma Q) =0$ by applying $\Hom_{\K_N(\E)}(\tau^{\leq 0} X, -)$ to the distinguished triangle
		\begin{center}
			\begin{tikzcd}
				\tau^{\geq -N+1} \Sigma Q \ar[r] & \Sigma Q \ar[r] & \tau^{\leq -N} \Sigma Q \ar[r] & \Sigma \tau^{\geq -N+1} \Sigma Q
			\end{tikzcd}
		\end{center}
		in $\K_N(\E)$, see \Cref{rmk: tau-triangle}. The second claim is again due to \Cref{lem: Verdier-crit}. \qedhere
	\end{enumerate}
\end{proof}

\begin{lem} \label{lem: part-Orlov}
	Let $\E$ be an exact idempotent complete category, $X = \Omega^1 \tilde P \in \mMor_{N-2}(\E)$ where $\tilde P \in \TAPC_N(\E)$ and $Q= \tau^{\leq 0} Q \in \D^{\perf}_N(\E)$. Then any morphism $g\colon \iota^0 X \to Q$ in $\D^b_N(\E)$ factors through an object of $\iota^0\Prj(\mMor_{N-2}(\E))$
\end{lem}

\begin{proof} Since $\Hom_{\D_N(\E)}(\iota^0 X, \tau^{\leq -N+1} Q) \cong \Hom_{\D_N(\E)}(\tau^{\leq 0} \tilde P, \tau^{\leq -N+1} Q) = 0$ by \Cref{lem: quasi,lem: Orlov-prep}.\ref{lem: Orlov-prep-1}, the morphism $g$ factors in $\D_N(\E)$ as
	\[
	\begin{tikzcd}[sep={22.5mm,between origins}]
		& \iota^0 X \ar[d, "g"] \ar[rd, "0"] \ar[ld, dashed] & \\
		\tau^{\geq -N+2} Q \ar[r] & Q \ar[r] & \tau^{\leq -N+1} Q,
	\end{tikzcd}
	\]
	 see \Cref{rmk: tau-triangle,rmk: cohom}.	We may thus assume that $ Q = \tau^{\geq -N+2} Q$. Then applying the exact functor $\tau^{\leq 0}$ to the termsplit short exact sequence \begin{tikzcd}[cramped]
			 \Sigma^{-1} Q \ar[r, tail] & P(Q) \ar[r, two heads] & Q,
		\end{tikzcd}
	see \Cref{con: cone,con: I-and-P}, yields a distinguished triangle
	\[
		\begin{tikzcd}
			\tau^{\leq 0} \Sigma^{-1} Q \ar[r] & \tau^{\leq 0} P(Q) \ar[r] &  Q \ar[r] & \Sigma \tau^{\leq 0} \Sigma^{-1} Q
		\end{tikzcd}
	\]
	in $\D_N(\E)$, see \Cref{lem: ses-triangle}. Since $\Hom_{\D_N(\E)}(\iota^0 X, \Sigma \tau^{\leq 0} \Sigma^{-1} Q) = \Hom_{\D_N(\E)}(\tau^{\leq 0} \tilde P, \Sigma \tau^{\leq 0} \Sigma^{-1} Q) = 0$ by \Cref{lem: quasi,lem: Orlov-prep}.\ref{lem: Orlov-prep-2}, the morphism $g$ factors in $\D_N(\E)$ as
	\[
		\begin{tikzcd}[sep={22.5mm,between origins}]
			& \iota^0 X \ar[d, "g"] \ar[rd, "0"] \ar[ld, dashed] & \\
			\tau^{\leq 0} P(Q) \ar[r] & Q \ar[r] & \Sigma \tau^{\leq 0} \Sigma^{-1} Q,
		\end{tikzcd}
	\]
	 see \Cref{rmk: cohom}, where $\tau^{\leq 0} P(Q) = \bigoplus_{k=1}^{N-1} \mu^{0}_{k}(Q^{-k+1}) \in \iota^0 \Prj(\mMor_{N-2}(\E))$, see \Cref{thm: mMor}.\ref{thm: mMor-Proj}.
\end{proof}

\subsection{Stabilized truncations} \label{subsection: stabilized-truncations}

In this subsection, we prove that the stabilized truncations $\underline \tau^{\leq n}$ for $n \in \ZZ$ are pairwise isomorphic fully faithful triangle functors.

\begin{prp} \label{prp: tau-stab}
	Let $\E$ be an exact idempotent complete category and $n \in \ZZ$. Then the truncation $\tau^{\leq n}\colon \C_N(\E) \to \C_N(\E)$ induces a triangulated functor $\underline \tau^{\leq n}\colon \underline \APC_N(\E) \to \underline \D^{b}_N(\E)$ such that $\underline \tau^{\leq n} \cong \underline \iota^n \circ \underline \Omega^{n+1}$.
\end{prp}

\begin{proof}
	Due to \Cref{lem: quasi,lem: Omega-stab}, \Cref{ntn: iota-stab} and \Cref{thm: D-diamond}, $\tau^{\leq n}$ induces a well-defined functor $\underline \APC_N(\E) \to \underline \D^{-, b}_N(\E) \simeq \underline \D^b_N(\E)$ with $\underline \tau^{\leq n} \cong \underline \iota^n \circ \underline \Omega^{n+1}$.
	To prove that $\underline \tau^{\leq n}$ is triangulated we may assume that $n=0$. Let $\Sigma_A$ and $\Sigma_D$ be the suspension functors in $\underline \APC_N(\E)$ and $\underline \D^{b}_N(\E)$, respectively. We need to establish a natural isomorphism $\eta\colon \underline \tau^{\leq 0} \Sigma_A \to \Sigma_D \underline \tau^{\leq 0}$ under which the functor $\underline \tau^{\leq 0}$ maps distinguished triangles to distinguished triangles. To simplify notation, we abbreviate $X^{\leq 0} := \tau^{\leq 0} X$ for an $N$-complex $X$ and likewise for morphisms. For any $P \in \APC_N(\E)$, there is an isomorphism
	\[I(P)^{\leq 0} =\left(\bigoplus_{k \in \ZZ} \mu_N^k(P^k)\right)^{\leq 0} \cong \bigoplus_{k=-\infty}^{0} \mu_N^k(P^k) \oplus  J(P) = I(P^{\leq 0})  \oplus J(P),  \text{ where } J(P) := \bigoplus_{k=1}^{N-1}\mu_{N-k}^0(P^k).  \]
	Let $f\colon P \to Q$ be a morphism in $\APC_N(\E)$. We apply the exact functor $\tau^{\leq 0}$ to  \eqref{eqn: std-con}, see \Cref{con: cone,con: I-and-P}, and combine it with \hyperref[eqn: std-con]{$(D(f^{\leq 0}))$}. This yields a diagram of termsplit short exact sequences, in which the dashed and dotted morphisms remain to be constructed:
	\begin{equation} \label{eqn: tau-triangulated}
		\begin{tikzcd}[column sep={17.5mm,between origins}, row sep={15mm,between origins}]
			P^{\leq 0} \ar[rr, tail] \ar[d, "f^{\leq 0}"'] && I(P^{\leq 0}) \ar[rr, two heads] \ar[d] && \Sigma (P^{\leq 0}) \ar[d, equal]  && \\
			Q^{\leq 0} \ar[rru, phantom, "\square"] \ar[rr, tail] && C(f^{\leq 0}) \ar[rr, two heads] \ar[rdd, <<-, dashed] && \Sigma (P^{\leq 0}) \ar[rdd, <<-, dotted, "\eta_{P}" near start]\\
			&P^{\leq 0} \ar[rr, tail, crossing over] \ar[d, "f^{\leq 0}"] \ar[luu, equal, crossing over] && I(P)^{\leq 0} \ar[rr, two heads, crossing over] \ar[d] \ar[luu, crossing over, two heads, "p"'] && (\Sigma P)^{\leq 0} \ar[d, equal] \ar[luu, two heads, dotted, "\eta_{P}"'] \\
			&Q^{\leq 0} \ar[luu, equal] \ar[rr, tail] \ar[rru, phantom, "\square"] && C(f)^{\leq 0} \ar[rr, two heads]  && (\Sigma P)^{\leq 0} \ar[rdd, <-<, dotted] \\
			&&&& J(P) \ar[luu, crossing over, tail] \ar[d, equal] \ar[rr, equal, crossing over ] && J(P) \ar[d, equal] \ar[luu, tail, dotted] \\
			&&&& J(P) \ar[luu, tail, dashed] \ar[rr, equal] && J(P)
		\end{tikzcd}
	\end{equation}	

	 The lower pushout is due to exactness of $\tau^{\leq 0}$, see \Cref{prp: exact-pushpull}. The dashed morphisms occur in the following commutative diagram of short exact sequences obtained from \Cref{prp: Buehler2.12}.\ref{prp: Buehler2.12-push} and \Cref{lem: Buehler3.7}:
	 
	\begin{center}
		\begin{tikzcd}[sep={25mm,between origins}, ampersand replacement=\&]
			0 \ar[d, tail] \ar[r, tail] \& J(P) \ar[r, equal] \ar[d, tail] \& J(P) \ar[d, tail, dashed] \\
			P^{\leq 0} \ar[r, tail] \ar[d, equal] \& Q^{\leq 0} \oplus I(P)^{\leq 0} \ar[r, two heads] \ar{d}{\begin{pmatrix} \id & 0 \\ 0 & p \end{pmatrix}} \& C(f)^{\leq 0} \ar[d, two heads, dashed] \\
			P^{\leq 0} \ar[r, tail] \& Q^{\leq 0} \oplus I(P^{\leq 0}) \ar[r, two heads] \& C(f^{\leq 0})
		\end{tikzcd}
	\end{center}

	Then \eqref{eqn: tau-triangulated} commutes, save the dotted morphisms. Both dotted termsplit short exact sequences arise from \Cref{lem: Buehler3.7} as well. They agree by precomposing with the epic $I(P)^{\leq 0} \twoheadrightarrow (\Sigma P)^{\leq 0}$.\\
	In the singularity category $\underline \D^{b}_N(\E)$, the perfect $N$-complex $J(P)$ is zero  and \eqref{eqn: tau-triangulated} yields an isomorphism of candidate triangles:
	
	\begin{center}
		\begin{tikzcd}[sep={17.5mm,between origins}]
			P^{\leq 0} \ar[r] \ar[d, equal] & Q^{\leq 0} \ar[r] \ar[d, equal] & C(f)^{\leq 0} \ar[r] \ar[d, "\cong"] & (\Sigma P)^{\leq 0} \ar[d, "\eta_{P}"', "\cong"] \\
			P^{\leq 0} \ar[r] & Q^{\leq 0} \ar[r] & C(f^{\leq 0}) \ar[r] & \Sigma(P^{\leq 0})
		\end{tikzcd}
	\end{center}
	
	The naturality of the isomorphisms $\eta = (\eta_{P})_{P}$ results from the commutative diagram
	\begin{center}
		\begin{tikzcd}[sep={15mm,between origins}]
			& P^{\leq 0} \ar[rr, tail] \ar[dd, "f^{\leq 0}" near start] && I(P^{\leq 0}) \ar[rr, two heads] \ar[dd] && \Sigma(P^{\leq 0}) \ar[dd, "\Sigma(f^{\leq 0})"] \\
			P^{\leq 0} \ar[rr, tail, crossing over] \ar[dd, "f^{\leq 0}"] \ar[ru, equal] && I(P)^{\leq 0} \ar[rr, two heads, crossing over] \ar[ru, two heads]  && (\Sigma P)^{\leq 0} \ar[ru, two heads, "\eta_{P}"'] \\
			& Q^{\leq 0} \ar[rr, tail] && I(Q^{\leq 0}) \ar[rr, two heads] && \Sigma(Q^{\leq 0}) \\
			Q^{\leq 0} \ar[rr, tail] \ar[ru, equal] && I(Q)^{\leq 0} \ar[rr, two heads] \ar[ru, two heads] \ar[uu, <-, crossing over] && (\Sigma Q)^{\leq 0} \ar[ru, two heads, "\eta_{Q}"'] \ar[uu, <-, crossing over, "(\Sigma f)^{\leq 0}"' near end]
		\end{tikzcd}
	\end{center}
	in $\C_N(\E)$, see \Cref{rmk: I-P-functor}, where the rightmost square commutes by precomposing with the epic $I(P)^{\leq 0} \twoheadrightarrow \Sigma(P)^{\leq 0}$.
\end{proof}

\begin{lem} \label{lem: tau-lim} Let $\E$ be an exact idempotent complete category. The canonical maps induce a directed system of isomorphic functors
	\begin{center}
		\begin{tikzcd}[sep=large]
			\cdots \ar[r, "\cong"] & \underline \tau^{\leq n+1} \ar[r, "\cong"] & \underline \tau^{\leq n} \ar[r, "\cong"] & \cdots
		\end{tikzcd}
	\end{center}
	\noindent
	in the category $\textup{Func}(\underline\APC_N(\E), \underline \D^{b}_N(\E))$. In particular, the inverse limit $\underline \tau^{\leq} := \varprojlim_{k} \underline \tau^{\leq k }$ exists and is isomorphic to any of the functors $\underline \tau^{\leq n}$, where $n \in \ZZ$.
\end{lem}

\begin{proof}
	For $P \in \APC_N(\E)$, there is a distinguished triangle
	\begin{center}
		\begin{tikzcd}
			\mu^0_{N-1}(P^1) \ar[r] & C(\tau^{\leq 1} P \to \tau^{\leq 0} P) \ar[r] &  C(\id_{\tau^{\leq 0} P}) \ar[r] & \Sigma \mu^0_{N-1}(P^1)
		\end{tikzcd}
	\end{center}
	in $\underline \D^{-, b}(\E) \simeq \underline \D^{b}(\E)$, see \Cref{con: Sigma-explicit}.\ref{con: Sigma-explicit-formula}, \Cref{lem: ses-triangle} and \Cref{thm: D-diamond}. As $C(\id_{\tau^{\leq 0} P})$ and $\mu^0_{N-1}(P^1)$ are both zero, so is $C(\tau^{\leq 1} P \to \tau^{\leq 0} P)$. Thus, $\underline \tau^{\leq 1} P \cong \underline \tau^{\leq 0} P$ which implies the claim.
\end{proof}

\subsection{Buchweitz's Theorem} \label{subsection: buchweitz-theorem}

In this subsection we finally prove \Cref{thm: buchweitz-intro}. The last missing ingredient is the essential surjectivity of the stabilized truncation.

\begin{dfn} \label{dfn: locally-finite}
	We say that an exact category $\E$ has \textbf{locally finite $\boldsymbol \F$-dimension} for a subcategory $\F$ if every objects $X \in \E$ has a projective resolution $P$ over $\E$ such that $\syz_P^g(X) \in \F$ for some $g \in \ZZ$, see \Cref{dfn: syz}. In particular, $\E$ has enough projectives.
\end{dfn}

\begin{asn} \label{asn: setting}\
	\begin{enumerate}
		\item $\E$ is an exact idempotent complete category;
		
		\item $\F$ is a Frobenius category;
		
		\item $\F$ is a fully exact, replete subcategory of $\E$;
		
		\item \label{asn: setting-syz} $\E$ has locally finite $ \F$-dimension and $\Prj(\F)=\Prj(\E)$.
	\end{enumerate}
\end{asn}

\begin{rmk} \label{rmk: asn-syz}
	Assume \ref{asn: setting} and consider $X$, $P$ and $g$ as in \Cref{dfn: locally-finite}. Resolving $\syz^g_P(X) \in \F$ by $\E$-projectives in $\F$ yields a projective resolution $Q$ of $X$ over $\E$ with $\syz^n_Q(X) \in \F$ for $n \geq g$. By \Cref{prp: syz-welldef}, then $\syz_P^n(X) \oplus \tilde Q^n \cong \syz_Q^n(X) \oplus \tilde P^n$ lies in $\F$ for some $\tilde P^n, \tilde Q^n \in \Prj(\E) \subseteq \F$.
\end{rmk}

\begin{lem} \label{lem: Omega-in-F}
	Assuming \ref{asn: setting}, for any $P \in \C^{\infty, +}_N(\Prj(\E))$, there is an acyclic $N$-complex $Q \in \C^{-, \varnothing_\E}_N(\Prj(\E)) \cap \Prj(C_N(\E))$ such that $\Omega^n(P \oplus Q) \in \mMor_{N-2}(\F)$ for any sufficiently small $n \in \ZZ$.  
\end{lem}

\begin{proof}
	By shifting $P$, we may assume that $P$ is acyclic at all non-positive positions. Suppose that $n \in \ZZ$ is sufficiently small. Then the cokernels of the admissible monics in $\Omega^n (P \oplus Q)$ occur in $\Omega^{n+1}(P \oplus Q), \dots, \Omega^{n+N-1}(P \oplus Q)$. If all their objects lie in $\F$, then $\Omega^n (P \oplus Q) \in \mMor_{N-2}(\F)$, since $\F$ is fully exact in $\E$.\\
	Write $n=k-Nl$, where $k \in \{-N+1,\dots,0\}$ and $l \in \NN$. For any $r \in \{1,\dots,N-1\}$, the projective resolution $\Theta^k \tau^{\leq k} \gamma^k_{r}(P)$ realizes $\Omega^{n}(P)^{r} = C^{k-Nl-N+r}_{(r)}(P)$ as a $2l$th syzygy of $C^k_{(N-r)}(P)$:
	\begin{center}
		\begin{tikzcd}[column sep={9mm,between origins}, row sep={10mm,between origins}]
		\cdots \ar[rr] && P^{n-N+r} \ar[rr, "d^{\{N-r\}}_P"] \ar[rd, two heads] && P^{n} \ar[rr, "d^{\{r\}}_P"] &&  \cdots \ar[rr, "d^{\{N-r\}}_P"] && P^{k-N} \ar[rr, "d^{\{r\}}_P"] && P^{k-N+r} \ar[rr, "d^{\{N-r\}}_P"] && P^k \ar[rr, "d^{\{r\}}_P"] \ar[rd, two heads] &&  P^{k+r} \ar[rr] && \cdots \\
		&&& \Omega^{n}(P)^{r} \ar[ru, tail] &&&&&&&&&& C^{k}_{(N-r)}(P) \ar[ru, tail]
		\end{tikzcd}
	\end{center}
	Due to \Cref{rmk: asn-syz}, there are $Q^{n}_{r} \in \Prj(\E)$, for $r \in \{1, \dots, N-1\}$, such that $\Omega^{n}(P)^{r} \oplus Q^{n}_{r} \in \F$. As $\Prj(\E) \subseteq \F$, the object $Q^{n} := \bigoplus_{r=1}^{N-1} Q^{n}_r$ works for all $r$. Due to \Cref{rmk: Omega-mu}, we have $\Omega^n(\mu^{n}_N(Q^n))=\mu_{N-1}(Q^n)$ and $\Omega^n(P \oplus \mu^{n}_N(Q^n))^r = \Omega^n(P)^r \oplus Q^n \in \F$ for any r. The claim follows with $Q := \bigoplus_{n \ll 0} \mu^{n}_N(Q^n)$, see \Cref{rmk: C-direct-sum,rmk: mu-acyclic}.\ref{rmk: mu-acyclic-statement} and \Cref{lem: mu-ProjInj}.
\end{proof}	

\begin{prp} \label{prp: tau-surj}
	Assuming \ref{asn: setting}, the restricted truncation $\underline \tau^\leq: \underline \APC_N(\F) \to \underline \D^{b}_N(\E)$ is essentially surjective, see \Cref{cor: APC}.
\end{prp}

\begin{proof}
	Due to \Cref{cor: resolution-exist,cor: resolution-cohomology}, any $X \in \C^b_N(\E)$ admits a projective resolution $P \to X $ with $ P\in \C^{-, b_\E}_N(\Prj(\E))$. By \Cref{lem: Omega-in-F}, we may assume that $\Omega^{n+1} P \in \mMor_{N-2}(\F)$ for any sufficiently small $n$. \Cref{prp: Omega-surj-full} yields a $Q \in \APC_N(\F)$ with $\Omega^{n+1} Q \cong \Omega^{n+1} P$ in $\mMor_{N-2}(\F)$. Using Lemmas \ref{lem: Orlov1.10}.\ref{lem: Orlov1.10-orig}, \ref{lem: quasi} and \ref{lem: tau-lim}, we conclude that $X \cong P \cong \underline \iota^{n} \underline \Omega^{n+1} Q \cong \underline \tau^{\leq n} Q \cong \underline \tau^\leq Q$ in $\underline \D^b_N(\E)$ for sufficiently small $n \in \ZZ$.
\end{proof}

\begin{thm}[Buchweitz's Theorem] \label{thm: buchweitz}
	Assuming \ref{asn: setting}, there is, for any $n \in \ZZ$, the following commutative diagram, where $\simeq$ indicates triangle equivalences:
	
	\begin{center}
		\begin{tikzcd}[sep={20mm,between origins}]
			 \underline{\textup{APC}}_N(\F) \ar[dd, hook] \ar[rd, "\simeq"] \ar[rr, "\underline \Omega^{n+1}_\F", "\simeq"'] && \underline \mMor_{N-2}(\F)   \ar[dd, hook] \ar[ld, "\simeq"'] \\
			&\underline \D^b_N(\E) \\
			 \underline \TAPC_N(\E) \ar[ru, "\underline \tau^{\leq n}"] \ar[rr, "\underline \Omega^{n+1}_\E"] &&  \underline \mMor_{N-2}(\E)  \ar[lu, "\underline \iota^n_\E"']
		\end{tikzcd}
	\end{center}
\end{thm}

\begin{proof} By \Cref{lem: stab-mMor-embed} and \Cref{cor: APC}, there is a fully faithful functor $\underline \mMor_{N-2}(\F) \to \underline \mMor_{N-2}(\E)$ and a fully faithful triangle functor $\underline \APC_N(\F) = \underline \TAPC_N(\F) \to \underline \TAPC_N(\E)$, see \Cref{rmk: Frobenius-tac}. The outer square commutes by construction. By \Cref{prp: tau-stab}, $\underline \tau^{\leq n} \cong \underline \iota^n_\E \circ \underline \Omega^{n+1}_\E$ is a triangle functor. The restricted truncation $\underline \tau^{\leq n}\colon \underline \APC_N(\F) \to \underline \D^{b}_N(\E)$ is then a triangle functor isomorphic to $\underline \iota^n_\E \circ \underline \Omega^{n+1}_\F$. Since $\underline \Omega^{n+1}_\F$ is a triangle equivalence by \Cref{thm: Omega-triangle-equiv}, it is fully faithful by \Cref{prp: Orlov1.11} and hence a triangle equivalence by \Cref{lem: tau-lim} and \Cref{prp: tau-surj}. The restriction $\underline \mMor_{N-2}(\F) \rightarrow \underline \D^b_N(\E)$ of $\underline \iota^n_\E$ is then a triangle equivalence as well. This concludes the proof.
\end{proof}

%%%%%%%%%%%%%%%%%%%%%%%%%%%%%%%%%%%%%%%%%%%%%%%%%%%%%%%%%%%%%%%%%%%%%%%%%%%%%%%%

\printbibliography

%%%%%%%%%%%%%%%%%%%%%%%%%%%%%%%%%%%%%%%%%%%%%%%%%%%%%%%%%%%%%%%%%%%%%%%%%%%%%%%%
\end{document}